\documentclass{amsart}

\usepackage{amssymb}
\usepackage{amsthm}
\usepackage{amsmath}

\usepackage{mathtools}

\usepackage{hyperref}

\theoremstyle{plain}
\newtheorem{proposition}{Proposition}[section]
\newtheorem{theorem}[proposition]{Theorem}
\newtheorem{lemma}[proposition]{Lemma}
\newtheorem{corollary}[proposition]{Corollary}
\theoremstyle{definition}
\newtheorem{example}[proposition]{Example}
\newtheorem{definition}[proposition]{Definition}
\newtheorem{observation}[proposition]{Observation}
\theoremstyle{remark}
\newtheorem{remark}[proposition]{Remark}

\newtheorem{question}[proposition]{Question}

\DeclareMathOperator{\Aut}{Aut}

\DeclareMathOperator{\Vol}{Vol}

\DeclareMathOperator{\SL}{SL}
\DeclareMathOperator{\GL}{GL}
\DeclareMathOperator{\SO}{SO}
\DeclareMathOperator{\PO}{PO}

\DeclareMathOperator{\PSL}{PSL}
\DeclareMathOperator{\PGL}{PGL}
\DeclareMathOperator{\Hom}{Hom}

\DeclareMathOperator{\End}{End}

\DeclareMathOperator{\Imag}{Im} 
 
\DeclareMathOperator{\Spanset}{Span} 
 
\DeclareMathOperator{\id}{id} 
\DeclareMathOperator{\Haus}{Haus} 
\DeclareMathOperator{\CAT}{CAT} 
\DeclareMathOperator{\Isom}{Isom} 
\DeclareMathOperator{\Sym}{Sym} 
\DeclareMathOperator{\OO}{O}
\DeclareMathOperator{\Ric}{Ric}
\DeclareMathOperator{\PSp}{PSp}

\DeclareMathOperator{\PSO}{PSO}

\DeclareMathOperator{\Cc}{\mathcal{C}}
\DeclareMathOperator{\Fc}{\mathcal{F}}

\DeclareMathOperator{\Hc}{\mathcal{H}}

\DeclareMathOperator{\Lc}{\mathcal{L}}

\DeclareMathOperator{\Pc}{\mathcal{P}}

\DeclareMathOperator{\Cb}{\mathbb{C}}

\DeclareMathOperator{\Gb}{\mathbb{G}}
\DeclareMathOperator{\Hb}{\mathbb{H}}

\DeclareMathOperator{\Nb}{\mathbb{N}}
\DeclareMathOperator{\Pb}{\mathbb{P}}
\DeclareMathOperator{\Rb}{\mathbb{R}}
\DeclareMathOperator{\Xb}{\mathbb{X}}
\DeclareMathOperator{\Zb}{\mathbb{Z}}

\newcommand{\abs}[1]{\left|#1\right|}

\newcommand{\norm}[1]{\left\|#1\right\|}
\newcommand{\wt}[1]{\widetilde{#1}}
\newcommand{\wh}[1]{\widehat{#1}}
\newcommand{\ip}[1]{\left\langle #1\right\rangle}
\newcommand{\ceil}[1]{\left\lceil #1\right\rceil}

\begin{document}

\title[Anosov representations and convex cocompact actions]{Projective Anosov representations, convex cocompact actions, and rigidity}
\author{Andrew Zimmer}\address{Department of Mathematics, University of Chicago, Chicago, IL 60637.}
\curraddr{Department of Mathematics, University of Wisconsin-Madison, Madison, WI, 53706}
\email{amzimmer2@wisc.edu}
\date{\today}
\subjclass[2010]{}

\begin{abstract} 
In this paper we show that many projective Anosov representations act convex cocompactly on some properly convex domain in real projective space. In particular, if a non-elementary word hyperbolic group is not commensurable to a non-trivial free product or the fundamental group of a closed hyperbolic surface, then any projective Anosov representation of that group acts convex cocompactly on some properly convex domain in real projective space. We also show that if a projective Anosov representation preserves a properly convex domain, then it acts convex cocompactly on some (possibly different) properly convex domain. 

We then give three applications. First, we show that Anosov representations into general semisimple Lie groups can be defined in terms of the existence of a convex cocompact action on a properly convex domain in some real projective space (which depends on the semisimple Lie group and parabolic subgroup). Next, we prove a rigidity result involving the Hilbert entropy of a projective Anosov representation. Finally, we prove a rigidity result which shows that the image of the boundary map associated to a projective Anosov representation is rarely a $C^2$ submanifold of projective space. This final  rigidity result also applies to Hitchin representations. 
\end{abstract}

\maketitle

\section{Introduction}

If $G$ is a connected simple Lie group with trivial center and $K \leq G$ is a maximal compact subgroup, then $X=G/K$ has a unique (up to scaling) Riemannian symmetric metric $g$ such that $G = \Isom_0(X,g)$. The metric $g$ is non-positively curved and $X$ is simply connected, hence every two points in $X$ are joined by a unique geodesic segment. A subset $\Cc \subset X$ is called \emph{convex} if for every $x,y \in \Cc$ the geodesic joining them is also in $\Cc$.  Finally, a discrete group $\Gamma \leq G$ is said to be \emph{convex cocompact} if there exists a non-empty closed convex set $\Cc \subset X$ such that $\gamma(\Cc) = \Cc$ for all $\gamma \in \Gamma$ and the quotient $\Gamma \backslash \Cc$ is compact.

In the case in which $G$ has real rank one, there are an abundance of examples of convex cocompact subgroups and one has the following characterization:

\begin{theorem}\label{thm:rank1} Suppose $G$ is a real rank one simple Lie group with trivial center, $(X,g)$ is the symmetric space associated to $G$, and $\Gamma \leq G$ is a discrete subgroup. Then the following are equivalent: 
\begin{enumerate}
\item $\Gamma \leq G$ is a convex cocompact subgroup,
\item $\Gamma$ is finitely generated and for some (hence any) $x \in X$ the map 
\begin{align*}
\gamma \in \Gamma \rightarrow \gamma \cdot x
\end{align*}
 induces a quasi-isometric embedding of $\Gamma$ into $X$,
\item $\Gamma$ is word hyperbolic and there exists an injective, continuous, $\Gamma$-equivariant map $\xi: \partial \Gamma \rightarrow X(\infty)$.
\end{enumerate}
\end{theorem}

\begin{remark} For a proof of this theorem see Theorem 5.15 in~\cite{GW2012} which relies on results in~\cite{B1995}. \end{remark}

When $G$ has higher rank, the situation is much more rigid:

\begin{theorem}[Kleiner-Leeb~\cite{KL2006}, Quint~\cite{Q2005}] Suppose $G$ is  a simple Lie group with real rank at least two and $\Gamma \leq G$ is a Zariski dense discrete subgroup. If $\Gamma$ is convex cocompact, then $\Gamma$ is a cocompact lattice in $G$. 
\end{theorem}

Although the most natural definition of convex cocompact subgroups leads to no interesting examples in higher rank, it turns out that the third characterization in Theorem~\ref{thm:rank1} can be used to define a rich class of representations called \emph{Anosov representations}. This class of representations was originally introduced by Labourie~\cite{L2006} and then extended by Guichard-Wienhard~\cite{GW2012}.  Since then several other characterizations have been given by Kapovich-Leeb-Porti~\cite{KLP2013, KLP2014,KLP2014b, KLP_TG_2016, KLP_EJM_2017}, Kapovich-Leeb~\cite{KL2018}, Gu{\'e}ritaud-Guichard-Kassel-Wienhard~\cite{GGKW2015}, and Bochi-Potrie-Sambarino~\cite{BPS2016}.

We refer the reader to~\cite{GW2012} for a precise definition of Anosov representations, but informally: if $\Gamma$ is word hyperbolic, $G$ is a semisimple Lie group, and $P$ is a parabolic subgroup, then  a representation $\rho: \Gamma \rightarrow G$ is called \emph{$P$-Anosov} if there exists an injective, continuous, $\rho$-equivariant map $\xi: \partial \Gamma \rightarrow G/P$ satisfying certain dynamical properties. In the case in which $G$ has real rank one, every two parabolic subgroups are conjugate and the quotient $G/P$ can naturally be identified with $X(\infty)$. 

Recently, Danciger, Gu{\'e}ritaud, and Kassel  established a close connection between Anosov representations into $\PO(p,q)$ and convex cocompact actions. However, the convex cocompact action is not on the associated symmetric space $X=\PO(p,q)/{\rm P}(\OO(p) \times \OO(q))$, but on a properly convex domain in the projective model of the pseudo-Riemannian hyperbolic space $\Hb^{p,q-1}$. In this context convex cocompactness can be defined as follows:

\begin{definition}\cite{DGK2017a} 
Let
\begin{align*}
\Hb^{p,q-1} = \{ [x]\in \Pb(\Rb^{p+q}) : \ip{x,x}_{p,q} < 0\}
\end{align*}
where $\ip{\cdot, \cdot}_{p,q}$ is the standard bilinear form of signature $(p,q)$. Then an irreducible discrete subgroup $\Lambda \leq \PO(p,q)$ is called \emph{$\Hb^{p,q-1}$-convex cocompact} if there exists a non-empty properly convex subset $\Cc$ of $\Pb(\Rb^{p+q})$ such that 
\begin{enumerate}
\item $\Cc$ is a closed subset of $\Hb^{p,q-1}$,
 \item $\Lambda$ acts properly discontinuously and cocompactly on $\Cc$, 
 \item $\Cc$ has non-empty interior, and 
 \item $\overline{\Cc} \backslash \Cc$ contains no projective line segments. 
 \end{enumerate}
\end{definition}

Danciger, Gu{\'e}ritaud, and Kassel then proved the following theorem. 

\begin{theorem}[{Danciger-Gu{\'e}ritaud-Kassel~\cite{DGK2017a}}]\label{thm:SO_case} For $p,q \in \Nb^*$ with $p+q \geq 3$, let $\Lambda$ be an irreducible discrete subgroup of $\PO(p,q)$ and let $P^{p,q}_1 \leq \PO(p,q)$ be the stabilizer of an isotropic line in $(\Rb^{p+q}, \ip{\cdot, \cdot}_{p,q})$. 
\begin{enumerate}
\item If $\Lambda$ is $\Hb^{p,q-1}$-convex cocompact, then it is word hyperbolic and the inclusion representation $\Lambda \hookrightarrow \PO(p,q)$ is $P^{p,q}_1$-Anosov.
\item Conversely, if $\Lambda$ is word hyperbolic, $\partial \Lambda$ is connected, and $\Lambda \hookrightarrow \PO(p,q)$ is $P^{p,q}_1$-Anosov, then $\Lambda$ is either $\Hb^{p,q-1}$-convex cocompact or $\Hb^{q,p-1}$-convex cocompact (after identifying $\PO(p,q)$ with $\PO(q,p)$). 
\end{enumerate}
 \end{theorem}
 
 \begin{remark}
 The special case when $q=2$ and $\Lambda$ is the fundamental group of a closed hyperbolic $p$-manifold follows from work of Mess~\cite{M2007} for $p = 2$ and  work of Barbot-M{\'e}rigot~\cite{BM2012} for $p \geq  3$.
 \end{remark}

In this paper we further explore connections between Anosov representations and convex cocompact actions on domains in real projective space.  In the general case, we make the following definition.

\begin{definition}\label{defn:CC}
Suppose $V$ is a finite dimensional real vector space, $\Omega \subset \Pb(V)$ is a properly convex domain, and $\Lambda \leq \Aut(\Omega)$ is a discrete subgroup. Then $\Lambda$ is a \emph{convex cocompact subgroup} of $\Aut(\Omega)$ if there exists a non-empty closed convex subset $\Cc \subset \Omega$ such that $g(\Cc) = \Cc$ for all $g \in \Lambda$ and the quotient $\Lambda \backslash \Cc$ is compact.  
\end{definition}

In the context of Anosov representations a more refined notion of convex cocompactness seems to be necessary: there exist properly convex domains $\Omega \subset \Pb(V)$ with convex cocompact subgroups $\Lambda \leq \Aut(\Omega)$ which are not word hyperbolic. To avoid such examples, we make the following stronger definition.

\begin{definition}\label{def:RCC} Suppose $V$ is a finite dimensional real vector space, $\Omega \subset \Pb(V)$ is a properly convex domain, and $\Lambda \leq \Aut(\Omega)$ is a discrete subgroup. Then $\Lambda$ is a \emph{regular convex cocompact subgroup} of $\Aut(\Omega)$ if there exists a non-empty closed convex subset $\Cc \subset \Omega$ such that $g(\Cc) = \Cc$ for all $g \in \Lambda$, the quotient $\Lambda \backslash \Cc$ is compact, and every point in $\overline{\Cc} \cap \partial \Omega$ is a $C^1$ extreme point of $\Omega$.
\end{definition}

\begin{remark}\label{rmk:connections_2} \ 
\begin{enumerate}
\item When $\Lambda$ is an irreducible subgroup of $\PGL(V)$, we show that every point in $\overline{\Cc} \cap \partial \Omega$ is a $C^1$ point of $\Omega$ if and only if every point in $\overline{\Cc} \cap \partial \Omega$ is a extreme point of $\Omega$ (see Theorem~\ref{thm:RCC_anosov_intro} below).
\item It turns out that $\Hb^{p,q-1}$-convex cocompact subgroups always satisfy this stronger condition. In particular, by Proposition 1.14 in~\cite{DGK2017a}: If $\Gamma \leq \PO(p,q)$ is irreducible, discrete, and $\Hb^{p,q-1}$-convex cocompact, then there exists a properly convex domain $\Omega \subset \Pb(\Rb^{p+q})$ such that $\Lambda$ is a regular convex cocompact subgroup of $\Aut(\Omega)$. 
\end{enumerate}
\end{remark}

Finally we are ready to state our first main result.

\begin{theorem}\label{thm:main}(see Section~\ref{subsec:pf_main_thm})
Suppose $G$ is a semisimple Lie group with finite center and $P \leq G$ is a parabolic subgroup. Then there exists a finite dimensional real vector space $V$ and an irreducible representation $\phi:G \rightarrow \PSL(V)$ with the following property: if $\Gamma$ is a word hyperbolic group and $\rho:\Gamma \rightarrow G$ is a Zariski dense representation with finite kernel, then the following are equivalent:
\begin{enumerate}
\item $\rho$ is $P$-Anosov,
\item there exists a properly convex domain $\Omega \subset \Pb(V)$ such that $(\phi\circ\rho)(\Gamma)$ is a regular convex cocompact subgroup of $\Aut(\Omega)$. 
\end{enumerate}
\end{theorem}

Properly convex domains and their projective automorphism groups have been extensively studied, especially in the case in which there exists a discrete group $\Gamma \leq \Aut(\Omega)$ such that $\Gamma \backslash \Omega$ is compact. Such domains are called \emph{convex divisible domains} and have a number of remarkable properties, see the survey papers by Benoist~\cite{B2008}, Marquis~\cite{M2014}, and Quint~\cite{Q2010}.

Theorem~\ref{thm:main} provides a way to use the rich theory of convex divisible domains to study general Anosov representations. For instance, the proofs of Theorem~\ref{thm:ent_rig} and Theorem~\ref{thm:main_reg_rigid_intro} below are inspired by  rigidity results for convex divisible domains. 

\subsection{Projective Anosov representations}

The first step in the proof of Theorem~\ref{thm:main} is to use a result of Guichard and Wienhard to reduce to the case of projective Anosov representations. A projective Anosov representation is simply an $P$-Anosov representation in the special case when $G=\PGL_d(\Rb)$ and $P \leq \PGL_d(\Rb)$ is the stabilizer of a line. This special class of Anosov representations can be defined as follows.

\begin{definition} Suppose that $\Gamma$ is a word hyperbolic group, $\partial \Gamma$ is the Gromov boundary of $\Gamma$, and $\rho: \Gamma \rightarrow \PGL_{d}(\Rb)$ is a representation.  Two maps $\xi: \partial \Gamma \rightarrow \Pb(\Rb^{d})$ and $\eta: \partial \Gamma \rightarrow \Pb(\Rb^{d*})$ are called:
\begin{enumerate}
\item \emph{$\rho$-equivariant} if $\xi \circ \gamma  = \rho(\gamma) \circ \xi$ and $\eta\circ \gamma = \rho(\gamma)\circ \eta$ for all  $\gamma \in \Gamma$,
\item \emph{dynamics-preserving} if for every $\gamma \in \Gamma$ of infinite order with attracting fixed point $x_\gamma^+ \in \partial \Gamma$ the points $\xi(x^+_{\gamma}) \in \Pb(\Rb^{d})$ and $\eta(x^+_{\gamma}) \in \Pb(\Rb^{d*})$ are attracting fixed points of the action of $\rho(\gamma)$ on $\Pb(\Rb^{d})$ and $\Pb(\Rb^{d*})$, and
\item \emph{transverse} if for every distinct pair $x, y \in \partial \Gamma$ we have $\xi(x) + \ker \eta(y) = \Rb^{d}$.
\end{enumerate}
\end{definition}

\begin{definition} Given an element $g \in \PGL_{d}(\Rb)$ let 
\begin{align*}
\lambda_1(g) \geq \dots \geq \lambda_{d}(g)
\end{align*}
denote the absolute values of the eigenvalues (counted with multiplicity) of some (any) lift $\tilde{g} \in \GL_d(\Rb)$ of $g$ with $\det \tilde{g}=\pm 1$. 
\end{definition}
 
\begin{definition} Suppose that $\Gamma$ is word hyperbolic, $S$ is a finite symmetric generating set, and $d_S$ is the associated word metric on $\Gamma$. Then for $\gamma \in \Gamma$, let $\ell_S(\gamma)$ denote the minimal translation distance of $\gamma$ acting on the Cayley graph of $(\Gamma,S)$, that is
\begin{align*}
\ell_S(\gamma) = \inf_{x \in \Gamma} d_S(\gamma x, x).
\end{align*}
A representation $\rho: \Gamma \rightarrow \PGL_d(\Rb)$ is then called a \emph{projective Anosov representation} if there exist continuous, $\rho$-equivariant, dynamics preserving, and transverse maps $\xi: \partial \Gamma \rightarrow \Pb(\Rb^{d})$, $\eta: \partial \Gamma \rightarrow \Pb(\Rb^{d*})$ and constants $C,c>0$ such that
\begin{align}
\label{eq:eigenvalue_est}
\log  \frac{\lambda_1(\rho(\gamma))}{\lambda_{2}(\rho(\gamma))} \geq C \ell_S(\gamma) -c
 \end{align}
for all $\gamma \in \Gamma$.
 \end{definition}
 
 \begin{remark} This is not the initial definition of Anosov representations given by Labourie~\cite{L2006} or Guichard-Wienhard~\cite{GW2012}, but a nontrivial characterization proved in~\cite[Theorem 1.7]{GGKW2015}. We use this characterization as our definition because it is more elementary to state than the original definition, but it is not necessary for any of the proofs in the paper. We should also note that this is far from the simplest definition of Anosov representations, for instance if one replaces the estimate in Equation~\eqref{eq:eigenvalue_est} with a similar estimate on singular values, then it follows from work of Kapovich, Leeb, and Porti~\cite{KLP2014b} that one does not need to assume that the maps $\eta,\xi$ exist or even that $\Gamma$ is a word hyperbolic group (only finite generation is required). But since many of the results that follow involve these boundary maps and eigenvalues, it seems like this definition is the most natural in the context of this paper. 
 \end{remark}
 
Guichard and Wienhard proved the following connection between general Anosov representations and projective Anosov representations. 
 
 \begin{theorem}[{Guichard-Wienhard~\cite[Section 4]{GW2012}}]\label{thm:GW}
Suppose $G$ is a semisimple Lie group with finite center and $P \leq G$ is a parabolic subgroup. Then there exist a finite dimensional real vector space $V_0$ and an irreducible representation $\phi_0:G \rightarrow \PSL(V_0)$ with the following property: if $\Gamma$ is a word hyperbolic group and $\rho:\Gamma \rightarrow G$ is a representation, then the following are equivalent:
\begin{enumerate}
\item $\rho$ is $P$-Anosov,
\item $\phi_0 \circ \rho$ is projective Anosov.
\end{enumerate}
\end{theorem}

\begin{remark} Proofs of this theorem can also be found in~\cite[Section 3]{GGKW2015} and~\cite[Subsection 2.3]{BCLS2015}. \end{remark}

Using Theorem~\ref{thm:GW}, the proof of Theorem~\ref{thm:main} essentially reduces to the case of projective Anosov representations. In this case, we consider the following two questions.

\begin{question} \ \begin{enumerate}
\item Given a properly convex domain $\Omega \subset \Pb(V)$ and a convex cocompact subgroup $\Lambda \leq \Aut(\Omega)$, what geometric conditions on $\Omega$ imply that the inclusion representation $\Lambda \hookrightarrow \PGL(V)$ is a projective Anosov representation?
\item Given a projective Anosov representation $\rho: \Gamma \rightarrow \PGL(V)$ what conditions on $\rho$ or $\Gamma$ imply that $\rho(\Gamma)$ acts convex cocompactly on a properly convex domain in $\Pb(V)$?
\end{enumerate}
\end{question}

\begin{remark}\label{rmk:connections}
Projective Anosov representations are closely related to the representations studied by Danciger, Gu{\'e}ritaud, and Kassel~\cite{DGK2017a} in the $\PO(p,q)$ case. In particular, if $\rho: \Gamma \rightarrow \PO(p,q)$ is a representation of a word hyperbolic group, then (by definition) the following are equivalent:
 \begin{enumerate}
 \item $\rho$ is  $P^{p,q}_1$-Anosov where $P^{p,q}_1 \leq \PO(p,q)$ be the stabilizer of an isotropic line in $(\Rb^{p+q}, \ip{\cdot, \cdot}_{p,q})$,
 \item $\rho$ is projective Anosov when viewed as a representation into $\PGL_{p+q}(\Rb)$. 
 \end{enumerate}
 Thus Theorem~\ref{thm:SO_case} provides answers to the above questions for projective Anosov representations whose images preserve a non-degenerate bilinear form.
 \end{remark}
 
\subsection{When a convex cocompact action leads to a projective Anosov representation}

When $\Omega \subset \Pb(\Rb^{d})$ and $\Lambda \leq \Aut(\Omega)$ is a discrete group which acts cocompactly on $\Omega$, Benoist has provided geometric conditions on $\Omega$ so that the inclusion representation $\Lambda \hookrightarrow \PGL_{d}(\Rb)$ is projective Anosov.

\begin{theorem}[Benoist~\cite{B2004}]\label{thm:ben_char}
Suppose $\Omega \subset \Pb(\Rb^{d})$ is a properly convex domain and $\Lambda \leq \Aut(\Omega)$ is a discrete group which acts cocompactly on $\Omega$. Then the following are equivalent:
\begin{enumerate}
\item $\Lambda$ is word hyperbolic, 
\item $\partial \Omega$ is a $C^1$ hypersurface, 
\item $\Omega$ is strictly convex.
\end{enumerate}
Moreover, when these conditions are satisfied the inclusion representation $\Lambda \hookrightarrow \PGL_{d}(\Rb)$ is projective Anosov.
\end{theorem}

\begin{remark} There exist examples of properly convex domains $\Omega \subset \Pb(\Rb^d)$ with discrete subgroups $\Lambda \leq \Aut(\Omega)$ where $\Lambda$ acts co-compactly on $\Omega$ and $\Lambda$ is not word hyperbolic, see~\cite{B2006} and~\cite{BDL2018}.
\end{remark}

The case of convex cocompact actions is more complicated as the next example shows.

\begin{example}\label{ex:bad_example} Let 
\begin{align*}
C = \left\{ (x_1, x_2, y) \in \Rb^3 : y> \sqrt{x_1^2 + x_2^2} \right\}.
\end{align*}
Then $C$ is a properly convex cone and the group $\SO_0(1,2)$ preserves $C$. Let $\Lambda_0 \leq \SO_0(1,2)$ be a cocompact lattice. Next consider the properly convex domain
\begin{align*}
\Omega = \{ [(v_1, v_2)] \in \Pb(\Rb^6) : v_1 \in C, v_2 \in C\}
\end{align*}
and the discrete group 
\begin{align*}
\Lambda = \left\{ \begin{bmatrix} \varphi & 0 \\ 0 & \varphi \end{bmatrix}  \in \PGL_6(\Rb) : \varphi \in \Lambda_0 \right\}.
\end{align*}
Let $\Cc_0 = \{ [(v,v)] \in \Pb(\Rb^d) : v \in C\}$ and for  $r > 0$ let 
\begin{align*}
\Cc_{r} = \{ p \in \Omega :  d_\Omega(p,\Cc_0) \leq r \} \subset \Omega.
\end{align*}
Then each $\Cc_r$ is convex (see \cite[Result 18.9]{HB1955} or \cite[Corollary 1.10]{CLT2015}) and the quotient $\Lambda \backslash \Cc_r$ is compact. This example has the following properties:
\begin{enumerate}
\item $\Lambda$ is word hyperbolic (since $\Lambda_0$ is word hyperbolic),
\item the inclusion representation $\Lambda \hookrightarrow \PGL_6(\Rb)$ is not projective Anosov, 
\item $\Lambda \leq \Aut(\Omega)$ is a convex cocompact subgroup,
\item when $r>0$  there exist line segments in $\partial \Omega \cap \overline{\Cc_{r}}$, and
\item every point in $\partial \Omega \cap \overline{\Cc_{r}}$ is not a $C^1$ point of $\partial\Omega$. 
\end{enumerate}
\end{example}

Despite examples like these, we will prove the following analogue of Benoist's theorem for convex cocompact subgroups 

\begin{theorem}(see Section~\ref{sec:convex_cocpct_implies_A})\label{thm:RCC_anosov_intro}
Suppose $\Omega \subset \Pb(\Rb^{d})$ is a properly convex domain and $\Lambda \leq \Aut(\Omega)$ is a discrete irreducible subgroup of $\PGL_{d}(\Rb)$. If $\Lambda$ preserves and acts cocompactly on a closed convex subset $\Cc \subset \Omega$, then the following are equivalent:
\begin{enumerate}
\item every point in $\overline{\Cc} \cap \partial \Omega$ is a $C^1$ point of $\partial \Omega$, 
\item every point in $\overline{\Cc} \cap \partial \Omega$ is an extreme point of $ \Omega$.
\end{enumerate}
Moreover, when these conditions are satisfied $\Lambda$ is word hyperbolic and the inclusion representation $\Lambda \hookrightarrow \PGL_{d}(\Rb)$ is projective Anosov.
\end{theorem}

\begin{remark} \ \begin{enumerate}
\item Theorem~\ref{thm:RCC_anosov_intro} can be seen as a generalization of Theorem~\ref{thm:SO_case} part (1) to the case when the representation is not assumed to preserve a non-degenerate bilinear form (see Remarks~\ref{rmk:connections_2} and~\ref{rmk:connections}). 
\item This result was established independently by Danciger, Gu{\'e}ritaud, and Kassel, see Theorems 1.4 and 1.15 in~\cite{DGK2017b} and Subsection~\ref{subsec:DGK} below. 
\end{enumerate}
\end{remark}

\subsection{When a projective Anosov representation acts convex cocompactly}

In general a projective Anosov representation will not preserve a properly convex domain:

\begin{example}\label{ex:non_convex_cocompact} Consider a cocompact lattice $\Lambda \leq \SL_2(\Rb)$ and consider the representation $\rho: \SL_2(\Rb) \rightarrow \SL_3(\Rb)$ given by 
\begin{align*}
\rho(g) = \begin{pmatrix} g & \\ & 1 \end{pmatrix}.
\end{align*}
Then the representation $\rho|_\Lambda: \Lambda \rightarrow \PSL_3(\Rb)$ is projective Anosov and the image of the boundary map is 
\begin{align*}
\Lc:=\{ [ x_1 : x_2 : 0 ] \in \Pb(\Rb^3) : x_1, x_2 \in \Rb, (x_1, x_2) \neq 0\}.
\end{align*}
From this, it is easy to see that $\rho(\Lambda)$ cannot preserve a properly convex domain $\Omega$ because then we would have $\Lc \subset \partial \Omega$. 
\end{example}

The above example is simple to construct, but is not an irreducible representation. To obtain an example of an irreducible projective Anosov representation which does not preserve a properly convex domain, one can consider Hitchin representations of surface groups in $\SL_{2d}(\Rb)$, see Proposition 1.7 in~\cite{DGK2017b}.

With some mild conditions on $\Gamma$ we can prove that every projective Anosov representation of $\Gamma$ acts convex cocompactly on a properly convex domain.

\begin{theorem}\label{thm:condC} (see Section~\ref{sec:construction}) Suppose $\Gamma$ is a non-elementary word hyperbolic group which is not commensurable to a non-trivial free product or the fundamental group of a closed hyperbolic surface. If $\rho: \Gamma \rightarrow \PGL_{d}(\Rb)$ is an irreducible projective Anosov representation, then there exists a properly convex domain $\Omega \subset \Pb(\Rb^{d})$ such that $\rho(\Gamma)$ is a regular convex cocompact subgroup of $\Aut(\Omega)$.
\end{theorem}

\begin{remark} Work of Stallings implies that $\Gamma$ is not commensurable to a non-trivial free product if and only if $\partial \Gamma$ is connected ~\cite{S1971, S1968}. So Theorem~\ref{thm:condC} can be seen as an analogue of Theorem~\ref{thm:SO_case} part (2) in the case when the representation is not assumed to preserve a non-degenerate bilinear form. 
\end{remark}

We can also prove that once the image acts on some properly convex domain, then it acts convex cocompactly on some (possibly different) properly convex domain: 

\begin{theorem}\label{thm:proper_action_implies} (see Section~\ref{sec:construction}) Suppose $\Gamma$ is a word hyperbolic group. If $\rho: \Gamma \rightarrow \PGL_{d}(\Rb)$ is an irreducible projective Anosov representation and $\rho(\Gamma)$ preserves a properly convex domain in $\Pb(\Rb^{d})$, then there exists a properly convex domain $\Omega \subset \Pb(\Rb^{d})$ such that $\rho(\Gamma)$ is a regular convex cocompact subgroup of $\Aut(\Omega)$.
\end{theorem}

\begin{remark} This result was established independently by Danciger, Gu{\'e}ritaud, and Kassel, see Theorems 1.4 and 1.15 in~\cite{DGK2017b} and Subsection~\ref{subsec:DGK} below. 
\end{remark}

Using Theorem~\ref{thm:proper_action_implies}, we can construct a convex cocompact action for any projective Anosov representation by post composing with another representation. 

\begin{example}\label{ex:pos_def_cone} Let $\Sym_{d}(\Rb)$ be the vector space of symmetric $d$-by-$d$ real matrices and consider the representation 
\begin{align*}
S: \PGL_{d}(\Rb) \rightarrow \PGL( \Sym_{d}(\Rb))
\end{align*}
given by 
\begin{align*}
S(g) X = gX \, \prescript{t}{}{g}
\end{align*}
Then 
\begin{align*}
\Pc := \{ [X] \in \Pb( \Sym_{d}(\Rb) ): X > 0\}
\end{align*}
is a properly convex domain in $\Pb(\Sym_{d}(\Rb))$ and $S(\PGL_{d}(\Rb)) \leq \Aut(\Pc)$. 
\end{example}

Combining Theorem~\ref{thm:proper_action_implies} with the above examples establishes the following corollary. 

\begin{corollary}\label{cor:convex_cocompact} (see Section~\ref{sec:cor_convex_cocompact})
 Suppose $\Gamma$ is a word hyperbolic group and $\rho: \Gamma \rightarrow \PGL_{d}(\Rb)$ is an irreducible projective Anosov representation. 
 Let
 \begin{align*}
 V = \Spanset_{\Rb} \{ \xi(x)  \prescript{t}{}{\xi(x)}: x \in \partial \Gamma\} \subset \Sym_{d}(\Rb).
 \end{align*}
 Then there exists a properly convex domain $\Omega \subset \Pb(V)$ such that $(S\circ \rho)(\Gamma)$ is a regular convex cocompact subgroup of $\Aut(\Omega)$.  \end{corollary}

In the context of Theorem~\ref{thm:proper_action_implies}, it is also worth mentioning a theorem of Benoist which gives a necessary and sufficient condition for a subgroup of $\GL_{d}(\Rb)$ to preserve a properly convex cone. Before stating Benoist theorem we need some terminology. An element $g \in \GL_{d}(\Rb)$ is called \emph{proximal} if it has a unique eigenvalue of maximal absolute value and a proximal element $g \in \GL_{d}(\Rb)$ is called \emph{positively proximal} if its unique eigenvalue of maximal absolute value is positive. Then a subgroup $G \leq \GL_{d}(\Rb)$ is called \emph{positively proximal} if $G$ contains a proximal element and every proximal element in $G$ is positively proximal. With this language, Benoist proved the following theorem. 

\begin{theorem}[{Benoist~\cite[Proposition 1.1]{B2000}}] If $G \leq \GL_{d}(\Rb)$ is an irreducible subgroup, then the following are equivalent: 
\begin{enumerate}
\item $G$ is positively proximal
\item $G$ preserves a properly convex cone $\Cc \subset \Rb^{d}$.
\end{enumerate}
\end{theorem} 

As an application, we will apply Theorem~\ref{thm:proper_action_implies} and Benoist's theorem to Hitchin representations in certain dimensions.

\begin{definition} Suppose that $\Gamma \leq \PSL_2(\Rb)$ is a torsion-free cocompact lattice and $\iota: \Gamma \hookrightarrow \PSL_2(\Rb)$ is the inclusion representation. For $d > 2$, let $\tau_d : \PSL_2(\Rb) \rightarrow \PSL_d(\Rb)$ be the unique (up to conjugation) irreducible representation. Then the connected component of $\tau_d \circ \iota$ in $\Hom(\Gamma, \PSL_d(\Rb))$, denoted $\Hc_d(\Gamma)$, is called the \emph{Hitchin component of $\Gamma$ in $\PSL_d(\Rb)$}.  Labourie~\cite{L2006} proved that every representation in $\Hc_d(\Gamma)$ is projective Anosov (it is actually $B$-Anosov where $B \leq \PSL_d(\Rb)$ is a minimal parabolic subgroup). 
\end{definition}

\begin{corollary}\label{cor:H_acts_CC}  Suppose that $\Gamma \leq \PSL_2(\Rb)$ is a torsion-free cocompact lattice and $\rho: \Gamma \rightarrow \PSL_d(\Rb)$ is in the Hitchin component. If $d$ is odd, then there exists a properly convex domain $\Omega \subset \Pb(\Rb^d)$ such that $\rho(\Gamma)$ is a regular convex cocompact subgroup of $\Aut(\Omega)$.
\end{corollary}

\begin{remark} This result was also established independently by Danciger, Gu{\'e}ritaud, and Kassel, see Proposition 1.7 in~\cite{DGK2017b} and Subsection~\ref{subsec:DGK} below. In the case where $d$ is even, Danciger, Gu{\'e}ritaud, and Kassel showed that $\rho(\Gamma)$ cannot even preserve a properly convex domain in $\Rb(\Rb^d)$. 
\end{remark}

Since the proof is short we include it here.

\begin{proof} If we identify $\Rb^d$ with the vector space of homogenous polynomials $P: \Rb^d \rightarrow \Rb$ of degree $d-1$, then the representation $\tau_d : \PSL_2(\Rb) \rightarrow \PSL_d(\Rb)$ is given by 
\begin{align*}
\tau_d(g) \cdot P = P \circ g^{-1}
\end{align*}

Since $d$ is odd, $\PSL_d(\Rb) = \SL_d(\Rb)$ and if $g \in \PSL_2(\Rb)$ has eigenvalues with absolute values $\lambda, \lambda^{-1}$ then $\tau_d(g)$ has eigenvalues 
\begin{align*}
\lambda^{d-1}, \lambda^{d-3}, \dots, \lambda^{-(d-1)}.
\end{align*}
Hence each eigenvalue of $\tau_d(g)$ is positive. 

Now fix some $\rho \in \Hc_d(\Gamma)$. Since $\Hc_d(\Gamma)$ is connected, we see that $\rho(\Gamma)$ is positively proximal. So by Benoist's theorem $\rho(\Gamma)$ preserves a properly convex cone $\Cc \subset \Rb^{d}$. So by Theorem~\ref{thm:proper_action_implies},  there exists a properly convex domain $\Omega \subset \Pb(\Rb^d)$ so that $\rho(\Gamma)$ is a regular convex cocompact subgroup of $\Aut(\Omega)$.
\end{proof}

\subsection{Other applications:}

\subsubsection{Entropy rigidity:}

Suppose that $\Gamma$ is a group and let $[\Gamma]$ be the conjugacy classes of $\Gamma$. Given a representation $\rho: \Gamma \rightarrow \PGL_{d}(\Rb)$ define the \emph{Hilbert entropy} to be 
\begin{align*}
H_\rho= \limsup_{r \rightarrow \infty} \frac{1}{r} \log \# \left\{ [\gamma] \in [\Gamma] : \frac{1}{2} \log \left(\frac{\lambda_{1}(\rho(\gamma))}{\lambda_{d}(\rho(\gamma))} \right)\leq r\right\}.
\end{align*}

We will prove the following upper bound on entropy.

\begin{theorem}\label{thm:ent_rig}(see Section~\ref{sec:ent_rigidity})
Suppose $\Gamma$ is a word hyperbolic group and $\rho:\Gamma \rightarrow \PGL_{d}(\Rb)$ is an irreducible projective Anosov representation. If $\rho(\Gamma)$ preserves a properly convex domain in $\Pb(\Rb^{d})$, then 
\begin{align*}
H_\rho \leq d-2
\end{align*}
with equality if and only if $\rho(\Gamma)$ is conjugate to a cocompact lattice in $\PO(1,d-1)$. 
\end{theorem}

\begin{remark}
Theorem~\ref{thm:condC} shows that Theorem~\ref{thm:ent_rig} applies to many Anosov representations. 
\end{remark}

Theorem~\ref{thm:ent_rig} is a generalization of a theorem of Crampon.

\begin{theorem}[Crampon~\cite{Cra2009}]\label{thm:crampon} Suppose $\Omega \subset \Pb(\Rb^{d})$ is a properly convex domain and $\Lambda \leq \Aut(\Omega)$ is a discrete word hyperbolic group which acts cocompactly on $\Omega$. If $\iota: \Lambda \hookrightarrow \PGL_{d}(\Rb)$ is the inclusion representation, then 
\begin{align*}
H_{\iota} \leq d-2
\end{align*}
with equality if and only if $\Lambda$ is conjugate to a cocompact lattice  in $\PO(1,d-1)$. 
\end{theorem}

\begin{remark} In the context of Theorem~\ref{thm:crampon}, Theorem~\ref{thm:ben_char} implies that $\iota$ is a projective Anosov representation and so Theorem~\ref{thm:ent_rig} is a true generalization of Theorem~\ref{thm:crampon}. Recently, Theorem~\ref{thm:crampon} was also generalized in a different direction in~\cite{BMZ2015}. \end{remark}

Theorem~\ref{thm:ent_rig} also improves, in some cases, bounds due to Sambarino.

\begin{theorem}[{Sambarino~\cite[Theorem A]{S2016}}]\label{thm:S} Suppose $\Gamma$ is a convex cocompact group of a $\CAT(-1)$ space $X$ and let $\rho:\Gamma \rightarrow \PGL_{d}(\Rb)$ be an irreducible projective Anosov representation with $d \geq 3$. Then 
\begin{align*}
\alpha H_{\rho} \leq \delta_\Gamma(X)
\end{align*}
where the boundary map $\xi: \partial \Gamma \rightarrow \Pb(\Rb^{d})$ is $\alpha$-H{\"o}lder and $\delta_\Gamma(X)$ is the Poincar{\'e} exponent of $\Gamma$ acting on $X$.
\end{theorem}

\begin{remark} In Theorem~\ref{thm:S}, $\xi$ is H{\"o}lder with respect to a visual metric of $X$ restricted to the limit set of $\Gamma$ and a distance on $\Pb(\Rb^{d})$ induced by a Riemannian metric. Sambarino also proves a rigidity result in the case when $\alpha H_{\rho} = \delta_\Gamma(X)$ and $X$ is real hyperbolic $k$-space, for details see Corollary 3.1 in~\cite{S2016}. \end{remark}

\begin{remark} If $\Gamma$ satisfies the hypothesis of Theorem~\ref{thm:condC} and
\begin{align*}
\alpha < \frac{\delta_\Gamma(X)}{d-2},
\end{align*}
then Theorem~\ref{thm:ent_rig} can be used to provide a better upper bound on entropy 
\end{remark}

\subsubsection{Regularity rigidity} In this subsection we describe some rigidity results related to the regularity of the limit curve of a projective Anosov representation. We should note that if the boundary of a word hyperbolic group is a topological manifold, then it actually must be a sphere (see for instance~\cite[Theorem 4.4]{KB2002}).

For certain types of projective Anosov representations, the image of the boundary map is actually a $C^1$ submanifold. 

\begin{example} 
Suppose $\Omega \subset \Pb(\Rb^{d})$ is a properly convex domain and $\Lambda \leq \Aut(\Omega)$ is a discrete group which acts cocompactly on $\Omega$. If $\Lambda$ is word hyperbolic, then Theorem~\ref{thm:ben_char} implies that the inclusion representation $\Lambda \hookrightarrow \PGL_{d}(\Rb)$ is projective Anosov. The image of the associated boundary map is $\partial \Omega$ which is a $C^1$ submanifold of $\Pb(\Rb^{d})$ by Theorem~\ref{thm:ben_char}.
\end{example}

\begin{example} Suppose that $\Gamma \leq \PSL_2(\Rb)$ is a torsion-free cocompact lattice and $\rho : \Gamma \rightarrow \PSL_d(\Rb)$ is in the Hitchin component. If $\xi: \partial \Gamma \rightarrow \Pb(\Rb^{d})$ is the boundary map associated to $\rho$, then $\xi(\partial \Gamma)$ is a $C^1$ submanifold of $\Pb(\Rb^d)$. This follows from the fact that $\xi$ is a \emph{hyperconvex Frenet curve}, see ~\cite[Theorem 1.4]{L2006}.
\end{example}

In both of theses cases it is known that the image of the boundary map cannot be too regular unless the representation is very special.

\begin{theorem}[Benoist~\cite{B2004}] Suppose $\Omega \subset \Pb(\Rb^{d})$ is a properly convex domain and $\Lambda \leq \Aut(\Omega)$ is a discrete group which acts cocompactly on $\Omega$. If $\partial \Omega$ is a $C^{1,\alpha}$ hypersurface for every $\alpha \in (0,1)$, then $\Omega$ is projectively isomorphic to the ball and hence $\Lambda$ is conjugate to a cocompact lattice in $\PO(1,d-1)$. 
\end{theorem}

\begin{theorem}[Potrie-Sambarino~\cite{PS2014}] Suppose that $\Gamma \leq \PSL_2(\Rb)$ is a torsion-free cocompact lattice  and $\rho : \Gamma \rightarrow \PSL_d(\Rb)$ is in the Hitchin component. If $\xi: \partial \Gamma \rightarrow \Pb(\Rb^{d})$ is the associated boundary map and $\xi(\partial \Gamma)$ is a $C^\infty$ submanifold of $\Pb(\Rb^d)$, then there exists a representation $\rho_0 : \Gamma \rightarrow \PSL_2(\Rb)$ such that $\rho$ is conjugate to $\tau_d \circ \rho_0$. 
\end{theorem}

Using Theorem~\ref{thm:condC}, we will prove the following.

\begin{theorem}\label{thm:main_reg_rigid_intro}(see Section~\ref{sec:reg_rigid}) Suppose $d > 2$, $\Gamma$ is a word hyperbolic group, and $\rho:\Gamma \rightarrow \PGL_{d}(\Rb)$ is an irreducible projective Anosov representation with boundary map $\xi: \partial \Gamma \rightarrow \Pb(\Rb^{d})$. If 
\begin{enumerate}
\item $M = \xi(\partial \Gamma)$ is a $C^2$ $k$-dimensional submanifold of $\Pb(\Rb^{d})$ and
\item the representation $\wedge^{k+1} \rho : \Gamma \rightarrow \PGL(\wedge^{k+1} \Rb^{d})$ is  irreducible,
\end{enumerate}
then
\begin{align*}
\frac{\lambda_1(\rho(\gamma))}{\lambda_2(\rho(\gamma))} =\frac{\lambda_{k+1}(\rho(\gamma))}{\lambda_{k+2}(\rho(\gamma))}
\end{align*}
for all $\gamma \in \Gamma$. 
\end{theorem}

\begin{remark}\ \begin{enumerate} \item Notice that the regularity assumption concerns the set $\xi(\partial \Gamma)$ and not the map  $\xi: \partial \Gamma \rightarrow \Pb(\Rb^d)$.
 \item As before, $\lambda_1(g)  \geq \dots \geq \lambda_{d}(g)$ denote the absolute values of the eigenvalues (counted with multiplicity) of some (any) lift $\tilde{g} \in \GL_d(\Rb)$ of $g$ with $\det \tilde{g}=\pm 1$. 
\item Theorem~\ref{thm:condC} is only needed in the case when $k >1$. 
\end{enumerate}
\end{remark}

When $\rho: \Gamma \rightarrow \PGL_{d}(\Rb)$ has Zariski dense image, then $\rho$ and $\wedge^{k+1} \rho$ are  irreducible. Moreover in this case the main result in~\cite{B1997} implies that there exists some $\gamma \in \Gamma$ such that 
\begin{align*}
\frac{\lambda_1(\rho(\gamma))}{\lambda_2(\rho(\gamma))} \neq \frac{\lambda_{k+1}(\rho(\gamma))}{\lambda_{k+2}(\rho(\gamma))}.
\end{align*}
So we have the following corollary of Theorem~\ref{thm:main_reg_rigid_intro}.

\begin{corollary}\label{cor:hitchin}Suppose $d > 2$, $\Gamma$ is a word hyperbolic group, and $\rho:\Gamma \rightarrow \PGL_{d}(\Rb)$ is a Zariski dense projective Anosov representation with boundary map $\xi: \partial \Gamma \rightarrow \Pb(\Rb^{d})$. Then $\xi(\partial \Gamma)$ is not a $C^2$ submanifold of $\Pb(\Rb^{d})$. 
\end{corollary}

%When $\Gamma \leq \PSL_2(\Rb)$ is a cocompact lattice and $\rho : \Gamma \rightarrow \PSL_d(\Rb)$ is in the Hitchin component, then the representations $\rho$ and $\wedge^{2} \rho$ are irreducible by~\cite{L2006} and so we have the following corollary of Theorem~\ref{thm:main_reg_rigid}:

The proof of Theorem~\ref{thm:main_reg_rigid_intro} can also be used to prove the following rigidity result for Hitchin representations. 

\begin{theorem}\label{thm:main_reg_rigid_hitchin_intro}(see Section~\ref{sec:reg_rigid}) Suppose that $\Gamma \leq \PSL_2(\Rb)$ is a torsion-free cocompact lattice  and  $\rho : \Gamma \rightarrow \PSL_d(\Rb)$ is in the Hitchin component. If $\xi: \partial \Gamma \rightarrow \Pb(\Rb^{d})$ is the associated boundary map and $\xi(\partial \Gamma)$ is a $C^2$ submanifold of $\Pb(\Rb^d)$, then 
\begin{align*}
\frac{\lambda_1(\rho(\gamma))}{\lambda_2(\rho(\gamma))} = \frac{\lambda_2(\rho(\gamma))}{\lambda_3(\rho(\gamma))}
\end{align*}
for all $\gamma \in \Gamma$. 
\end{theorem}

\begin{remark} This corollary greatly restricts the Zariski closure of $\rho(\Gamma)$ when $\rho$ is Hitchin and $\xi(\partial \Gamma)$ is a $C^2$ submanifold (see~\cite{B1997} again). In particular, the corollary implies that in this case:
\begin{enumerate}
\item $\rho(\Gamma)$ cannot be Zariski dense, 
\item if $d = 2n>2$, then the Zariski closure of $\rho(\Gamma)$ cannot be conjugate to $\PSp(2n,\Rb)$, 
\item if $d = 2n+1 > 3$ then the Zariski closure of $\rho(\Gamma)$ cannot be conjugate to $\PSO(n,n+1)$, and
\item if $d=7$, then the Zariski closure of $\rho(\Gamma)$ cannot be conjugate to the standard realization of $G_2$ in $\PSL_{7}(\Rb)$.
\end{enumerate}
See Section~\ref{sec:eigenvalues} in the appendix for details.

Guichard has announced that these are the only possibilities for the Zariski closure of $\rho(\Gamma)$ when $\rho$ is Hitchin but not Fuchsian (that is conjugate to a representation of the form $\tau_d \circ \rho_0$), see for instance~\cite[Section 11.3]{BCLS2015}.
\end{remark}

\subsection{Convex cocompactness in the work of Danciger, Gu{\'e}ritaud, and Kassel}\label{subsec:DGK} 

After I finished writing this paper, Danciger, Gu{\'e}ritaud, and Kassel informed me of their preprint~\cite{DGK2017b} which has some overlapping results with this paper. They consider a class of subgroups of $\PGL_d(\Rb)$ which they call \emph{strongly convex cocompact} which (using the terminology of this paper) are discrete subgroups $\Gamma \leq \PGL_d(\Rb)$ which act convex cocompactly on a properly convex domain which is strictly convex and has $C^1$ boundary. This notion appears to be first studied in work of Crampon and Marquis~\cite{CM2014}. Danciger, Gu{\'e}ritaud, and Kassel also show (stated with different terminology) that if $\Lambda \leq \Aut(\Omega)$ is a regular convex cocompact subgroup (as in Definition~\ref{def:RCC}), then it is actually a strongly convex cocompact subgroup of $\PGL_d(\Rb)$, that is there exists a possibly different properly convex domain $\Omega^\prime$ where $\Lambda \leq \Aut(\Omega^\prime)$ is a convex cocompact subgroup and $\Omega^\prime$ is a strictly convex domain with $C^1$ boundary (see Theorem 1.15 in~\cite{DGK2017b}). Danciger, Gu{\'e}ritaud, and Kassel also study a notion of convex cocompact actions on general properly convex domains (see Definition 1.11 in~\cite{DGK2017b}) that is different than the one we consider in Definition~\ref{defn:CC} above.

The main overlap in the two papers is in Theorems~\ref{thm:RCC_anosov_intro}, ~\ref{thm:proper_action_implies}, and Corollary~\ref{cor:H_acts_CC} above and Theorems 1.4, 1.15 and Proposition 1.7 in~\cite{DGK2017a}.

\subsection*{Acknowledgements} 

I would like to thank Thomas Barthelm\'e and Ludovic Marquis for many helpful conversations. In particular, we jointly observed the fact that an argument due to G. Liu could be used to prove Proposition~\ref{prop:liu} during the course of writing our joint paper \emph{Entropy rigidity of Hilbert and Riemannian metrics}~\cite{BMZ2015}.  

I would also like to thank the referees for their careful reading of this paper and their many helpful comments and corrections. 

This material is based upon work supported by the National Science Foundation under grants DMS-1400919 and DMS-1760233.

\section{Preliminaries}

In this section we recall some facts that we will use in the arguments that follow.

\subsection{Some notations} \ \begin{enumerate}
\item If $M \subset \Pb(\Rb^d)$ is a $C^1$ $k$-dimensional submanifold of $\Pb(\Rb^{d})$ and $m \in M$ we will let $T_m M \subset \Pb(\Rb^d)$ be the $k$-dimensional projective subspace of $\Pb(\Rb^d)$ which is tangent to $M$ at $m$. 
\item If $V \subset \Rb^d$ is a linear subspace, we will let $\Pb(V) \subset \Pb(\Rb^d)$ denote its projectivization. In most other cases, we will use $[o]$ to denote the projective equivalence class of an object $o$, for instance: 
\begin{enumerate}
\item if $v \in \Rb^{d} \setminus \{0\}$, then $[v]$ denotes the image of $v$ in $\Pb(\Rb^{d})$, 
\item if $\phi \in \GL_{d}(\Rb)$, then $[\phi]$ denotes the image of $\phi$ in $\PGL_{d}(\Rb)$, and 
\item if $T \in \End(\Rb^{d}) \setminus\{0\}$, then $[T]$ denotes the image of $T$ in $\Pb(\End(\Rb^{d}))$. 
\end{enumerate}
\item A \emph{line segment} in $\Pb(\Rb^{d})$ is a connected subset of a projective line. Given two points $x,y \in \Pb(\Rb^{d})$ there is no canonical line segment with endpoints $x$ and $y$, but we will use the following convention: if $\Omega$ is a properly convex domain and $x,y \in \overline{\Omega}$, then (when the context is clear) we will let $[x,y]$ denote the closed line segment joining $x$ to $y$ which is contained in $\overline{\Omega}$. In this case, we will also let $(x,y)=[x,y]\setminus\{x,y\}$, $[x,y)=[x,y]\setminus\{y\}$, and $(x,y]=[x,y]\setminus\{x\}$.
\end{enumerate}

\subsection{Gromov hyperbolicity}

Suppose $(X,d)$ is a metric space. If $I \subset \Rb$ is an interval, a curve $\sigma: I \rightarrow X$ is a \emph{geodesic} if 
\begin{align*}
d(\sigma(t_1),\sigma(t_2)) = \abs{t_1-t_2}
\end{align*}
for all $t_1, t_2 \in I$.  A \emph{geodesic triangle} in a metric space is a choice of three points in $X$ and geodesic segments  connecting these points. A geodesic triangle is said to be \emph{$\delta$-thin} if any point on any of the sides of the triangle is within distance $\delta$ of the other two sides. 

\begin{definition}
A proper geodesic metric space $(X,d)$ is called \emph{$\delta$-hyperbolic} if every geodesic triangle is $\delta$-thin. If $(X,d)$ is $\delta$-hyperbolic for some $\delta\geq0$ then $(X,d)$ is called \emph{Gromov hyperbolic}.
\end{definition}

We will use the following (probably well known) characterization of Gromov hyperbolicity.

\begin{proposition}\label{prop:GH_suff} Suppose $(X,d)$ is a proper geodesic metric space, $\delta > 0$, and there exists a map 
\begin{align*}
(x,y) \in X \times X \rightarrow \sigma_{x,y}  \in C([0,d(x,y)],X)
\end{align*}
where $\sigma_{x,y}$ is a geodesic segment joining $x$ to $y$. If for every $x,y,z \in X$ distinct, the geodesic triangle formed by $\sigma_{x,y}, \sigma_{y,z}, \sigma_{z,x}$ is $\delta$-thin, then $(X,d)$ is Gromov hyperbolic. 
\end{proposition}

We begin the proof with a definition and a lemma. Define the Gromov product of $x,y \in X$ with respect to $o\in X$ to be 
\begin{align*}
(x | y)_o := \frac{1}{2} \left( d(x,o)+d(o,y) - d(x,y) \right).
\end{align*}

\begin{lemma} Suppose $(X,d)$ is a metric space, $x,y,o \in X$, and $\sigma:[0,T] \rightarrow X$ is a geodesic with $\sigma(0)=x$ and $\sigma(T)=y$. Then 
\begin{align*}
(x|y)_o \leq d(o,\sigma) : = \inf\{ d(o,\sigma(t)) : t \in [0,T] \}.
\end{align*}
\end{lemma}

\begin{proof} For $t \in [0,T]$,
\begin{align*}
d(x,y) = d(x, \sigma(t)) + d(\sigma(t), y)
\end{align*}
and so the triangle inequality implies that: 
\begin{align*}
2(x|y)_o & = d(x,o)+d(o,y)-d(x,y)  \leq 2d(o, \sigma(t)). \qedhere
\end{align*}

\end{proof}

\begin{proof}[Proof of Proposition~\ref{prop:GH_suff}] We start by proving the following claim: \\

\textbf{Claim:} If $x,y,o \in X$ and $t \leq (x|y)_o - \delta$, then
\begin{align*}
d(\sigma_{ox}(t), \sigma_{oy}(t)) \leq 2 \delta.
\end{align*}

It is enough to consider the case when $t < (x|y)_o -\delta$. In this case 
\begin{align*}
d(\sigma_{ox}(t), \sigma_{xy}) \geq d(o, \sigma_{xy})-d(\sigma_{ox}(t), o) \geq (x|y)_o - t > \delta.
\end{align*}
So by the thin triangle condition, there exists $s$ such that $d(\sigma_{ox}(t), \sigma_{oy}(s)) \leq \delta$. Then 
\begin{align*}
\delta \geq d(\sigma_{ox}(t), \sigma_{oy}(s)) \geq \abs{d(\sigma_{ox}(t), o) - d(o,\sigma_{oy}(s))} = \abs{t-s}.
\end{align*}
So 
\begin{align*}
d(\sigma_{ox}(t), \sigma_{oy}(t)) \leq d(\sigma_{ox}(t), \sigma_{oy}(s))+d(\sigma_{oy}(s), \sigma_{oy}(t)) \leq 2 \delta
\end{align*}
and the claim is established. 

By Proposition 1.22 in Chapter III.H in~\cite{BH1999}, $(X,d)$ is Gromov hyperbolic if and only if there exists some $\delta_0 >0$ such that
\begin{align*}
(x|y)_o \geq \min\{ (x|z)_o, (y|z)_o\} - \delta_0
\end{align*}
 for all $o,x,y,z \in X$.
 
Fix $o,x,y,z \in X$. We claim that
\begin{align*}
(x|y)_o \geq \min\{ (x|z)_o, (y|z)_o\} - 3\delta.
\end{align*}
Let $m =  \min\{ (x|z)_o, (y|z)_o\}$. Since $(x|y)_o \geq 0$, the inequality is trivial when $m \leq \delta$. So we can assume $m > \delta$. Then the triangle inequality implies that
\begin{align*}
\min\{ d(x,o), d(y,o), d(z,o)\} \geq m > \delta. 
\end{align*}
Then let $x^\prime = \sigma_{ox}(m-\delta)$, $y^\prime = \sigma_{oy}(m-\delta)$, and $z^\prime = \sigma_{oz}(m-\delta)$. Then by the claim
\begin{align*}
d(x^\prime, y^\prime) \leq d(x^\prime, z^\prime)+d(z^\prime, y^\prime) \leq 4 \delta.
\end{align*}
Then 
\begin{align*}
2(x|y)_o & = d(x,o)+d(o,y)-d(x,y) = d(x,x^\prime)+d(x^\prime,o)+d(o,y^\prime)+d(y^\prime,y)-d(x,y)\\
& \geq d(x^\prime,o)+d(o,y^\prime) - d(x^\prime, y^\prime)  \geq m-\delta + m-\delta - 4\delta = 2m - 6\delta.
\end{align*}
So 
\begin{equation*}
(x|y)_o \geq \min\{ (x|z)_o, (y|z)_o\} - 3\delta. \qedhere
\end{equation*}
\end{proof}

By combining several deep theorems from geometric group theory we can deduce the following.

\begin{theorem}\label{thm:connected}
Suppose $\Gamma$ is a non-elementary word hyperbolic group which does not split over a finite group and is not commensurable to the fundamental group of a closed hyperbolic surface. Then 
\begin{enumerate}
\item $\partial \Gamma$ is connected, 
\item $\partial \Gamma \setminus \{x\}$ is connected for every $x \in \partial \Gamma$, and
\item there exist $u,w \in \partial \Gamma$ distinct such that $\partial \Gamma \setminus \{u,v\}$ is connected.
\end{enumerate}
\end{theorem}

The argument below comes from the proof of Theorem 3.1 in~\cite{P2005}. 

\begin{proof} By work of Stallings, $\partial \Gamma$ is disconnected if and only if $\Gamma$ splits over a finite group ~\cite{S1971, S1968}.
So $\partial \Gamma$ must be connected. Then a theorem of Swarup~\cite{S1996} implies that $\partial \Gamma \setminus \{x\}$ is connected for every $x \in \partial \Gamma$.

Now suppose for a contradiction that $\partial \Gamma \setminus \{u,v\}$ is disconnected for every $u,v \in \partial \Gamma$ distinct. Then $ \partial \Gamma$ is homeomorphic to the circle by~\cite[Chapter IV, Theorem 12.1]{N1992}. But then by work of Gabai~\cite{G1992} and Tukia~\cite{T1988}, $\Gamma$ is commensurable to the fundamental group of a closed hyperbolic surface.
\end{proof}

\subsection{Properly convex domains} 

In this subsection we review some basic definitions involving convexity in real projective space. 

\begin{definition} \ \begin{enumerate}
\item A set $\Omega \subset \Pb(\Rb^d)$ is called a \emph{domain} if $\Omega$ is open and connected
\item A set $\Omega \subset \Pb(\Rb^{d})$ is called \emph{convex} if  $L \cap \Omega$ is connected and $L \cap \Omega \neq L$ for every projective line $L \subset \Pb(\Rb^{d})$. 
\item A convex domain $\Omega \subset \Pb(\Rb^{d})$ is called a \emph{properly convex domain} if $\overline{L \cap \Omega} \neq L$ for every projective line $L \subset \Pb(\Rb^{d})$.
\end{enumerate}
\end{definition}

When $\Omega \subset \Pb(\Rb^d)$ is a properly convex domain, there exists an affine chart $\mathbb{A} \subset \Pb(\Rb^d)$ which contains $\Omega$ as a bounded convex domain (see for instance~\cite[Chapter 1]{APS2004}). 

\begin{definition}
Given a properly convex domain $\Omega \subset \Pb(\Rb^d)$, a hyperplane $H \subset \Pb(\Rb^d)$ is a \emph{supporting hyperplane of $\Omega$ at $x \in \partial \Omega$} if $x \in H$ and $H \cap \Omega = \emptyset$. 
\end{definition}

One of the most important properties of properly convex domains is that every boundary point is contained in at least one supporting hyperplane (which follows from the supporting hyperplane characterization of convexity in Euclidean space). 

\begin{definition} Suppose that $\Omega \subset \Pb(\Rb^d)$ is a properly convex domain. Then
\begin{enumerate}
\item a point $x \in \partial \Omega$ is a \emph{$C^1$ point of $\Omega$} if $x$ is contained in a unique supporting hyperplane of $\Omega$. In this case, we let $T_x \partial \Omega$ denote this unique supporting hyperplane. 
\item a point $x \in \partial \Omega$ is an \emph{extreme point of $\Omega$} if there does not exist a line segment $(p,q)$ in $\partial \Omega$ with $x \in (p,q)$. 
\end{enumerate}
\end{definition}

It is straightforward to show that $x \in \partial \Omega$ is a $C^1$ point of $\Omega$ (in the sense above) if and only if $\partial\Omega$ is locally the graph of a function which is differentiable at $x$. Moreover, in this case if $x_n \in \partial \Omega$ is a sequence converging to $x$ and $H_n$ is a supporting hyperplane at $x_n$, then $\lim_{n \rightarrow \infty} H_n=T_x \partial \Omega$.

Given a properly convex domain $\Omega \subset \Pb(\Rb^{d})$ the \emph{dual set} is defined to be:
\begin{align*}
\Omega^* = \{ f \in \Pb(\Rb^{d*}): f(x) \neq 0 \text{ for all } x \in \overline{\Omega} \}.
\end{align*}
The set $\Omega^*$ is a properly convex domain in $\Pb(\Rb^{d*})$ and the two sets have the following relation.

\begin{observation}
If $f \in \partial \Omega^*$, then $\Pb(\ker f)$ is a supporting hyperplane of $\Omega$. 
\end{observation}

\subsection{The Hilbert metric}

For distinct points $x,y \in \Pb(\Rb^{d})$ let $\overline{xy}$ be the projective line containing them. Suppose $\Omega \subset \Pb(\Rb^{d})$ is a properly convex domain. If $x,y \in \Omega$ are distinct let $a,b$ be the two points in $\overline{xy} \cap \partial\Omega$ ordered $a, x, y, b$ along $\overline{xy}$. Then define \emph{the Hilbert distance between $x$ and $y$ to be}
\begin{align*}
d_{\Omega}(x,y) = \frac{1}{2}\log [a, x,y, b]
\end{align*}
 where 
 \begin{align*}
 [a,x,y,b] = \frac{\abs{x-b}\abs{y-a}}{\abs{x-a}\abs{y-b}}
 \end{align*}
 is the cross ratio. Using the invariance of the cross ratio under projective maps and the convexity of $\Omega$ it is possible to establish the following (see for instance~\cite[Section 28]{BK1953}). 
 
 \begin{proposition}\label{prop:hilbert_basic}
Suppose $\Omega \subset \Pb(\Rb^{d})$ is a properly convex domain. Then $d_{\Omega}$ is a complete $\Aut(\Omega)$-invariant metric on $\Omega$ which generates the standard topology on $\Omega$. Moreover, if $p,q \in \Omega$, then there exists a geodesic joining $p$ and $q$ whose image is the line segment $[p,q]$.
\end{proposition}

We will use the following observation (which follows immediately from the definition of $d_\Omega$). 

\begin{observation}[see Lemma 3 in~\cite{V1970}]\label{obs:dist_zero} Suppose $\Omega \subset \Pb(\Rb^{d})$ is a properly convex domain and $p_n, q_n \in \Omega$ are sequences. If $p_n \rightarrow p \in \overline{\Omega}$ and $d_\Omega(p_n, q_n) \rightarrow 0$, then $q_n \rightarrow p$. 
\end{observation}

We will also use the following estimate.

\begin{lemma}[see Lemma 8.3 in~\cite{Cra2009}] \label{lem:hilbert_1} Suppose $\Omega \subset \Pb(\Rb^{d})$ is a properly convex domain and $[a,b], [c,d] \subset \Omega$ are line segments. If $p \in [a,b]$, then
\begin{align*}
d_\Omega(p, [c,d]) \leq d_\Omega(a,c) + d_\Omega(b,d).
\end{align*}
\end{lemma}

We will also consider the Gromov product induced by the Hilbert metric: given a properly convex domain $\Omega \subset \Pb(\Rb^d)$ define the \emph{Gromov product of $p,q \in \Omega$ based at $o \in \Omega$} to be 
\begin{align*}
(p|q)_o^{\Omega} = \frac{1}{2} \left( d_\Omega(p,o)+d_\Omega(o,q)-d_\Omega(p,q) \right).
\end{align*}

Karlsson and Noskov established the following estimates.

\begin{lemma}\label{lem:GP}\cite[Theorem 5.2]{KN2002} Suppose $\Omega \subset \Pb(\Rb^{d})$ is a properly convex domain, $o\in \Omega$, $p_n \in \Omega$ is a sequence with $p_n \rightarrow p \in \partial \Omega$, and $q_m \in \Omega$ is a sequence with $q_m \rightarrow q \in \partial \Omega$.
\begin{enumerate}
\item If $p= q$, then $\lim_{n,m \rightarrow \infty} (p_n|q_m)_o^{\Omega} = \infty$. 
\item If $\limsup_{n,m \rightarrow \infty} (p_n|q_m)_o^{\Omega} = \infty$, then $[p,q] \subset \partial \Omega$. 
\end{enumerate}
\end{lemma}

Since the proof is short we include it. 

\begin{proof} Both parts are consequences of the fact that every line segment in $\Omega$ can be parametrized to be a geodesic. 

First suppose that $p_n, q_m \in \Omega$ both converge to some $p \in \partial \Omega$. Let $\sigma_n : \Rb_{\geq 0} \rightarrow \Omega$ and $\overline{\sigma}_m:\Rb_{\geq 0} \rightarrow \Omega$ be the geodesic rays in $(\Omega, d_\Omega)$ whose images are line segments with $\sigma_n(0)=\overline{\sigma}_m(0)=o$ and $\sigma_n(T_n) = p_n$, $\overline{\sigma}_m(S_m) = q_m$ for some $T_n, S_m \in \Rb_{\geq0}$. Then if $t \leq \min\{T_n, S_m\}$ we have
\begin{align*}
2(p_n|q_m)_o^{\Omega} 
&= 2t + d_\Omega(p_n,\sigma_n(t))+d_\Omega(\overline{\sigma}_m(t), q_m) - d_{\Omega}(p_n,q_m) \\
&\geq 2t-d_\Omega(\sigma_n(t), \overline{\sigma}_m(t)).
\end{align*}
Since 
\begin{align*}
\limsup_{n,m \rightarrow \infty} d_\Omega(\sigma_n(t), \overline{\sigma}_m(t)) = 0
\end{align*}
for any fixed $t > 0$, we then see that 
\begin{align*}
\lim_{n,m \rightarrow \infty} (p_n|q_m)_o^{\Omega} = \infty.
\end{align*}

Next suppose that $p_n \rightarrow p \in \partial \Omega$, $q_m \rightarrow q \in \partial \Omega$, and 
\begin{align*}
\limsup_{n,m \rightarrow \infty} (p_n|q_m)_o^{\Omega} = \infty.
\end{align*}
By passing to subsequences we can suppose that 
\begin{align*}
\limsup_{n \rightarrow \infty} (p_n|q_n)_o^{\Omega} = \infty.
\end{align*}
Next let $\sigma_n : [0,R_n] \rightarrow \Omega$ be a geodesic joining $p_n$ to $q_n$ whose image is a line segment.  Then for any $t \in [0,R_n]$ we have
\begin{align*}
2(p_n|q_n)_o^{\Omega} 
&=  d_\Omega(p_n,o)+d_\Omega(o,q_n)-d_\Omega(p_n,\sigma_n(t)) - d_\Omega(\sigma_n(t),q_n) \\
& \leq 2d_\Omega(o, \sigma_n(t))
\end{align*}
and so
\begin{align*}
(p_n|q_n)_o^{\Omega} \leq \inf_{t \in [0,R_n]} d_{\Omega}(o,\sigma_n(t)).
\end{align*}
Since the image of $\sigma_n$ is the line segment $[p_n,q_n]$ we then must have $[p,q] \subset \partial \Omega$. 
\end{proof}

\subsection{A fact about Anosov representations} In this subsection we describe the behavior of sequences of elements in a projective Anosov representation.

When a matrix is proximal, its iterates have the following behavior. 

\begin{observation}\label{obs:proximal_iterates} Suppose $g \in \PGL_d(\Rb)$ is proximal. Viewing $\PGL_d(\Rb)$ as a subset of $\Pb(\End(\Rb^d))$, the limit
\begin{align*}
T = \lim_{n \rightarrow \infty} g^n
\end{align*}
exists in $\Pb(\End(\Rb^d))$. Moreover, the image of $T$ is the eigenline of $g$ corresponding to the eigenvalue with maximal modulus.
\end{observation} 

\begin{proof} By changing coordinates we can assume that 
\begin{align*}
g = \begin{bmatrix} \lambda & 0 \\ 0 & A \end{bmatrix}
\end{align*}
where $[1:0:\dots:0]$ is the eigenline of $g$ corresponding to the eigenvalue with maximal modulus and $A$ is a $(d-1)$-by-$(d-1)$ matrix. Then
\begin{align*}
g^n = \begin{bmatrix} 1 & 0 \\ 0 & \frac{1}{\lambda^n}A^n \end{bmatrix}
\end{align*}
and the observation immediately follows from Gelfand's formula (see Theorem~\ref{thm:Gelfand}). 
\end{proof}

Notice that if $g \in \PGL_d(\Rb)$ is proximal, then the representation $m \in \Zb \rightarrow g^m$ is projective Anosov. A well known analogue of the above observation holds for  general projective Anosov representations.

\begin{lemma}\label{lem:dynamics} Suppose that $\Gamma$ is a word hyperbolic group. Let $\rho: \Gamma \rightarrow \PGL_{d}(\Rb)$ be an irreducible projective Anosov representation with boundary maps $\xi: \partial \Gamma \rightarrow \Pb(\Rb^{d})$ and $ \eta : \partial \Gamma \rightarrow \Pb(\Rb^{d*})$. Assume $\gamma_n \in \Gamma$ is a sequence such that $\gamma_n \rightarrow x^+ \in \partial \Gamma$ and $\gamma_n^{-1} \rightarrow x^- \in \partial \Gamma$. Then viewing $\PGL_d(\Rb)$ as a subset of $\Pb(\End(\Rb^d))$, 
\begin{align*}
T=\lim_{n \rightarrow \infty} \rho(\gamma_n)
\end{align*}
where $\Imag(T) = \xi(x^+)$ and $\ker T = \ker \eta(x^-)$. In particular, 
\begin{align*}
\xi(x^+) = \lim_{n \rightarrow \infty} \rho(\gamma_n)v
\end{align*}
for all $v \in \Pb(\Rb^d) \setminus \Pb(\ker \eta(x^-))$ and the convergence is uniform on compact subsets of $\Pb(\Rb^d) \setminus \Pb(\ker \eta(x^-))$. 
\end{lemma}

Since the proof is short we include it. 

\begin{proof} We first consider the case in which $\# \partial \Gamma = 2$. Then since $\rho$ is irreducible and $\rho$ preserves $\xi(\partial \Gamma)$ we see that $d=2$. Then the lemma follows easily from the dynamics of $2$-by-$2$ matrices acting on $\Pb(\Rb^2)$.

So suppose that $\# \partial \Gamma >2$. Then $\#\partial\Gamma = \infty$ and $\partial \Gamma$ is a perfect space. Since $\Pb( \End(\Rb^{d}))$ is compact it is enough to show that every convergent subsequence of $\rho(\gamma_n)$ converges to $T$. So suppose that $\rho(\gamma_n) \rightarrow S$ in $\Pb( \End(\Rb^{d}))$. 

We first claim that $\Imag(S) = \xi(x^+)$. Since $\rho: \Gamma \rightarrow \PGL_{d}(\Rb)$ is irreducible, there exists $x_1, \dots, x_{d} \in \partial \Gamma$ such that  $\xi(x_1), \dots, \xi(x_{d})$ spans $\Rb^{d}$. Since $\partial \Gamma$ is a perfect space, we can perturb the $x_i$ (if necessary) and assume that
\begin{align*}
x^- \notin  \{ x_1, \dots, x_d\}.
\end{align*}

Then $\gamma_n x_i \rightarrow x^+$ and since $\xi$ is $\rho$-equivariant, we then see that $\rho(\gamma_n)\xi(x_i) \rightarrow \xi(x^+)$. Since $\xi(x_1), \dots, \xi(x_{d})$ spans $\Rb^{d}$ this implies that 
\begin{align*}
\Imag(S) = \xi(x^+).
\end{align*}

Next view $^t\rho(\gamma_n)$ as an element of $\Pb(\End(\Rb^{d*}))$. Then $^t\rho(\gamma_n)$ converges to $^tS$ in $\Pb(\End(\Rb^{d*}))$. Since 
\begin{align*}
^t\rho(\gamma_n) \eta(x) = \eta( \gamma_n^{-1} x),
\end{align*}
repeating the argument above shows that 
\begin{align*}
\Imag(^tS) = \eta(x^-).
\end{align*}
But this implies that 
\begin{equation*}
\ker S = \ker \eta(x^-). \qedhere
\end{equation*}
\end{proof}

This lemma has the following corollary.

\begin{corollary}\label{cor:dynamics} Suppose that $\Gamma$ is a word hyperbolic group. Let $\rho: \Gamma \rightarrow \PGL_{d}(\Rb)$ be an irreducible projective Anosov representation with boundary maps $\xi: \partial \Gamma \rightarrow \Pb(\Rb^{d})$ and $ \eta : \partial \Gamma \rightarrow \Pb(\Rb^{d*})$. If $\rho(\Gamma)$ preserves a properly convex domain $\Omega \subset \Pb(\Rb^{d})$, then 
\begin{align*}
\xi(\partial \Gamma) \subset \partial \Omega \text{ and } \eta(\partial \Gamma) \subset \partial \Omega^*.
\end{align*}
\end{corollary}

\begin{proof} Fix some $x \in \partial \Gamma$. Then there exists $\gamma_n \in \Gamma$ such that $\gamma_n \rightarrow x$. Now suppose that $\gamma_n^{-1} \rightarrow x^-$. Since $\Omega$ is open, there exists some $v \in \Omega \setminus \Pb(\ker \eta(x^-))$. Then 
\begin{align*}
\xi(x) = \lim_{n \rightarrow \infty} \rho(\gamma_n) v \in \overline{\Omega}.
\end{align*}
On the other hand, $\rho$ has finite kernel (by definition) and discrete image by Theorem 5.3 in~\cite{GW2012}. Further, since $\Aut(\Omega)$ preserves the Hilbert metric on $\Omega$, $\Aut(\Omega)$ acts properly on $\Omega$. So we must have $\xi(x) \in \partial \Omega$. Since $x \in \partial \Gamma$ was an arbitrary point we then have $\xi(\partial \Gamma) \subset \partial \Omega$. 

Repeating the same argument on $\Omega^*$ shows that $\eta(\partial \Gamma) \subset \partial \Omega^*$.
\end{proof}

\section{Constructing a convex cocompact action}\label{sec:construction}

In this section we establish Theorems ~\ref{thm:condC} and~\ref{thm:proper_action_implies} from the introduction. The argument has two parts: first we show that we can lift the boundary maps $\xi, \eta$ to maps into $\Rb^d, \Rb^{d*}$ and then we will show that whenever we can lift $\xi,\eta$ we obtain a regular convex cocompact action. 

\subsection{Lifting the maps}

Before stating the theorem we need some notation: fix a norm $\norm{\cdot}$ on $\Rb^d$, this induces a norm on $\Rb^{d*}$ by 
\begin{align*}
\norm{f} = \max\{ \abs{f(v)} : \norm{v}=1\}.
\end{align*}
Then let $S^{d-1} \subset \Rb^d$ and $S^{(d-1)*} \subset \Rb^{d*}$ be the unit spheres relative to these norms. In the statement and proof of the next theorem we will use the standard action of $\GL_{d}(\Rb)$ on $S^{d-1}$ and $S^{(d-1)*}$ given by 
\begin{align*}
g \cdot v = \frac{gv}{\norm{gv}} \text{ and } g \cdot f = \frac{ f \circ {g^{-1}}}{\norm{f \circ {g^{-1}}}}.
\end{align*}
Finally let
\begin{align*}
\SL_{d}^{\pm}(\Rb) = \{ g \in \GL_d(\Rb) : \det g = \pm 1 \}.
\end{align*}

\begin{theorem}\label{thm:bd_maps_lift}
Suppose $\Gamma$ is a word hyperbolic group. Let $\rho: \Gamma \rightarrow \PGL_{d}(\Rb)$ be an irreducible projective Anosov representation with boundary maps $\xi: \partial \Gamma \rightarrow \Pb(\Rb^{d})$ and $\eta:\partial \Gamma \rightarrow \Pb(\Rb^{d*})$. 

If one of the following conditions hold:
\begin{enumerate}
\item there exists a properly convex domain $\Omega_0 \subset \Pb(\Rb^{d})$ such that $\rho(\Gamma) \leq \Aut(\Omega_0)$ or
\item $\Gamma$ is a non-elementary word hyperbolic group which is not commensurable to a non-trivial free product or the fundamental group of a closed hyperbolic surface, 
\end{enumerate}
then there exist lifts $\wt{\rho}: \Gamma \rightarrow \SL^{\pm}_{d}(\Rb)$, $\wt{\xi}: \partial \Gamma \rightarrow S^{d-1}, \wt{\eta}: \partial \Gamma \rightarrow S^{(d-1)*}$ of $\rho, \xi, \eta$ respectively such that $\wt{\xi}$ and $\wt{\eta}$ are continuous, $\wt{\rho}$-equivariant, and 
\begin{align*}
\wt{\eta}(y) \Big( \wt{\xi}(x) \Big) > 0
\end{align*}
for all $x,y \in \partial \Gamma$ distinct. 
\end{theorem}

\begin{proof}[Proof of Theorem~\ref{thm:bd_maps_lift}] We will consider each case separately.\newline

\textbf{Case 1:} Suppose that there exists a properly convex domain $\Omega_0 \subset \Pb(\Rb^{d})$ such that $\rho(\Gamma) \leq \Aut(\Omega_0)$. \newline

Let $\pi: \Rb^{d} \setminus \{0\} \rightarrow \Pb(\Rb^{d})$ be the natural projection. Since $\Omega_0$ is properly convex, $\pi^{-1}(\Omega_0)$ has two connected components $\Cc_1$ and $\Cc_2$. Moreover, $\Cc_1$ and $\Cc_2$ are properly convex cones and $\Cc_1 = -\Cc_2$. 

By Corollary~\ref{cor:dynamics}, we see that $\xi(\partial \Gamma) \subset \partial \Omega_0$ and $\eta(\partial \Gamma) \subset \partial \Omega_0^*$. Now for $x \in \partial \Gamma$ let $\wt{\xi}(x) \in S^{d-1}$ be the unique representative of $\xi(x)$ such that $\wt{\xi}(x) \in \overline{\Cc_1}$ and let $\wt{\eta}(x) \in S^{(d-1)*}$ be the unique representative of $\eta(x)$ such that 
\begin{align*}
\wt{\eta}(x)(v) > 0
\end{align*}
for all $v \in \Cc_1$. Then by construction,   
\begin{align*}
\wt{\eta}(x) \left(\wt{\xi}(y) \right) \geq 0
\end{align*}
with equality if and only if $x=y$. Moreover, uniqueness implies that $\wt{\xi}$ and $\wt{\eta}$ are continuous. 

Now for $\gamma \in \Gamma$ let $\wt{\rho}(\gamma) \in \SL_{d}^{\pm}(\Rb)$ be the unique lift that preserves $\Cc_1$. Then $\wt{\rho} : \Gamma \rightarrow \SL_{d}^{\pm}(\Rb)$ is a homomorphism and $\wt{\xi}$ and $\wt{\eta}$ are $\wt{\rho}$-equivariant. \newline

\noindent \textbf{Case 2:} Suppose that $\Gamma$ is a non-elementary word hyperbolic group which is not commensurable to a non-trivial free product or a fundamental group of a closed hyperbolic surface. \newline

We will reduce to Case 1 by constructing a properly convex domain $\Omega_0 \subset \Pb(\Rb^{d})$ such that $\rho(\Gamma) \leq \Aut(\Omega_0)$.

Let $\Lambda = \rho(\Gamma)$. Then by Selberg's lemma $\Lambda$ has a torsion-free finite index subgroup $\Lambda_0$. Moreover, $\Lambda_0$ is commensurable to $\Gamma$ and $\partial \Lambda_0$ is homeomorphic to $\partial \Gamma$. Since $\Lambda_0$ is torsion-free, the condition on $\Gamma$ implies that $\Lambda_0$ does not split over a finite group and is not commensurable to the fundamental group of a closed hyperbolic surface. Hence by Theorem~\ref{thm:connected}, we see that 
\begin{enumerate}
\item $\partial \Gamma$ is connected, 
\item $\partial \Gamma \setminus \{x\}$ is connected for every $x \in \partial \Gamma$, and
\item there exist $u,w \in \partial \Gamma$ distinct such that $\partial \Gamma \setminus \{u,w\}$ is connected.
\end{enumerate}

The space $\Pb(\Rb^{d}) \setminus (\Pb( \ker \eta(u)) \cup \Pb(\ker \eta(w) ))$ has two connected components which we denote by $A^+$ and $A^-$. Since $\xi(\partial \Gamma \setminus\{u,w\})$ is connected, by relabelling we can assume that $\xi(\partial \Gamma \setminus \{u,w\}) \subset A^+$. Then $\xi(\partial \Gamma) \subset \overline{A^+}$. 

Next define $C : = \cap_{\gamma \in \Gamma} \rho(\gamma)\overline{A^+}$. By construction $C$ is closed, $\xi(\partial \Gamma) \subset C$, and $\rho(\gamma) C = C$ for every $\gamma \in \Gamma$. Let $C_0$ denote the connected component of $C$ which contains $\xi(\partial \Gamma)$ and let $\Omega_0$ denote the interior of $C_0$. 

We claim that $\Omega_0$ is a properly convex domain and $\rho(\Gamma) \leq \Aut(\Omega_0)$. By construction, $\rho(\gamma) \Omega_0 = \Omega_0$ for every $\gamma \in \Gamma$ and so it is enough to show that $\Omega_0$ is a properly convex domain. To accomplish this we recall the following terminology: a subset $E \subsetneq \Pb(\Rb^d)$ is called \emph{linearly convex} if for every $x \in \Pb(\Rb^d) \setminus E$ there exists a hyperplane $H \subset \Pb(\Rb^d)$ such that $H \cap \Omega = \emptyset$. We also recall the following basic properties of these sets:
\begin{enumerate}  
 \item\label{enum:lin_conx_1} every convex set is linearly convex,
 \item\label{enum:lin_conx_2} every connected component of a linearly convex set is convex, 
\item\label{enum:lin_conx_3} the intersection of a collection of linearly convex sets is linearly convex, and 
\item\label{enum:lin_conx_4} if $E \subset \Pb(\Rb^d)$ is linearly convex and $g \in \PGL_d(\Rb)$, then $gE$ is linearly convex. 
\end{enumerate}
Proofs of Properties~\ref{enum:lin_conx_1} and~\ref{enum:lin_conx_2} can be found in~\cite[Chapter 1]{APS2004}. Properties~\ref{enum:lin_conx_3} and~\ref{enum:lin_conx_4} are direct consequences of the definition. Since $A^+$ is projectively equivalent to $ \{ [1:x_1:\dots:x_{d-1}] : x_1 > 0\}$, we see that $\overline{A^+}$ is linearly convex. Thus by Properties~\ref{enum:lin_conx_2}, \ref{enum:lin_conx_3}, and \ref{enum:lin_conx_4}, $C_0$ is convex. 

Since $\rho$ is irreducible, $\{ \xi(x) : x \in \partial \Gamma\}$ spans $\Rb^d$. Since $\xi(\partial \Gamma) \subset C_0$, this implies that $C_0$ has non-empty interior. So $\Omega_0$ is a non-empty convex domain. Since $\Omega_0 \subset A^+$, we see that $\Omega_0 \cap \Pb(\ker \eta(u)) = \emptyset$. Since $\Gamma \cdot u \subset \partial \Gamma$ is dense, $\eta$ is continuous, and $\rho(\gamma) \Omega_0 = \Omega_0$ for every $\gamma \in \Gamma$, we then have
\begin{align*}
\Omega_0 \cap\Pb( \ker \eta(x)) = \emptyset 
\end{align*}
for all $x \in \partial \Gamma$. Since $\rho$ is irreducible, $\eta(\partial \Gamma)$ spans $\Rb^{d*}$ and so $\Omega_0$ must be properly convex. 

\end{proof}

\begin{remark} It is easy to construct examples of ``half spaces'' $E_1, E_2 \subset \Pb(\Rb^d)$ such that $E_1 \cap E_2$ is disconnected (and hence not convex). For instance, let $E_1$ be the connected component of $\Pb(\Rb^d) \setminus ( \{x_1 = 0\} \cup \{x_2=0\})$ which contains $[1:1:0:\dots:0]$. And let $E_2$ be the connected component of \begin{align*}
   \Pb(\Rb^d)\setminus  ( \{ x_1-x_2=0\} \cup \{ 2x_1 - x_2 =0\})                                                                                                                                                                                                                                                                                                                                                                                                                                                                                                                                                                                                                                                                                                    \end{align*}
which contains $[1:3:0:\dots:0]$. Then 
\begin{align*}
E_1 \cap E_2 = \{ [1:x_2 : \dots :x_d] : x_2 \in (0,1) \cup (2,\infty) \}.
\end{align*}
Examples like these are why we consider linearly convex sets in the proof of Theorem~\ref{thm:bd_maps_lift}. 
\end{remark}

\subsection{Showing the action is convex cocompact} 

\begin{theorem}\label{thm:convex_exist} Suppose $\Gamma$ is a word hyperbolic group. Let $\rho: \Gamma \rightarrow \PGL_{d}(\Rb)$ be an irreducible projective Anosov representation with boundary maps $\xi: \partial \Gamma \rightarrow \Pb(\Rb^{d})$ and $\eta:\partial \Gamma \rightarrow \Pb(\Rb^{d*})$. 

If there exist lifts $\wt{\rho}: \Gamma \rightarrow \SL^{\pm}_{d}(\Rb)$, $\wt{\xi}: \partial \Gamma \rightarrow S^{d-1}, \wt{\eta}: \partial \Gamma \rightarrow S^{(d-1)*}$ of $\rho, \xi, \eta$ respectively such that $\wt{\xi}$ and $\wt{\eta}$ are continuous, $\wt{\rho}$-equivariant, and 
\begin{align*}
\wt{\eta}(y) \Big( \wt{\xi}(x) \Big) > 0 
\end{align*}
for all $x,y \in \partial \Gamma$ distinct, then there exists a properly convex domain $\Omega \subset \Pb(\Rb^{d})$ such that $\rho(\Gamma)$ is a regular convex cocompact subgroup of $\Aut(\Omega)$. 
\end{theorem}

For the rest of this subsection let $\Gamma$, $\rho$, $\xi$, $\eta$, $\wt{\rho}$, $\wt{\xi}$, and $\wt{\eta}$ satisfy the hypothesis of Theorem~\ref{thm:convex_exist}. 

Define
\begin{align*}
\Omega := \left\{ [v] \in \Pb(\Rb^{d}) : \wt{\eta}(x)(v) > 0 \text{ for all } x \in \partial \Gamma\right\}.
\end{align*}

\begin{lemma}\label{lem:CC1} With the notation above, $\Omega$ is a properly convex domain, $\rho(\Gamma) \leq \Aut(\Omega)$, and if
$N > 1$; $\lambda_1,\dots, \lambda_N > 0$;  and  $x_1, \dots, x_N \in \partial \Gamma$ are distinct, then
\begin{align*}
 \left[ \sum_{i=1}^N \lambda_i \wt{\xi}(x_i) \right] \in \Omega. 
\end{align*}
\end{lemma}

\begin{proof} If $N > 1$; $\lambda_1,\dots, \lambda_N > 0$;  and  $x_1, \dots, x_N \in \partial \Gamma$ are distinct, then
\begin{align*}
\wt{\eta}(y) \left( \sum_{i=1}^N \lambda_i \wt{\xi}(x_i) \right)>0
\end{align*}
for all $y \in \partial \Gamma$. So 
\begin{align*}
 \left[ \sum_{i=1}^N \lambda_i \wt{\xi}(x_i) \right] \in \Omega. 
\end{align*}
In particular, $\Omega$ is non-empty.

We now show that $\Omega$ is open. Suppose $p_0 \in \Omega$. Then there exists $v_0 \in \Rb^d$ such that $p_0=[v_0]$ and $\wt{\eta}(x)(v_0) > 0$ for all $x \in \partial \Gamma$. Since $\partial \Gamma$ is compact and $\wt{\eta}:\partial \Gamma \rightarrow S^{(d-1)*}$ is continuous, we have
\begin{align*}
0< r:=\inf_{x \in \partial \Gamma} \wt{\eta}(x)(v_0).
\end{align*}
So 
\begin{align*}
\{ [v] \in \Pb(\Rb^d) : \norm{v-v_0} < r\} \subset \Omega.
\end{align*}
Hence  $\Omega$ is open. 

By construction $\Omega$ is a convex domain and
\begin{align*}
\Omega \cap \Pb(\ker \eta(x)) = \emptyset 
\end{align*}
for all $x \in \partial \Gamma$. Since $\rho$ is irreducible, $\eta(\partial \Gamma)$ spans $\Rb^{d*}$ and so $\Omega$ must be properly convex. 

Finally, since 
\begin{align*}
\rho(\gamma)[v] = [\wt{\rho}(\gamma)(v)]
\end{align*}
when $v \in \Rb^d$ and $\gamma \in \Gamma$, we see that $\rho(\Gamma) \leq \Aut(\Omega)$.
 \end{proof}
 
 \begin{lemma}\label{lem:CC2} With the notation above, $\xi(\partial \Gamma) \subset \partial \Omega$ and $\eta(\partial \Gamma) \subset \partial \Omega^*$.
  \end{lemma}
  
  \begin{proof} This follows immediately from Corollary~\ref{cor:dynamics}, but here is a direct proof: by the definition of $\Omega$ we see that $\eta(\partial \Gamma) \subset \overline{\Omega^*}$. Moreover, if $x,y \in \partial \Gamma$ are distinct, then 
  \begin{align*}
\xi(x) = \lim_{\lambda \rightarrow \infty} \left[ \lambda \wt{\xi}(x) + \wt{\xi}(y) \right] \in \overline{\Omega}.
\end{align*}
So $\xi(\partial \Gamma) \subset \overline{\Omega}$. Then, since 
\begin{align*}
\eta(x)(\xi(x)) = 0
\end{align*}
for all  $x \in \partial \Gamma$ we see that $\xi(\partial \Gamma) \subset \partial \Omega$ and $\eta(\partial \Gamma) \subset \partial \Omega^*$.
\end{proof}

Next let $\Cc$ be the closed convex hull of $\xi(\partial \Gamma)$ in $\Omega$. 

\begin{proposition}\label{prop:CC_key_step} With the notation above, $\rho(\Gamma)$ acts cocompactly on $\Cc$. \end{proposition}

Proposition~\ref{prop:CC_key_step} follows from either a recent result of Kapovich and Leeb~\cite{KL2018} or a recent result of Kapovich, Leeb, and Porti~\cite{KLP2013}. In particular, the action of $\rho(\Gamma)$ on $\xi(\partial \Gamma)$ is a uniform convergence action and so $\rho(\Gamma)$ acts cocompactly on $\Cc$ by Theorem 1.9 in~\cite{KL2018}. Alternatively, one can use $\Cc$ to construct an invariant set in the space of flags of the form (line, hyperplane) and then apply Theorem 1.5 in~\cite{KLP2013} to see that $\rho(\Gamma)$ acts cocompactly on $\Cc$. 

We will provide a proof of Proposition~\ref{prop:CC_key_step} that only uses elementary properties of convex sets. This direct argument requires a few preliminary lemmas.

Given a set $A \subset \Omega$ and a point $p \in \Omega$ define
\begin{align*}
d_\Omega(p,A) := \inf_{a \in A} d_\Omega(p,a).
\end{align*}
Then given two sets $A,B \subset \Omega$ define the \emph{Hausdorff distance in $d_\Omega$ between $A$ and $B$} to be:
\begin{align*}
d_{\Omega}^{\Haus}(A,B) := \max\left\{ \sup_{a \in A} d_\Omega(a,B), \sup_{b \in B} d_\Omega(b,A) \right\}.
\end{align*}

Next fix a finite, symmetric generating set $S$ of $\Gamma$ and let $d_S$ be the induced word metric on $\Gamma$.

\begin{lemma}\label{lem:geod_shadowing} With the notation above, suppose that $p_0 \in \Omega$. Then there exists some $R >0$ with the following property: if $g_1, \dots, g_N \in \Gamma$ is a geodesic in $(\Gamma, d_S)$, then 
\begin{align*}
d_{\Omega}^{\Haus}\Big( \{ \rho(g_1)p_0, \dots, \rho(g_N)p_0\}, [\rho(g_1) p_0, \rho(g_N) p_0] \Big) \leq R. 
\end{align*}
\end{lemma}

\begin{proof} We first claim that there exists some $R_1 > 0$ with the following property: if $g_1, \dots, g_N \in \Gamma$ is a geodesic in $(\Gamma, d_S)$, then 
\begin{align*}
\max_{1 \leq i \leq N} d_\Omega \Big( \rho(g_i) p_0, \left[\rho(g_1) p_0, \rho(g_N) p_0\right] \Big) \leq R_1.
\end{align*}
Suppose not, then after possibly translating by elements in $\Gamma$ we can assume: for any $n > 0$ there exists a geodesic 
\begin{align*}
g_{-M_n}^{(n)}, g_{-M_n+1}^{(n)}, \dots, g_{N_n}^{(n)}
\end{align*}
in $(\Gamma,d_S)$ such that $g_0^{(n)}=\id$ and
\begin{align*}
d_\Omega \left( \rho\left(g_0^{(n)}\right)p_0, \left[\rho\left(g_{-M_n}^{(n)}\right) p_0, \rho\left(g_{N_n}^{(n)}\right) p_0\right] \right)=d_\Omega \left( p_0, \left[\rho\left(g_{-M_n}^{(n)}\right) p_0, \rho\left(g_{N_n}^{(n)}\right) p_0\right] \right) > n.
\end{align*}
Notice that we must have $M_n, N_n \rightarrow \infty$. Now by passing to a subsequence, we can suppose that the limits
\begin{align*}
g_i = \lim_{n \rightarrow \infty} g_i^{(n)}
\end{align*}
exists for all $i$. Then 
\begin{align*}
\dots, g_{-2}, g_{-1}, g_0=\id, g_1, g_2, \dots
\end{align*}
is a geodesic in $(\Gamma,d_S)$. So there exist $x^+,x^- \in \partial \Gamma$ distinct such that 
\begin{align*}
\lim_{i \rightarrow \pm \infty} g_i = x^{\pm}.
\end{align*}
By the standard geodesic ray definition of the topology on $\Gamma \cup \partial \Gamma$, we have
\begin{align*}
\lim_{n \rightarrow \infty} g_{N_n}^{(n)} = x^+
\end{align*}
and
\begin{align*}
\lim_{n \rightarrow \infty} g_{-M_n}^{(n)} = x^-.
\end{align*}

Now  $\Pb(\ker\eta(x^-)) \cap \Omega = \emptyset$ and so Lemma~\ref{lem:dynamics} implies that
\begin{align*}
\lim_{n \rightarrow \infty} \rho\left(g_{N_n}^{(n)}\right) p_0 = \xi(x^+).
\end{align*}
The same reasoning implies that
\begin{align*}
\lim_{n \rightarrow \infty} \rho\left(g_{-M_n}^{(n)}\right) p_0 = \xi(x^-).
\end{align*}
Since $x^+,x^- \in \partial \Gamma$ are distinct, Lemma~\ref{lem:CC1} implies that $(\xi(x^-), \xi(x^+)) \subset \Omega$ and so
 \begin{align*}
 d_\Omega \left( p_0, (\xi(x^-), \xi(x^+)) \right) < \infty.
\end{align*}
Then since 
\begin{align*}
\infty = \lim_{n \rightarrow \infty} d_\Omega \left( p_0, \left[\rho\left(g_{-M_n}^{(n)}\right) p_0, \rho\left(g_{N_n}^{(n)}\right) p_0\right] \right) =d_\Omega \left( p_0, (\xi(x^-), \xi(x^+)) \right) < \infty
\end{align*}
we have a contradiction. Hence, there exists some $R_1 > 0$ such that: if $g_1, \dots, g_N \in \Gamma$ is a geodesic in $(\Gamma, d_S)$, then 
\begin{align*}
\max_{1 \leq i \leq N} d_\Omega \left( \rho(g_i) p_0, [\rho(g_1) p_0, \rho(g_N) p_0] \right) \leq R_1.
\end{align*}

Now let 
\begin{align*}
C = \max\{ d_\Omega(p_0, \rho(g) p_0) : g \in S\}.
\end{align*}
We claim that: if $g_1, \dots, g_N \in \Gamma$ is a geodesic in $(\Gamma, d_S)$ and if $p \in [\rho(g_1) p_0, \rho(g_N) p_0]$, then 
\begin{align*}
 d_\Omega\left( p, \{ \rho(g_1)p_0, \dots, \rho(g_N)p_0\}\right) \leq 2R_1 + C/2.
\end{align*}
For each $1 \leq i \leq N$, let $p_i$ be a closest point to $\rho(g_i)p_0$ in $[\rho(g_1) p_0, \rho(g_N) p_0]$. Then 
\begin{align*}
d_\Omega(p_i, p_{i+1}) 
&\leq d_\Omega(p_i, \rho(g_i)p_0) + d_\Omega(\rho(g_i)p_0, \rho(g_{i+1})p_0) + d_\Omega(\rho(g_{i+1})p_0, p_{i+1}) \\
& \leq R_1 + C+R_1 = 2R_1+C.
\end{align*}
Since $p_1 = \rho(g_1) p_0$ and $p_N = \rho(g_N) p_0$ we see that: for any $p \in [\rho(g_1) p_0, \rho(g_N) p_0]$ 
\begin{align*}
\min_{1 \leq i \leq N} d_\Omega(p, p_i) \leq \frac{1}{2} ( 2R_1+C) = R_1 +C/2
\end{align*}
and so 
\begin{align*}
d_\Omega\left( p, \{ \rho(g_1)p_0, \dots, \rho(g_N)p_0\}\right) \leq 2R_1 + C/2.
\end{align*}
So $R = 2R_1 + C/2$ satisfies the conclusion of the lemma. 
\end{proof}

\begin{lemma} With the notation above, suppose that $p_0 \in \Omega$. For any $N \geq 2$ there exists $C_N > 0$ such that: if
\begin{align*}
p = \left[ \sum_{i=1}^N \lambda_i \wt{\xi}(x_i)\right]
\end{align*}
where $\lambda_1, \dots, \lambda_N > 0$ and $x_1, \dots, x_N \in \partial \Gamma$ are distinct, then 
\begin{align*}
d_\Omega(p, \rho(\Gamma) \cdot p_0) \leq C_N.
\end{align*}
\end{lemma}

\begin{proof} We induct on $N$. For the remainder of the proof let $R$ be the constant from Lemma~\ref{lem:geod_shadowing}.

For the $N=2$ case suppose that $x_1,x_2 \in \partial \Gamma$ are distinct. Then there exist sequences $g_n, h_n \in \Gamma$ such that $g_n \rightarrow x_1$ and $h_n \rightarrow x_2$. By Lemma~\ref{lem:dynamics} 
\begin{align*}
\rho(g_n)p_0 \rightarrow \xi(x_1) \text{ and } \rho(h_n)p_0 \rightarrow \xi(x_2).
\end{align*}
So if 
\begin{align*}
p =  \left[ \lambda_1\wt{\xi}(x_1)+\lambda_2\wt{\xi}(x_2)\right]
\end{align*}
for some $\lambda_1,\lambda_2 >0$, then there exists a sequence $p_n \in [\rho(g_n)p_0, \rho(h_n)p_0]$ such that $p_n \rightarrow p$.  Lemma~\ref{lem:geod_shadowing} implies that 
\begin{align*}
d_\Omega(p_n, \rho(\Gamma) \cdot p_0) \leq R
\end{align*}
and so 
\begin{align*}
d_\Omega(p, \rho(\Gamma) \cdot p_0) \leq R.
\end{align*}

Next suppose that $N > 2$ and consider 
\begin{align*}
p = \left[ \sum_{i=1}^N \lambda_i \wt{\xi}(x_i)\right]
\end{align*}
where $\lambda_1, \dots, \lambda_N > 0$ and $x_1, \dots, x_N \in \partial \Gamma$ are distinct. We claim that 
\begin{align*}
d_\Omega(p, \rho(\Gamma) \cdot p_0) \leq 2C_{\ceil{N/2}}+R.
\end{align*}

Let
\begin{align*}
p_1 = \left[ \sum_{1 \leq i < \ceil{N/2}} \lambda_i \wt{\xi}\left(x_i\right) + \frac{1}{2}\lambda_{\ceil{N/2}} \wt{\xi}\left( x_{\ceil{N/2}}\right)\right]
\end{align*}
and
\begin{align*}
p_2 = \left[\frac{1}{2}\lambda_{\ceil{N/2}} \wt{\xi}\left( x_{\ceil{N/2}}\right)+\sum_{\ceil{N/2} < i \leq N } \lambda_i \wt{\xi}(x_i)\right].
\end{align*}
Then, by induction there exist elements $g_1, g_2 \in \Gamma$ such that 
\begin{align*}
d_\Omega(p_i, \rho(g_i) \cdot p_0) \leq C_{\ceil{N/2}}.
\end{align*}
Now $p \in [p_1, p_2]$ and so by Lemma~\ref{lem:hilbert_1} there exists $q \in [\rho(g_1) \cdot p_0, \rho(g_2) \cdot p_0]$ such that 
\begin{align*}
d_\Omega(p,q) \leq 2C_{\ceil{N/2}}.
\end{align*}
Then Lemma~\ref{lem:geod_shadowing} implies that
\begin{align*}
d_\Omega(q, \rho(\Gamma) \cdot p_0) \leq R
\end{align*}
and hence
\begin{align*}
d_\Omega(p, \rho(\Gamma) \cdot p_0) \leq 2C_{\ceil{N/2}}+R.
\end{align*}
\end{proof}

\begin{proof}[Proof of Proposition~\ref{prop:CC_key_step}] By Carath{\'e}odory's convex hull theorem any $p \in \Cc$ can be written as 
\begin{align*}
p = \left[ \sum_{i=1}^{N} \lambda_i \wt{\xi}(x_i)\right]
\end{align*}
where $2 \leq N \leq d+1$; $\lambda_1, \dots, \lambda_{N} >0$; and $x_1, \dots, x_N \in \partial \Gamma$ are distinct. Thus by the previous lemma there exists some $M > 0$ such that 
\begin{align*}
\Cc= \cup_{g \in \Gamma} \rho(g) \left(\overline{B}_\Omega(p_0;M) \cap \Cc\right)
\end{align*}
where $\overline{B}_\Omega(p_0;M)$ is the closed metric ball of radius $M$ in $(\Omega, d_\Omega)$. 
\end{proof}

\begin{lemma}\label{lem:CC3} With the notation above, if $f \in \overline{\Omega^*}$, then there exists $1 \leq N \leq d+1$; $\lambda_1, \dots, \lambda_N > 0$; and $x_1, \dots, x_N \in \partial \Gamma$ distinct so that 
\begin{align*}
f = \left[ \sum_{i=1}^N \lambda_i \wt{\eta}(x_i) \right].
\end{align*}
\end{lemma}

\begin{proof}
By the definition of $\Omega$, the set $\overline{\Omega^*}$ is the image of 
\begin{align*}
\overline{{\rm ConvexHull}\left\{ \wt{\eta}(x) : x \in \partial \Gamma \right\}} \subset \Rb^{d*}
\end{align*}
in $\Pb(\Rb^{d*})$. Then since $\wt{\eta}: \partial \Gamma \rightarrow \Rb^{d*}$ is continuous, Carath{\'e}odory's convex hull theorem implies that $f$ can be written as 
\begin{align*}
f =  \left[ \sum_{i=1}^N \lambda_i \wt{\eta}(x_i) \right] 
\end{align*}
for some $1 \leq N \leq d+1$; $\lambda_1, \dots, \lambda_N > 0$; and $x_1, \dots, x_N \in \partial \Gamma$.
\end{proof}

\begin{lemma} With the notation above, 
\begin{align*}
\xi(\partial \Gamma) = \overline{\Cc} \cap \partial \Omega,
\end{align*}
every point in $\overline{\Cc} \cap \partial \Omega$ is a $C^1$ extreme point of $\Omega$, and
\begin{align*}
T_{\xi(x)} \partial \Omega =\Pb( \ker \eta(x))
\end{align*}
for all $x \in \partial \Gamma$. \end{lemma}

\begin{proof} Lemma~\ref{lem:CC1} and the definition of $\Cc$ imply that 
\begin{align*}
\xi(\partial \Gamma) = \overline{\Cc} \cap \partial \Omega.
\end{align*}

So suppose that $x \in \partial \Gamma$. We first show that $\xi(x)$ is a $C^1$ point of $\Omega$. Suppose that $H$ is a supporting hyperplane of $\Omega$ at $\xi(x)$. Then $H=\Pb(\ker f)$ for some $f \in \overline{\Omega^*}$. By Lemma~\ref{lem:CC3}
\begin{align*}
f = \left[ \sum_{i=1}^N \lambda_i \wt{\eta}(x_i) \right]
\end{align*}
for some $1 \leq N \leq d+1$; $\lambda_1, \dots, \lambda_N > 0$; and $x_1, \dots, x_N \in \partial \Gamma$ distinct. Since $f(\xi(x))=0$, we then have 
\begin{align*}
0=\sum_{i=1}^N \lambda_i \wt{\eta}(x_i)\left(\wt{\xi}(x)\right). 
\end{align*}
By hypothesis  
\begin{align*}
\wt{\eta}(y)\left(\wt{\xi}(z)\right) > 0
\end{align*}
when $y,z \in \partial \Gamma$ are distinct and so we must have $N=1$ and $x_1 = x$. Thus $f = \eta(x)$ and $H = \Pb(\ker \eta(x))$. Since $H$ was an arbitrary supporting hyperplane of $\Omega$ at $\xi(x)$ we see that $\xi(x)$ is a $C^1$ point of $\partial \Omega$ and 
\begin{align*}
T_{\xi(x)} \partial \Omega = \Pb(\ker \eta(x)).
\end{align*}

We next show that $\xi(x)$ is an extreme point of $\Omega$. This follows immediately from Lemma~\ref{lem:E} below, but we will provide a direct argument. Suppose for a contradiction that $\xi(x)$ is not an extreme point, then there exists $p^\prime, q^\prime \in \partial \Omega$ such that 
\begin{align*}
\xi(x) \in (p^\prime, q^\prime) \subset \partial \Omega.
\end{align*}
Fix a point $c_0 \in \Cc$ and consider a sequence of points $q_n$ along the line $[c_0, \xi(x))$ which converge to $\xi(x)$. Since $\rho(\Gamma)$ acts cocompactly on $\Cc$, there exist some $M_1 >0$ and elements $\gamma_n \in \Gamma$ such that 
\begin{align*}
d_\Omega(\rho(\gamma_n) c_0, q_n) \leq M_1.
\end{align*}
Next fix some $p \in (p^\prime, q^\prime) \subset \partial \Omega$ with $p \neq \xi(x)$. Then, by the definition of the Hilbert metric, we can find a sequence of points $p_n$ along the line $[c_0,p)$ such that 
\begin{align*}
M_2:=\sup_{n \geq 0} d_\Omega(p_n, q_n) < +\infty.
\end{align*}
Next let $k_n = \rho(\gamma_n)^{-1} q_n$ and $\ell_n = \rho(\gamma_n)^{-1} p_n$. Then 
\begin{align*}
k_n, \ell_n \in \overline{B}_\Omega(c_0; M_1+M_2)
\end{align*}
where $\overline{B}_\Omega(c_0; M_1+M_2)$ is closed metric ball of radius $M_1+M_2$ in $(\Omega, d_\Omega)$. Since the Hilbert metric is proper, we can pass to a subsequence such that $k_n \rightarrow k \in \Omega$ and $\ell_n \rightarrow \ell \in \Omega$. Then 
\begin{align*}
\lim_{n \rightarrow \infty} d_\Omega(\rho(\gamma_n) k, \rho(\gamma_n) k_n) =\lim_{n \rightarrow \infty} d_\Omega(k, k_n)= 0
\end{align*}
which implies from the definition of the Hilbert metric, see Observation~\ref{obs:dist_zero}, that 
\begin{align*}
\lim_{n \rightarrow \infty} \rho(\gamma_n) k =  \lim_{n \rightarrow \infty} \rho(\gamma_n) k_n =\xi(x).
\end{align*}
The same reasoning shows that
\begin{align*}
\lim_{n \rightarrow \infty} \rho(\gamma_n) \ell =  \lim_{n \rightarrow \infty} \rho(\gamma_n) \ell_n =p.
\end{align*}

Next view $\PGL_{d}(\Rb)$ as a subset of $\Pb(\End(\Rb^{d}))$ and pass to a subsequence so that $\rho(\gamma_n)$ converges to some $T$ in $\Pb(\End(\Rb^{d}))$. By Lemma~\ref{lem:dynamics}, $T$ has image $\xi(x^+)$ and kernel $\ker \eta(x^-)$ for some $x^+, x^- \in \partial \Gamma$. Since $\Pb(\ker \eta(x^-))\cap\Omega = \emptyset$ we see that 
\begin{align*}
\xi(x^+)=T(k) = \lim_{n \rightarrow \infty} \rho(\gamma_n)k  = \xi(x).
\end{align*}
However, by the same reasoning we have
\begin{align*}
\xi(x^+) = T(\ell) = \lim_{n \rightarrow \infty} \rho(\gamma_n)\ell = p.
\end{align*}
Hence $\xi(x) = p$ which is a contradiction. Thus $\xi(x)$ is an extreme point of $\Omega$.

\end{proof}

\subsection{Proof of Corollary~\ref{cor:convex_cocompact}}\label{sec:cor_convex_cocompact}

For the rest of this subsection suppose that $\Gamma$ is a word hyperbolic group and $\rho: \Gamma \rightarrow \PGL_{d}(\Rb)$ is an irreducible projective Anosov representation. Let $\xi: \partial \Gamma \rightarrow \Rb(\Rb^d)$ and $\eta : \partial \Gamma \rightarrow \Pb(\Rb^{d*})$ denote the boundary maps associated to $\rho$. Then define
 \begin{align*}
 V = \Spanset_{\Rb} \{ \xi(x)  \otimes \xi(x) : x \in \partial \Gamma\} \subset \Sym_{d}(\Rb)
 \end{align*}
where we make the identification $v \otimes v = v \prescript{t}{}{v} \in  \Sym_{d}(\Rb)$ when $v \in \Rb^d$.

Let $\rho_S : \Gamma \rightarrow \PGL(V)$ be the representation 
\begin{align*}
\rho_S(\gamma) X = \rho(\gamma)X \, \prescript{t}{}{\rho(\gamma)}.
\end{align*}
Using Theorem~\ref{thm:proper_action_implies} it is enough to show that $\rho_S$ is an irreducible projective Anosov representation and there exists a properly convex domain $\Omega_0 \subset \Pb(V)$ such that $\rho_S(\Gamma) \leq \Aut(\Omega_0)$. 

\begin{lemma} There exists a properly convex domain $\Omega_0 \subset \Pb(V)$ such that $\rho_S(\Gamma) \leq \Aut(\Omega_0)$. 
\end{lemma}

\begin{proof} As in Example~\ref{ex:pos_def_cone}, let 
\begin{align*}
\Pc := \{ [X] \in \Pb( \Sym_{d}(\Rb) ): X > 0\}.
\end{align*}
Then $\Pc$ is a properly convex domain in $\Pb(\Sym_{d}(\Rb))$. Since $\rho$ is irreducible, there exists $x_1, \dots, x_d \in \partial \Gamma$ such that $\xi(x_1),\dots, \xi(x_d)$ span $\Rb^d$. If $v_1, \dots, v_d \in \Rb^d$ are representatives of $\xi(x_1),\dots, \xi(x_d)$ respectively, then 
\begin{align*}
\left[\sum_{i=1}^d v_i \otimes v_i \right]\in \Pc \cap V.
\end{align*}
So $\Omega_0: = \Pc \cap \Pb(V)$ is a non-empty properly convex domain in $\Pb(V)$ and by construction $\rho_S(\Gamma) \leq \Aut(\Omega_0)$. 
\end{proof}

Given $\gamma \in \Gamma$ with infinite order, let $x^+_\gamma \in \partial \Gamma$ be the attracting fixed point of $\gamma$. And given a vector space $W$ and  $g \in \PGL(W)$ proximal let $\ell_g^+ \in \Pb(W)$ be the eigenline of $g$ corresponding to the eigenvalue of maximal modulus. 

\begin{lemma}\label{lem:preserve_proximality}\label{lem:rhoS_proximal} If $\gamma \in \Gamma$ has infinite order, then $g = \rho_S(\gamma)$ is proximal and 
\begin{align*}
\ell_g^+ = \xi(x_\gamma^+)  \otimes \xi(x_\gamma^+).
\end{align*}
\end{lemma}

\begin{proof} If $\lambda_1 > \lambda_2 \geq \dots \geq \lambda_d$ are the absolute values of the eigenvalues of $\rho(\gamma)$ normalized to have product one, then there exists $C > 0$ such that some subset of
\begin{align*}
C\lambda_i \lambda_j \text{ for } 1 \leq i \leq j \leq d
\end{align*}
are the the absolute values of the eigenvalues of $g=\rho_S(\gamma)$ normalized to have product one. By construction $\xi(x^+_\gamma) \otimes \xi(x^+_\gamma) \in V$ and is the eigenline corresponding to $C \lambda_1^2$, so $g$ is proximal and
\begin{align*}
\ell_g^+ = \xi(x_\gamma^+)  \otimes \xi(x_\gamma^+).
\end{align*}
\end{proof}

\begin{lemma}\label{lem:rhoS_irred} $\rho_S$ is irreducible. \end{lemma}

\begin{proof} Let $G$ be the Zariski closure of $\rho(\Gamma)$ in $\PGL_d(\Rb)$ and consider the representation 
\begin{align*}
\tau : G \rightarrow \PGL(V)
\end{align*}
given by 
\begin{align*}
\tau(g) X =gX \, \prescript{t}{}{g}.
\end{align*}

Since $\rho$ is an irreducible representation, $G$ acts irreducibly on $\Rb^d$. So $G$ acts minimally on the set 
\begin{align*}
\{ \ell^+_g : g \in G \text{ is proximal} \} \subset \Pb(\Rb^d),
\end{align*}
see for instance~\cite[Lemma 2.5]{B2000}. So $\tau(G)$ acts minimally on the set 
\begin{align*}
X=\{ \ell^+_g \otimes \ell_g^+: g \in G \text{ is proximal} \} \subset \Pb(\Sym_d(\Rb)). 
\end{align*}
Since $X \cap \Pb(V) \neq \emptyset$, $\tau(G)$ acts minimally on $X$, and $\tau(G) \cdot V = V$, we see that $X \subset \Pb(V)$. Further, $X$ spans $V$ by the definition of $V$.

Since $G$ is semisimple (see for instance~\cite[Lemma 2.19]{BCLS2015}), we can decompose $V= \oplus_{i=1}^m W_i$ where each $W_i \leq V$ is $\tau(G)$-invariant and the induced representation $G \rightarrow \PGL(W_i)$ is irreducible (see for instance~\cite[Chapter 5, Theorem 13]{OV1990}). 

Fix some $\gamma \in \Gamma$ with infinite order and let $h =\rho(\gamma)$. Then $\tau(h) \leq \PGL(V)$ is proximal by Lemma~\ref{lem:preserve_proximality}. Viewing $\PGL(V)$ as a subset of $\Pb(\End(V))$, Observation~\ref{obs:proximal_iterates} implies that 
\begin{align*}
 T = \lim_{n \rightarrow \infty} \phi(h)^n
\end{align*}
in $\Pb(\End(V))$ and the image of $T$ is $\ell^+_h\otimes \ell_h^+ $. By relabeling the $W_i$, we can suppose that there exists some element $w \in W_1 \setminus \ker T$. Then 
\begin{align*}
\ell^+_h\otimes \ell_h^+ = T([w]) = \lim_{k \rightarrow \infty} \phi(h)^{n_k} [w] \subset W_1.
\end{align*}
Then since $\tau(G)$ acts minimally on the set 
\begin{align*}
X=\{ \ell^+_g\otimes \ell^+_g : g \in G \text{ is proximal} \} 
\end{align*}
and $X$ spans $V$, we see that $W_1 = V$. Hence $\tau: G \rightarrow \PGL(V)$ is an irreducible representation. Since $\rho(\Gamma)$ is Zariski dense in $G$ and $\rho_S = \tau \circ \rho$, we then see that $\rho_S$ is also irreducible. 
\end{proof}

\begin{lemma}\label{lem:rhoS_proj} $\rho_S$ is projective Anosov. \end{lemma}

\begin{proof}
We define boundary maps $\xi_S : \partial \Gamma \rightarrow \Pb(V)$ and $\eta_S : \partial \Gamma \rightarrow \Pb(V^*)$ as follows. First, let
\begin{align*}
\xi_S(x) = \xi(x)  \otimes \xi(x).
\end{align*}
Next, let $f \in \Rb^{d*}$ be a lift of $\eta(x)$ and pick $w \in \Rb^d$ such that $f(v) = \prescript{t}{}{w} v$. Then define $\eta_S(x)$ by
\begin{align*}
\eta_S(x)\left(X \right) =  \prescript{t}{}{w} X w.
\end{align*}
By construction the maps $\xi_S$, $\eta_S$ are $\rho_S$-equivariant and continuous. Since the maps $\xi$, $\eta$ are transverse and
\begin{align*}
\eta_S(x)\left(\xi_S(y) \right) = \eta(x) \left( \xi(y) \right)^2,
\end{align*}
the maps $\xi_S$, $\eta_S$ are also transverse. Thus $\rho_S$ is projective Anosov by Proposition 4.10 in~\cite{GW2012}.
\end{proof}

\section{Basic properties of convex cocompact actions}

In this section we establish some basic properties of convex cocompact actions on properly convex domains. 

\subsection{Quasi-isometries}

The fundamental lemma of geometric group theory (see~\cite[Chapter IV, Theorem 23]{dlH2000})  immediately implies the following.

\begin{proposition}\label{thm:QI}
 Suppose $\Omega \subset \Rb(\Rb^{d})$ is a properly convex domain and $\Lambda \leq \Aut(\Omega)$ is a discrete convex cocompact group. Then $\Lambda$ is finitely generated and for any $p_0 \in \Omega$ the map 
\begin{align*}
\varphi \in \Lambda \rightarrow \varphi p_0
\end{align*}
induces an quasi-isometric embedding $\Lambda \rightarrow (\Omega, d_\Omega)$. 
\end{proposition}

  \subsection{Rescaling}
  
Given a finite dimensional real vector space $V$, let $K(V)$ denote the set of all compact subsets in $\Pb(V)$ equipped with the Hausdorff topology (with respect to a distance on $\Pb(V)$ induced by a Riemannian metric). 

Next let $\Xb(V)$ denote the set of properly convex open sets in $\Pb(V)$. Then the map 
\begin{align*}
\Omega \in \Xb(V) \rightarrow \overline{\Omega} \in K(V)
\end{align*}
is injective and so $\Xb(V)$ has a natural topology coming from $K(V)$. Finally, we let 
\begin{align*}
\Xb_0(V) = \left\{ (\Omega, x): \Omega \in \Xb(V), x \in \Omega\right\}
\end{align*}
equipped with the product topology.

In the 1960's Benz{\'e}cri proved the following theorem.

\begin{theorem} [Benz{\'e}cri's theorem]\label{thm:bens_thm}The group $\PGL(V)$ acts properly and cocompactly on $\Xb_0(V)$. Moreover, if $\Omega \subset \Pb(V)$ is a properly convex domain and $\Aut(\Omega)$ acts cocompactly on $\Omega$, then the orbit $\PGL(V) \cdot \Omega$ is closed in $\Xb(V)$. 
\end{theorem}
  
In this section we will use a result of Benoist to prove an analogue of Benz{\'e}cri's theorem for convex cocompact actions.
 
\begin{theorem}\label{thm:rescaling}
Suppose $\Omega \subset \Pb(\Rb^{d})$ is a properly convex domain, $G \leq \Aut(\Omega)$ is a subgroup, and there exists a closed convex subset $\Cc \subset \Omega$ such that $g\Cc=\Cc$ for all $g \in G$ and $G \backslash \Cc$ is compact. Assume $V \subset \Pb(\Rb^{d})$ is a subspace that intersects $\Cc$, $c_n \in \Cc \cap V$, and $h_n \in \PGL(V)$ satisfy
\begin{enumerate}
\item $h_n(\Omega \cap V) \rightarrow \Omega_V$ where $\Omega_V$ is a properly convex domain in $\Pb(V)$,
\item $h_n(\overline{\Cc} \cap V) \rightarrow \Cc_V$ where $\Cc_V$ is a properly convex closed set in $\Pb(V)$,
\item $h_n(c_n) \rightarrow p_\infty \in \Omega_V$.
\end{enumerate}
Then there exists some $\varphi \in \PGL_{d}(\Rb)$ so that 
\begin{align*}
\varphi(\Omega) \cap V= \Omega_V
\end{align*}
and 
\begin{align*}
\varphi(\overline{\Cc}) \cap V \supset \Cc_V.
\end{align*}
\end{theorem}

Before starting the proof of the theorem we make two observation about the Hausdorff topology.

\begin{observation}\label{obs:rescaling1} Suppose $\Omega_n \rightarrow \Omega$ in $\Xb(\Rb^d)$ and $K \subset \Omega$ is a compact set. Then $K \subset \Omega_n$ for $n$ sufficiently large. 
\end{observation}

\begin{proof} We can pick an affine chart $\mathbb{A} \subset \Pb(\Rb^d)$ such that $\Omega$ is relatively compact in $\mathbb{A}$. Then for $n$ sufficiently large, $\Omega_n$ is also relatively compact in $\mathbb{A}$. Then we can identify $\mathbb{A}$ with $\Rb^{d-1}$ and view $\Omega_n, \Omega$ as convex subsets of $\Rb^{d-1}$ (at least for $n$ sufficiently large). Then $\overline{\Omega_n} \rightarrow \overline{\Omega}$ is the Hausdorff distance induced by the Euclidean distance on $\Rb^{d-1}$. 

Now suppose, for a contradiction, that there exist $n_j \rightarrow \infty$ and $k_j \in K$ such that $k_j \notin \Omega_{n_j}$. By passing to a subsequence we can assume that $k_j \rightarrow k$. Now since $\Omega$ is open, there exists some $\epsilon > 0$ such that 
\begin{align*}
\{ x \in \Rb^{d-1} : \norm{k-x} \leq \epsilon\} \subset \Omega.
\end{align*}
Since each $\Omega_{n_j}$ is convex, we can find an real hyperplane $H_j$ such that $k_j \in H_j$ and $\Omega_{n_j} \cap H_j = \emptyset$. Then for $j$ sufficiently large, there exists some $x_j \in \Rb^{d}\setminus \Omega_{n_j}$ such that $d_{{\rm Euc}}(x_j, H_j) \geq \epsilon/2$ and $\norm{k-x_j} \leq \epsilon$. But then $x_j \in \Omega$ and so
\begin{align*}
d_{{ \rm Euc}}^{{\rm Haus}}(\overline{\Omega}_{n_j}, \Omega) \geq d_{{\rm Euc}}(\Omega_{n_j}, x_j) \geq \epsilon/2
\end{align*}
which is a contradiction. 
\end{proof}

\begin{observation}\label{obs:rescaling2}Suppose $\Omega_n \rightarrow \Omega$ in $\Xb(\Rb^d)$.  If $V \subset \Pb(\Rb^d)$ is a subspace and $V \cap \Omega \neq \emptyset$, then $\Omega_n \cap V \rightarrow \Omega \cap V$ in $\Xb(V)$.
\end{observation}

\begin{proof}
Since $K(V)$ is compact, it is enough to show that every convergent subsequence of $\overline{\Omega_n \cap V}$ converges to $\overline{\Omega \cap V}$. So suppose that $\overline{\Omega_n \cap V} \rightarrow C$ in $K(V)$. 

Then by the definition of the Hausdorff topology we have $C \subset \overline{\Omega \cap V}$. 

Since $\Omega$ is convex, we have $\overline{\Omega} \cap V = \overline{\Omega \cap V}$. So we can pick a sequence $K_m \subset \Omega \cap V$ of compact sets such that 
\begin{align*}
 \overline{\cup K_m} = \overline{\Omega \cap V}.
\end{align*}
Fix $m$. Then $K_m \subset \Omega_n$ for $n$ sufficiently large by Observation~\ref{obs:rescaling1}. So $K_m \subset C$. Since $m$ was arbitrary 
\begin{align*}
 \overline{\Omega \cap V} =  \overline{\cup K_m}  \subset C.
\end{align*}
Hence $C = \overline{\Omega \cap V}$. 
\end{proof}

\begin{proof}[Proof of Theorem~\ref{thm:rescaling}] By Lemma 2.8 in~\cite{B2003a} there exists $g_n \in \PGL_{d}(\Rb)$ and a properly convex domain $\Omega^\prime \subset \Pb(\Rb^{d})$ such that
\begin{enumerate}
 \item $g_n|_V = h_n$,
 \item $\Omega_n := g_n\Omega \rightarrow \Omega^\prime$, and 
 \item $\Omega^\prime \cap V = \Omega_V$.
\end{enumerate}

Now fix a point $p_0 \in \Cc$. Then there exist $R \geq 0$ and a sequence $\gamma_n \in G$ such that 
\begin{align*}
d_\Omega(c_n, \gamma_n p_0) \leq R.
\end{align*}
Next consider the element $\varphi_n = g_n \gamma_n$. Note that  
\begin{align*}
d_{\Omega_n}(\varphi_np_0, g_n c_n) = d_{\Omega}(\gamma_np_0, c_n) \leq R.
\end{align*}
Then since 
\begin{align*}
\lim_{n \rightarrow \infty} g_nc_n = \lim_{n \rightarrow \infty} h_n c_n =   p_\infty \in \Omega^\prime
\end{align*}
and $d_{\Omega_n}$ converges locally uniformly to $d_{\Omega^\prime}$ we can pass to a subsequence so that $\varphi_np_0 \rightarrow q_\infty \in \Omega^\prime$. 

Then $\varphi_n(\Omega, p_0) \rightarrow (\Omega^\prime, q_\infty)$ and since $\PGL_{d}(\Rb)$ acts properly on $\Xb_0(\Rb^d)$, we can pass to a subsequence such that $\varphi_n \rightarrow \varphi \in \PGL_{d}(\Rb)$. 

Then by the Observation~\ref{obs:rescaling2}
\begin{align*}
\varphi(\Omega) \cap V = \lim_{n \rightarrow \infty} \varphi_n(\Omega) \cap V = \lim_{n \rightarrow \infty} g_n(\Omega) \cap V = \Omega^\prime \cap V = \Omega_V.
\end{align*}

By passing to a subsequence we can suppose that the sequence $g_n(\overline{\Cc}) \cap V$ converges in $K(V)$. Then, by the definition of the Hausdorff topology,
 \begin{align*}
\varphi(\overline{\Cc}) \cap V 
& \supset \lim_{n \rightarrow \infty} \varphi_n(\overline{\Cc}) \cap V  = \lim_{n \rightarrow \infty} g_n(\overline{\Cc}) \cap V \\
& \supset  \lim_{n \rightarrow \infty} h_n(\overline{\Cc} \cap V) \cap V =  \lim_{n \rightarrow \infty} h_n(\overline{\Cc} \cap V) = \Cc_V.
\end{align*}
\end{proof}

\section{Regular convex cocompactness implies projective Anosovness}\label{sec:convex_cocpct_implies_A}

In this section we prove Theorem~\ref{thm:RCC_anosov_intro} from the introduction. The proof uses many ideas from Benoist's work on the Hilbert metric~\cite{B2003a, B2004}.

\begin{theorem}\label{thm:RCC_anosov}
Suppose $\Omega \subset \Pb(\Rb^{d})$ is a properly convex domain and $\Lambda \leq \Aut(\Omega)$ is a discrete convex cocompact subgroup. Let $\Cc$ be a closed convex subset of $\Omega$ such that $g\Cc = \Cc$ for all $g \in \Lambda$ and $\Lambda \backslash \Cc$ is compact. If $\Lambda$ is an irreducible subgroup of $\PGL_{d}(\Rb)$, then the following are equivalent: 
\begin{enumerate}
\item every point in $\overline{\Cc} \cap \partial \Omega$ is a $C^1$ point of $\partial \Omega$, 
\item every point in $\overline{\Cc} \cap \partial \Omega$ is an extreme point of $\partial \Omega$ 
\end{enumerate}
Moreover, when these conditions are satisfied $\Lambda$ is word hyperbolic and the inclusion representation $\Lambda \hookrightarrow \PGL_{d}(\Rb)$ is projective Anosov.
\end{theorem}

\begin{remark} In the special case when $\Omega = \Cc$, Theorem~\ref{thm:RCC_anosov} was established by Benoist~\cite{B2004}, see Theorem~\ref{thm:ben_char} in the introduction. \end{remark}

For the rest of the section fix a properly convex domain $\Omega \subset \Pb(\Rb^{d})$, a discrete convex cocompact subgroup $\Lambda \leq \Aut(\Omega)$, and a closed convex subset $\Cc \subset \Omega$ which satisfy the hypothesis of Theorem~\ref{thm:RCC_anosov}.

Notice that $\Cc$ has non-empty interior since $\Lambda$ is irreducible and preserves the subspace $ \Spanset_{\Rb} \left\{ c: c \in  \Cc \right\}$.

\begin{lemma}\label{lem:E} With the notation above, if each $q \in \partial \Omega \cap \overline{\Cc}$ is a $C^1$ point of $\partial \Omega$, then each $q \in \partial \Omega \cap \overline{\Cc}$ is an extreme point of $\Omega$.
\end{lemma}

\begin{proof} 
Suppose for a contradiction that there exists a point $q \in \partial \Omega \cap \overline{\Cc}$ which is not an extreme point of $ \Omega$. Then after making a change of coordinates we can assume the following: 
\begin{enumerate}
\item $q=[1:0:\dots:0] \in \partial \Omega \cap \overline{\Cc}$, 
\item $[1:0:1:0:\dots:0] \in \Cc$, 
\item $\Omega \subset \{ [1:x_1:x_2:\dots : x_{d-1}] \in \Pb(\Rb^{d})  : x_2 >0\}$, and
\item $\{ [1:t:0:\dots:0] \in \Pb(\Rb^{d}) : t \in [-1,1]\} \subset \partial \Omega$.
\end{enumerate}

Now let 
\begin{align*}
V = \{ [x_1 : x_2 : x_3: 0 : \dots : 0]  \in \Pb(\Rb^{d}) : x_1, x_2, x_3 \in \Rb\},
\end{align*}
$c_n = \left[ 1: 0:\frac{1}{n} : 0 : \dots : 0\right] \in \Cc \cap V$, and $h_n \in \PGL(V)$ be given by
\begin{align*}
h_n[x_1:x_2:x_3:0:\dots:0] = [x_1:x_2:nx_3:0:\dots:0].
\end{align*}
Then $h_nc_n \rightarrow [1:0:1:0:\dots:0]$, 
\begin{align*}
h_n(\Omega \cap V) \rightarrow \Omega_V := \{ [1:s:t:0:\dots:0] \in \Pb(\Rb^{d}) : [1:s:0:\dots:0] \in \partial \Omega \text{ and } t > 0\},
\end{align*}
and
\begin{align*}
h_n(\Cc \cap V) \rightarrow \Cc_V := \{ [1:s:t:0:\dots:0] \in \Pb(\Rb^{d}) : [1:s:0:\dots:0] \in \partial \Omega \cap \overline{\Cc} \text{ and } t > 0\}.
\end{align*}
Clearly $\Omega_V$ is properly convex and so by Theorem~\ref{thm:rescaling}, there exists some $\varphi \in \PGL_{d}(\Rb)$ such that $\varphi(\Omega) \cap V = \Omega_V$ and $\Cc_V \subset \varphi(\overline{\Cc}) \cap V$. But then
\begin{align*}
[0:0:1:0:\dots:0] 
\end{align*} 
is a not a $C^1$ point of $\partial \Omega_V$ and hence 
\begin{align*}
\varphi^{-1}[0:0:1:0:\dots:0] \in \overline{\Cc} \cap \partial \Omega
\end{align*}
 is not a $C^1$ point of $\partial \Omega$. So we have a contradiction. 
\end{proof}

\begin{lemma} With the notation above, if each $q \in \partial \Omega \cap \overline{\Cc}$ is an extreme point of $\Omega$, then each $q \in \partial \Omega \cap \overline{\Cc}$ is a $C^1$ point of $\partial\Omega$.
\end{lemma}

\begin{proof}
Suppose for a contradiction that there exists a point $q \in \partial \Omega \cap \overline{\Cc}$ which is not a $C^1$ point of $\partial \Omega$. Then there exist two different hyperplanes $H_1, H_2$ such that $q \in H_1 \cap H_2$ and $H_1 \cap \Omega = H_2 \cap \Omega = \emptyset$. Since $\Cc$ has non-empty interior, there exists a two dimensional subspace $V \subset \Pb(\Rb^{d})$ so that $V$ intersects the interior of $\Cc$, and $V \cap H_1 \neq V \cap H_2$.

By making a change of coordinates, we can assume that 
\begin{enumerate}
\item $q = [1:0:\dots:0]$, 
\item $V = \{ [x_1 : x_2 : x_3: 0 : \dots : 0]  \in \Pb(\Rb^{d}) : x_1, x_2, x_3 \in \Rb\}$,
\item $\Omega \cap V \subset \{ [1:x_1:x_2:0:\dots:0] \in \Pb(\Rb^{d}) : x_2 > 0\}$,
\item $[1:0:1:\dots:0]$ is contained in the interior of $\Cc$, and
\item there exists $\alpha_1 < 0 < \alpha_2$ such that 
\begin{align*}
H_i \cap V = \{ [1:t:\alpha_i t:0:\dots:0] \in \Pb(\Rb^{d}): t \in \Rb\} \cup\{[0:1:\alpha_i]\}.
\end{align*}
\end{enumerate}

Now since $[1:0:1:\dots:0]$ is contained in the interior of $\Cc$, there exists $\epsilon>0$ and $\beta_1 < 0 < \beta_2$ such that 
\begin{align*}
\{ [1:t:\beta_2 t:0:\dots:0] \in \Pb(\Rb^{d}) : 0 < t < \epsilon \} \subset \Cc
\end{align*}
and 
\begin{align*}
\{ [1:t:\beta_1 t:0:\dots:0] \in \Pb(\Rb^{d}) : -\epsilon < t < 0 \} \subset \Cc.
\end{align*}

Next consider the points $c_n = [1:0:\frac{1}{n}:0:\dots:0]$ and let $h_n \in \PGL(V)$ be given by
\begin{align*}
h_n[x_1:x_2:x_3:0:\dots:0] = \left[x_1: n x_2:n x_3:0:\dots:0\right].
\end{align*}
Then $h_nc_n \rightarrow [1:0:1:0:\dots:0]$, $h_n(\Omega \cap V)$ converges to the tangent cone $\mathcal{TC}_q(\Omega \cap V)$ of $\Omega \cap V$ at $q$, and $h_n(\Cc \cap V)$ converges to the tangent cone $\mathcal{TC}_q(\Cc \cap V)$ of $\Cc \cap V$ at $q$. 

By construction $\mathcal{TC}_q(\Omega \cap V)$ is a properly convex domain in $V$. So by Theorem~\ref{thm:rescaling}, there exists some $\varphi \in \PGL_{d+1}(\Rb)$ such that $\varphi(\Omega) \cap V = \mathcal{TC}_g(\Omega \cap V)$ and $\varphi(\overline{\Cc}) \cap V \supset \overline{\mathcal{TC}_g(\Cc \cap V)}$. But then 
\begin{align*}
\varphi^{-1} \{  [0:1:s :0:\dots:0] \in \Pb(\Rb^{d}) : \beta_1 \leq s \leq \beta_2 \}\subset \overline{\Cc} \cap \partial \Omega
\end{align*}
which contradicts the fact that every point in $\overline{\Cc} \cap \partial \Omega$ is an extreme point. 
\end{proof}

For the remainder of the section we assume, in addition, that 
\begin{enumerate}
\item every point in $\overline{\Cc} \cap \partial \Omega$ is a $C^1$ point of $\partial \Omega$ and
\item every point in $\overline{\Cc} \cap \partial \Omega$ is an extreme point of $\partial \Omega$. 
\end{enumerate}

\begin{lemma} With the notation above, $\Lambda$ is word hyperbolic. \end{lemma}

\begin{proof}
Fix a finite symmetric generating set $S$ of $\Lambda$. By Proposition~\ref{thm:QI}, $(\Lambda, d_S)$ is quasi-isometric to $(\Cc, d_\Omega)$ and so it is enough to show that $(\Cc, d_\Omega)$ is Gromov hyperbolic. Now for each $x,y \in \Cc$ let $\sigma_{x,y}$ be the geodesic joining $x$ to $y$ which parametrizes the line segment joining them. By Proposition~\ref{prop:GH_suff} it is enough to show that there exists an $\delta > 0$ such that every geodesic triangle in $(\Cc, d_\Omega)$ of the form  $\sigma_{x,y}, \sigma_{y,z}, \sigma_{z,x}$ is $\delta$-thin. Suppose not. Then for every $n > 0$ there exists points $x_n, y_n, z_n, u_n \in \Cc$ such that $u_n \in \sigma_{x_n, y_n}$ and 
\begin{align*}
d_\Omega(u_n, \sigma_{y_n,z_n} \cup \sigma_{z_n, x_n}) > n.
\end{align*}
By replacing the the points $x_n, y_n, z_n,u_n$ by $g_n x_n, g_n y_n, g_n z_n, g_n u_n$ for some $g_n \in \Lambda$ we can  assume that the sequence $u_n$ is relatively compact in $\Cc$. Then by passing to a subsequence we can suppose that $u_n \rightarrow u \in \Cc$. By passing to another subsequence we can assume that $x_n, y_n, z_n \rightarrow x,y,z \in \overline{\Cc}$. Since 
\begin{align*}
d_\Omega(u_n, \{x_n,y_n,z_n\}) > n
\end{align*}
we must have $x,y,z \in \overline{\Cc} \cap \partial \Omega$. The image of $\sigma_{x_n,y_n}$ converges to a line segment containing $x,y,u$. Since $u \in \Cc$ and $x,y \in \partial \Omega$ we must have $x \neq y$. Then either $z \neq x$ or $z \neq y$. By relabeling we may assume that $z \neq x$. Then the image of $\sigma_{x_n,z_n}$ converges to the line segment $[x,z]$. Since every point in $\overline{\Cc} \cap \partial \Omega$ is an extreme point of $\partial \Omega$ and $x \neq z$, we must have $(x,z) \subset \Omega$. So 
\begin{align*}
\infty = \lim_{n \rightarrow \infty} d_\Omega(u_n, \sigma_{z_n,x_n}) = d_\Omega(u, (z,x)) < \infty.
\end{align*}
So we have a contradiction and hence $\Lambda$ is word hyperbolic. 
\end{proof}

\begin{lemma} With the notation above, there exists a $\Lambda$-equivariant homeomorphism $\xi : \partial \Lambda \rightarrow \overline{\Cc} \cap \partial \Omega$.
\end{lemma}

\begin{proof}
Since every point in $\overline{\Cc} \cap \partial \Omega$ is an extreme point, this follows from Lemma~\ref{lem:GP}.
\end{proof}

\begin{lemma} With the notation above, the inclusion representation $\Lambda \hookrightarrow \PGL_{d}(\Rb)$ is projective Anosov. \end{lemma}

\begin{proof} Let $\xi : \partial \Lambda \rightarrow \overline{\Cc} \cap \partial \Omega$ be the $\Lambda$-equivariant homeomorphism from the previous lemma. Since every point in $\overline{\Cc} \cap \partial \Omega$ is a $C^1$ point, the map $\eta: \partial \Lambda \rightarrow \Pb(\Rb^{d*})$ with
\begin{align*}
\Pb(\ker \eta(x)) = T_{\xi(x)} \partial \Omega
\end{align*}
is well defined, continuous, and $\Lambda$-equivariant. 

We claim that 
\begin{align*}
\xi(x) + \ker \eta(y) = \Rb^{d}
\end{align*}
for $x,y \in \partial \Gamma$ distinct. If not, then 
\begin{align*}
[\xi(x), \xi(y)] \subset \overline{\Cc} \cap \Pb(\ker \eta(y)) =\overline{\Cc} \cap T_{\xi(y)} \partial \Omega \subset \overline{\Cc} \cap \partial\Omega.
\end{align*}
But since each $q \in \partial \Omega \cap \overline{\Cc}$ is an extreme point of $\partial \Omega$ we see that this is impossible. 

Then Proposition 4.10 in~\cite{GW2012} implies that the inclusion representation is projective Anosov. 

\end{proof}

\section{Proof of Theorem~\ref{thm:main}}\label{subsec:pf_main_thm}

We now prove Theorem~\ref{thm:main} from the introduction: 

\begin{theorem}
Suppose $G$ is a semisimple Lie group with finite center and $P \leq G$ is a parabolic subgroup. Then there exists a finite dimensional real vector space $V$ and an irreducible representation $\phi:G \rightarrow \PSL(V)$ with the following property: if $\Gamma$ is a word hyperbolic group and $\rho:\Gamma \rightarrow G$ is a Zariski dense representation with finite kernel, then the following are equivalent:
\begin{enumerate}
\item $\rho$ is $P$-Anosov,
\item there exists a properly convex domain $\Omega \subset \Pb(V)$ such that $(\phi\circ\rho)(\Gamma)$ is a regular convex cocompact subgroup of $\Aut(\Omega)$. 
\end{enumerate}
\end{theorem}

For the rest of the section fix $G$ a semisimple Lie group with finite center and $P \leq G$ a parabolic subgroup.

By Theorem~\ref{thm:GW}, there exist a finite dimensional real vector space $V_0$ and an irreducible representation $\phi_0:G \rightarrow \PSL(V_0)$ with the following property: if $\Gamma$ is a word hyperbolic group and $\rho:\Gamma \rightarrow G$ is a representation, then the following are equivalent:
\begin{enumerate}
\item $\rho$ is $P$-Anosov,
\item $\phi_0 \circ \rho$ is projective Anosov.
\end{enumerate}

We will construct a new representation of $G$ by taking the tensor product of $\phi_0$ with itself. In general, this will not produce an irreducible representation and so we will construct a subspace of $V_0 \otimes V_0$ where $\phi_0 \otimes \phi_0$ acts irreducibly.  

For a proximal element $g \in \PSL(V_0)$ let $\ell^+_g \in \Pb(V_0)$ be the eigenline of $g$ corresponding to the eigenvalue of largest absolute value.  Then consider the vector space 
\begin{align*}
V =\Spanset_{\Rb}\{ \ell^+_g\otimes \ell^+_g: g \in \phi_0(G) \text{ is proximal} \}
\end{align*}
and the representation $\phi: G \rightarrow \SL(V)$ given by 
\begin{align*}
\phi(g)(v \otimes v) = (\phi_0(g)v) \otimes (\phi_0(g)v).
\end{align*}
Notice that we can assume that $V \neq (0)$, for otherwise there is nothing to prove.

\begin{lemma}\label{lem:phi0_proximal} With the notation above, if $g \in G$ and $\phi_0(g)$ is proximal, then $\phi(g)$ is proximal and $\ell_{\phi_0(g)}^+ \otimes \ell_{\phi_0(g)}^+$ is the eigenline of $\phi(g)$ corresponding to the eigenvalue of largest absolute value. 
\end{lemma}

\begin{proof} The argument is similar to the proof of Lemma~\ref{lem:rhoS_proximal}.
\end{proof}

\begin{lemma} With the notation above, $\phi: G \rightarrow \SL(V)$ is an irreducible representation. \end{lemma}

\begin{proof} The argument is similar to the proof of Lemma~\ref{lem:rhoS_irred}. 
\end{proof}

We now complete the proof of the theorem.

\begin{lemma} With the notation above, if $\Gamma$ is a word hyperbolic group and $\rho:\Gamma \rightarrow G$ is a Zariski dense representation with finite kernel, then the following are equivalent:
\begin{enumerate}
\item $\rho$ is $P$-Anosov,
\item there exists a properly convex domain $\Omega \subset \Pb(V)$ such that $(\phi\circ\rho)(\Gamma)$ is a regular convex cocompact subgroup of $\Aut(\Omega)$. 
\end{enumerate}
\end{lemma}

\begin{proof}
If $\rho$ is $P$-Anosov, then $\phi_0 \circ \rho$ is projective Anosov representation by our choice of $\phi_0$.   Let $\xi_0: \partial \Gamma \rightarrow \Pb(V_0)$ and $\eta_0: \partial \Gamma \rightarrow \Pb(V_0^*)$ be the associated boundary maps. Since $\phi_0: G \rightarrow \PSL(V_0)$ is irreducible and $\rho(\Gamma) \leq G$ is Zariski dense, we see that $\phi_0 \circ \rho : \Gamma \rightarrow \PSL(V_0)$ is irreducible. So by Corollary~\ref{cor:convex_cocompact}, if 
\begin{align*}
V^\prime = \Spanset\{ \xi_0(x) \otimes \xi_0(x) : x \in \partial \Gamma\},
\end{align*}
then there exists a properly convex domain $\Omega \subset \Pb(V^\prime)$ so that  $(\phi \circ \rho)(\Gamma)$ is a regular convex cocompact subgroup of $\Aut(\Omega)$.  Since $\phi: G \rightarrow \PSL(V)$ is irreducible and $\rho(\Gamma) \leq G$ is Zariski dense, we see that $\phi \circ \rho : \Gamma \rightarrow \PSL(V)$ is irreducible. Then since $V^\prime \subset V$, we must have $V^\prime = V$. 

Next suppose that there exists some properly convex domain $\Omega \subset \Pb(V)$ such that $(\phi \circ \rho)(\Gamma) \leq \Aut(\Omega)$ is a regular convex cocompact subgroup.  Since $\phi: G \rightarrow \PSL(V)$ is irreducible and $\rho(\Gamma) \leq G$ is Zariski dense, we see that $\phi \circ \rho : \Gamma \rightarrow \PSL(V)$ is irreducible. Hence Theorem~\ref{thm:RCC_anosov}  implies that $\phi \circ \rho$ is a projective Anosov representation. Let $\xi: \partial \Gamma \rightarrow \Pb(V)$ and $\eta: \partial \Gamma \rightarrow \Pb(V^*)$ be the associated boundary maps. 

We claim that there exist maps $\xi_0: \partial \Gamma \rightarrow \Pb(V_0)$ and $\eta_0: \partial \Gamma \rightarrow \Pb(V_0^*)$ such that 
\begin{align*}
\xi(x) = \xi_0(x) \otimes \xi_0(x)
\end{align*}
and 
\begin{align*}
\eta(x) = \eta_0(x) \otimes \eta_0(x)
\end{align*}
for all $x \in \partial \Gamma$. Since $\rho(\Gamma)$ is Zariski dense in $G$ and $\phi_0(G)$ contains proximal elements, there exists some $\varphi \in \Gamma$ such that $(\phi_0 \circ \rho)(\varphi)$ is proximal, see for instance~\cite{P1994}. Let $x^+ \in \partial \Gamma$ be the attracting fixed point of $\varphi$ in $\partial \Gamma$. Then $\xi(x^+)$ is the eigenline of $(\phi \circ \rho)(\varphi)$ whose eigenvalue has maximal absolute value. Since  $(\phi_0 \circ \rho)(\varphi)$ is proximal, Lemma~\ref{lem:phi0_proximal} says that
\begin{align*}
\xi(x^+) = \ell^+ \otimes \ell^+
\end{align*}
where $\ell^+ \in \Pb(V)$ is the eigenline of $(\phi_0 \circ \rho)(\varphi)$ whose eigenvalue has maximal absolute value. Now 
\begin{enumerate}
\item $\xi: \partial \Gamma \rightarrow \Pb(V)$ is continuous and $(\phi \circ \rho)$-equivariant, 
\item the set 
\begin{align*}
A=\{ [v \otimes v] : v \in V_0 \setminus \{0\}\} \subset \Pb(V)
\end{align*}
is closed and $\phi(G)$-invariant, and 
\item the set $ \Gamma \cdot x^+$ is dense in $\partial \Gamma$.
\end{enumerate}
Since $\xi(x^+) \in A$, the three properties above imply that $\xi(\partial \Gamma) \subset A$. Hence there exists a map $\xi_0: \partial \Gamma \rightarrow \Pb(V_0)$ such that 
\begin{align*}
\xi(x) = \xi_0(x) \otimes \xi_0(x)
\end{align*}
for all $x \in \partial \Gamma$. Since 
\begin{align*}
[v] \in \Pb(V_0) \rightarrow [v \otimes v] \in A
\end{align*}
is a diffeomorphism, the map $\xi_0$ is continuous. Finally, by construction, the map $\xi_0$ is $(\phi_0 \circ \rho)$-equivariant. 

Applying this same argument to $\eta$ yields a continuous $(\phi_0 \circ \rho)$-equivariant map $\eta_0: \partial \Gamma \rightarrow \Pb(V_0^*)$ such that 
\begin{align*}
\eta(x) = \eta_0(x) \otimes \eta_0(x)
\end{align*}
for all $x \in \partial \Gamma$. 

If $x,y \in \partial \Gamma$, then
\begin{align*}
\eta(y) \left(\xi(x)\right) = \eta_0(y) \otimes \eta_0(y)\left( \xi_0(x) \otimes \xi_0(x) \right) = \eta_0(y)\left(\xi_0(x)\right)^2.
\end{align*}
Since $\xi$ and $\eta$ are transverse, this implies that $\xi_0$ and $\eta_0$ are transverse. 

Finally, since the representation $\phi_0: G \rightarrow \PSL(V_0)$ is irreducible and $\rho(\Gamma) \leq G$ is Zariski dense, we see that $\phi_0 \circ \rho : \Gamma \rightarrow \PSL(V_0)$ is irreducible. Hence by Proposition 4.10 in~\cite{GW2012} we see that $\phi_0 \circ \rho: \Gamma \rightarrow G$ is a projective Anosov representation. Thus by our choice of $\phi_0$ we see that $\rho: \Gamma \rightarrow G$ is $P$-Anosov.
 
\end{proof}

\section{Entropy rigidity}\label{sec:ent_rigidity}

The proof of Theorem~\ref{thm:ent_rig} has three steps: first we use results of Coornaert-Knieper, Coornaert, and Cooper-Long-Tillmann to transfer to the Hilbert metric setting, then we use a result of Tholozan to transfer to the Riemmanian metric setting, and finally we use an argument of Liu to prove rigidity. This general approach is based on the arguments in~\cite{BMZ2015}.

It will also be more notationally convenient in this section to work with $\Pb(\Rb^{d+1})$ instead of $\Pb(\Rb^d)$. 

\subsection{Some notation}

Suppose $(X,d)$ is a proper metric space and $x_0 \in X$ is some point. If $G \leq \Isom(X,d)$ is a discrete subgroup, then define the \emph{Poincar{\'e} exponent of $G$} to be 
\begin{align*}
\delta_{G}(X,d) := \limsup_{r \rightarrow \infty} \frac{1}{r} \log \# \left\{ g \in G : d(x_0, g x_0) \leq r \right\}.
\end{align*}
Notice that $\delta_G(X,d)$ does not depend on $x_0$. If $X$ has a measure $\mu$ one can also define the \emph{volume growth entropy relative to $\mu$} as
\begin{equation*}
h_{vol}(X,d,\mu) := \limsup_{r \rightarrow \infty} \frac{1}{r} \log \mu \left( \left\{ x \in X : d(x, x_0) \leq r \right\} \right).
\end{equation*}
If the measure $\mu$ is $\Isom(X,d)$-invariant, finite on bounded sets, and positive on open sets, then 
\begin{equation*}
\delta_G(X,d) \leq h_{vol}(X,d,\mu)
\end{equation*}
by the proof of Proposition 2 in~\cite{LW2010}. In the case in which $(X,g)$ is a Riemannian manifold, we will let 
\begin{align*}
h_{vol}(X,g):=h_{vol}(X,d, \Vol)
\end{align*}
where $d$ is the distance induced by $g$ and $\Vol$ is the Riemannian volume associated to $g$.

\subsection{Transferring to the Hilbert metric setting} As in the introduction, we define the Hilbert entropy of a representation $\rho: \Gamma \rightarrow \PGL_{d}(\Rb)$ to be 
\begin{align*}
H_\rho= \limsup_{r \rightarrow \infty} \frac{1}{r} \log \# \left\{ [\gamma] \in [\Gamma] : \frac{1}{2} \log \left(\frac{\lambda_{1}(\rho(\gamma))}{\lambda_{d}(\rho(\gamma))} \right)\leq r\right\}
\end{align*}
where $[\Gamma]$ is the set of conjugacy classes of $\Gamma$. By combining results of Coornaert-Knieper, Coornaert, and Cooper-Long-Tillmann, we will establish the following proposition.

\begin{proposition}\label{prop:CK} Suppose $\Gamma$ is a word hyperbolic group, $\rho: \Gamma \rightarrow \PGL_{d+1}(\Rb)$ is an irreducible projective Anosov representation, and  $\Omega \subset \Pb(\Rb^{d+1})$ is a properly convex domain such that $\rho(\Gamma) \leq \Aut(\Omega)$ is a regular convex cocompact subgroup. Then
\begin{align*}
H_{\rho} = \delta_{\rho(\Gamma)}(\Omega, d_\Omega).
\end{align*}
Moreover, for any $p_0 \in \Omega$ there exists $C \geq 1$ such that
\begin{align*}
 \frac{1}{C} e^{H_{\rho} r} \leq \# \left\{ \gamma \in \Gamma : d_\Omega(p_0, \rho(\gamma) p_0) \leq r \right\} \leq C e^{H_{\rho} r}.
\end{align*}
\end{proposition}

\begin{proof} Let $\Cc \subset \Omega$ be a closed convex subset such that $g\Cc = \Cc$ for all $g \in \rho(\Gamma)$, $\rho(\Gamma) \backslash \Cc$ is compact, and every point in $\overline{\Cc} \cap \partial \Omega$ is a $C^1$ extreme point of $\Omega$.

Using Selberg's lemma, we can find a finite index subgroup $\Gamma_0 \leq \Gamma$ such that $\rho(\Gamma_0)$ is torsion free.  Then $H_{\rho} = H_{\rho|_{\Gamma_0}}$, $\delta_{\rho(\Gamma_0)}(\Omega, d_\Omega)=\delta_{\rho(\Gamma)}(\Omega, d_\Omega)$, and $\rho(\Gamma_0) \backslash \Cc$ is compact.

For $\gamma \in \Gamma_0$ define 
\begin{align*}
\tau(\gamma) = \inf_{c \in \Cc} d_\Omega(\rho(\gamma)c,c).
\end{align*}
Since $(\Cc, d_\Omega)$ is a proper geodesic metric space, $\rho(\Gamma_0)$ acts cocompactly on $\Cc$, $\Gamma_0$ is word hyperbolic, and $\ker \rho$ is finite, Theorem 1.1 in~\cite{CK2002} says that
\begin{align}
\label{eq:entropy_periodic_orbits}
\delta_{\rho(\Gamma_0)}(\Omega, d_\Omega) =  \lim_{r \rightarrow \infty} \frac{1}{r} \log \# \left\{ [\gamma] \in [\Gamma_0] : \tau(\gamma) \leq r \right\}.
\end{align}
Next we claim that 
\begin{align*}
\tau(\gamma) = \frac{1}{2} \log \left(\frac{\lambda_{1}(\rho(\gamma))}{\lambda_{d+1}(\rho(\gamma))} \right)
\end{align*}
for every $\gamma \in \Gamma_0$. Fix some $\gamma \in \Gamma_0$. Then Proposition 2.1 in~\cite{CLT2015} says that 
\begin{align*}
\inf_{x \in \Omega} d_\Omega(\rho(\gamma)x,x) = \frac{1}{2} \log \left(\frac{\lambda_{1}(\rho(\gamma))}{\lambda_{d+1}(\rho(\gamma))} \right).
\end{align*}
and so 
\begin{align*}
\tau(\gamma) \geq \frac{1}{2} \log \left(\frac{\lambda_{1}(\rho(\gamma))}{\lambda_{d+1}(\rho(\gamma))} \right).
\end{align*}
Since $\gamma$ has infinite order we see that $\rho_0(\gamma)$ is biproximal, that is $\rho_0(\gamma)$ and $\rho_0(\gamma)^{-1}$ are proximal. So if $\ell^+$ and $\ell^-$ are the attracting and repelling eigenlines of $\rho_0(\gamma)$ respectively, then Corollary~\ref{cor:dynamics} implies that $\ell^+, \ell^- \in \overline{\Cc} \cap \partial \Omega$. Since every point in $\overline{\Cc} \cap \partial \Omega$ is an extreme point of $\Omega$, we then see that $(\ell^+, \ell^-) \subset \Cc$. But if $p \in (\ell^+, \ell^-)$, then
\begin{align*}
d_\Omega(\rho_0(\gamma) p,p) = \frac{1}{2} \log \left(\frac{\lambda_{1}(\rho(\gamma))}{\lambda_{d+1}(\rho(\gamma))} \right)
\end{align*}
by the definition of the Hilbert distance.
Hence 
\begin{align}
\label{eq:min_translate}
\tau(\gamma) = \frac{1}{2} \log \left(\frac{\lambda_{1}(\rho(\gamma))}{\lambda_{d+1}(\rho(\gamma))} \right).
\end{align}

Then by Equations~\eqref{eq:entropy_periodic_orbits} and~\eqref{eq:min_translate}
\begin{align*}
\delta_{\rho(\Gamma_0)} & (\Omega, d_\Omega) =  \lim_{r \rightarrow \infty} \frac{1}{r} \log \# \left\{ [\gamma] \in [\Gamma_0] : \tau(\gamma) \leq r \right\}  \\ 
&=  \lim_{r \rightarrow \infty} \frac{1}{r} \log \# \left\{ [\gamma] \in [\Gamma_0] : \frac{1}{2} \log \left(\frac{\lambda_{1}(\rho(\gamma))}{\lambda_{d+1}(\rho(\gamma))} \right) \leq r \right\} =  H_{\rho|_{\Gamma_0}}
\end{align*}
and so
\begin{align*}
H_{\rho} = H_{\rho|_{\Gamma_0}}=\delta_{\rho(\Gamma_0)}(\Omega, d_\Omega)=\delta_{\rho(\Gamma)}(\Omega, d_\Omega).
\end{align*}
Finally by Th{\'e}or{\`e}me 7.2 in~\cite{C1993}, for any $p_0 \in \Omega$ there exists $C \geq 1$ such that
\begin{align*}
 \frac{1}{C} e^{H_{\rho} r} \leq \# \left\{ \gamma \in \Gamma : d_\Omega(p_0, \rho(\gamma) p_0) \leq r \right\} \leq C e^{H_{\rho} r}.
\end{align*}

\end{proof}

\subsection{Transferring to the Riemannian setting} 

Associated to every properly convex domain $\Omega \subset \Pb(\Rb^{d+1})$ is a Riemannian distance $B_\Omega$ on $\Omega$ called the \emph{Blaschke distance} (see, for instance, \cite{Lof2001,BH2013}). This Riemannian distance is $\Aut(\Omega)$-invariant and by a result of Calabi~\cite{C1972} has Ricci curvature bounded below by $-(d-1)$. Since the Ricci curvature is bounded below by $-(d-1)$, the Bishop-Gromov volume comparison theorem implies that 
\begin{align*}
h_{vol}(\Omega, B_\Omega) \leq d-1.
\end{align*}

Benz{\'e}cri's  theorem (see Theorem~\ref{thm:bens_thm}) provides a simple proof that the Hilbert distance and the Blaschke distance are bi-Lipschitz (see for instance~\cite[Section 9.2]{M2014}) and Tholozan recently proved the following refined relationship between the two distances:

\begin{theorem}\cite{Tho2015}\label{thm:tho}
If $\Omega \subset \Pb(\Rb^{d+1})$ is a properly convex domain, then 
\begin{equation*}
B_\Omega < d_\Omega +1.
\end{equation*}
In particular,  if $\Gamma \leq \Aut(\Omega)$ is a discrete group, then 
\begin{equation*}
\delta_\Gamma(\Omega, d_\Omega) \leq \delta_\Gamma(\Omega, B_\Omega) \leq h_{vol}(\Omega, B_\Omega) \leq d-1.
\end{equation*}
\end{theorem}

\subsection{Rigidity in the Riemannian setting} The Bishop-Gromov volume comparison theorem implies that amongst the class of Riemannian $d$-manifolds with $\Ric \geq -(d-1)$ the volume growth entropy is maximized when $(X,g)$ is isometric to real hyperbolic $d$-space. There are many other examples which maximize volume growth entropy, but if $X$ has ``enough'' symmetry then it is reasonable to suspect that $h_{vol}(X,g)=d-1$ if and only if $X$ is isometric to real hyperbolic $d$-space. This was proved by Ledrappier and Wang when $X$ covers a compact manifold: 

\begin{theorem}[Ledrappier-Wang \cite{LW2010}] Let $(X,g)$ be a complete simply connected Riemannian $d$-manifold with $\Ric \geq -(d-1)$. Suppose that $X$ is the Riemannian universal cover of a compact manifold. Then $h_{vol}(X,g)=d-1$ if and only if $X$ is isometric to real hyperbolic $d$-space. 
\end{theorem}

Later Liu~\cite{L2011} provided an alternative proof of Ledrappier and Wang's result and Liu's argument can be adapted  to prove the following.

\begin{proposition}\label{prop:liu} Let $(X,g)$ be a complete simply connected Riemannian $d$-manifold with $\Ric \geq -(d-1)$ and bounded sectional curvature. Suppose $\Gamma \leq \Isom(X,g)$ is a discrete subgroup and there exist $C,r_0>0$ and $x_0 \in X$ such that
\begin{align*}
C e^{(d-1)r} \leq \#\{ \gamma \in \Gamma : d_X(x_0, \gamma x_0) \leq r \}
\end{align*}
 for every $r > r_0$. Then $X$ is isometric to real hyperbolic $d$-space. 
\end{proposition}

We will prove this result in Section~\ref{sec:Lius_argument} of the appendix.

\subsection{Proof of Theorem~\ref{thm:ent_rig}}

Suppose $\Gamma$ is a finitely generated word hyperbolic group, $\rho:\Gamma \rightarrow \PGL_{d+1}(\Rb)$ is an irreducible projective Anosov representation, and $\rho(\Gamma)$ preserves a properly convex domain in $\Pb(\Rb^{d+1})$. Using Theorem~\ref{thm:proper_action_implies}, there exists a properly convex domain  $\Omega \subset \Pb(\Rb^{d+1})$ such that $\rho(\Gamma) \leq \Aut(\Omega)$ is a regular convex cocompact subgroup. 

Combining Proposition~\ref{prop:CK} and Theorem~\ref{thm:tho} we see that 
\begin{align*}
H_\rho = \delta_{\rho(\Gamma)}(\Omega, d_\Omega) \leq d-1.
\end{align*}

Now suppose that $H_{\rho} = d-1$. By Theorem~\ref{thm:tho} and Proposition~\ref{prop:CK} there exists some $C_0 > 0$ such that 
\begin{align*}
 C_0 e^{(d-1)r} \leq \# \left\{ \gamma \in \Gamma : B_\Omega(p_0, \rho(\gamma) p_0) \leq r \right\}
\end{align*}
for all $r \geq 0$.  Moreover, Benz{\'e}cri's theorem implies that $B_\Omega$ has bounded sectional curvature (see for instance~\cite[Lemma 3.1]{BMZ2015}). So by Proposition~\ref{prop:liu}, $(\Omega, B_\Omega)$ is isometric to the real hyperbolic $d$-space. Hence $(\Omega, d_{\Omega})$ is projectively equivalent to the Klein-Beltrami model of hyperbolic space (see~\cite{VLS1991}). In particular, by conjugating $\rho(\Gamma)$ we can assume 
\begin{align*}
\Omega = \left\{ [1:x_1 : \dots : x_d] \in \Pb(\Rb^d) : \sum_{i=1}^d x_i^2 < 1\right\}
\end{align*}
and $\Aut(\Omega)=\PO(1,d)$. Then $\rho(\Gamma)$ is a convex cocompact subgroup of $\PO(1,d)$ in the classical sense. 

Since
\begin{align*}
\delta_{\rho(\Gamma)}(\Omega, d_\Omega) = d-1,
\end{align*}
Theorem D in~\cite{T1984} implies that 
\begin{align*}
\partial \Omega \cap \overline{\Cc} = \partial \Omega.
\end{align*}
Then since $\Cc$ is convex we see that $\Cc=\Omega$. Then since $\rho(\Gamma) \backslash \Cc= \rho(\Gamma) \backslash \Omega$ is compact, we see that $\rho(\Gamma)\leq \PO(1,d)$ is a co-compact lattice.

\section{Regularity rigidity}\label{sec:reg_rigid}

In this section we will prove Theorems~\ref{thm:main_reg_rigid_intro} and~\ref{thm:main_reg_rigid_hitchin_intro} from the introduction. The proof of both theorems are based on the following observation.

\begin{observation}\label{obs:dynamics} Suppose $g \in \PGL_{d}(\Rb)$ is proximal and $\ell_g^+ \in \Pb(\Rb^d)$ is the eigenline of $g$ corresponding to the eigenvalue of largest absolute value. Let $d_{\Pb}$ is a distance on $\Pb(\Rb^{d})$ induced by a Riemannian metric. If $v \neq \ell^+_g$ and $g^n v \rightarrow \ell^+_g$, then
\begin{align*}
\log\frac{\lambda_2(g)}{\lambda_1(g)} \geq \limsup_{n \rightarrow \infty} \frac{1}{n} \log d_{\Pb}\Big(g^n v, \ell^+_g \Big).
\end{align*}
Moreover, there exists a proper subspace $V \subset \Pb(\Rb^d)$ such that: if $v \in \Pb(\Rb^{d}) \setminus V$ and
$g^n v \rightarrow \ell^+_g$, then
\begin{align*}
\log\frac{\lambda_2(g)}{\lambda_1(g)} = \lim_{n \rightarrow \infty} \frac{1}{n} \log d_{\Pb}\Big(g^n v, \ell^+_g \Big).
\end{align*}
\end{observation}

We give a proof of the observation in Appendix~\ref{app:linear_maps}.

\subsection{Proof of Theorem~\ref{thm:main_reg_rigid_intro}} We begin by recalling the theorem. 

\begin{theorem}\label{thm:main_reg_rigid} Suppose $d>2$, $\Gamma$ is a word hyperbolic group, and $\rho:\Gamma \rightarrow \PGL_{d}(\Rb)$ is an irreducible projective Anosov representation with boundary map $\xi: \partial \Gamma \rightarrow \Pb(\Rb^{d})$. If 
\begin{enumerate}
\item $M = \xi(\partial \Gamma)$ is a $C^2$ $k$-dimensional submanifold of $\Pb(\Rb^{d})$ and
\item the representation $\wedge^{k+1} \rho : \Gamma \rightarrow \PGL(\wedge^{k+1} \Rb^{d})$ is irreducible,
\end{enumerate}
then
\begin{align*}
\frac{\lambda_1(\rho(\gamma))}{\lambda_2(\rho(\gamma))} = \frac{\lambda_{k+1}(\rho(\gamma))}{\lambda_{k+2}(\rho(\gamma))}
\end{align*}
for all $\gamma \in \Gamma$. 
\end{theorem}

For the rest of the subsection, fix a word hyperbolic group $\Gamma$ and a projective Anosov representation $\rho: \Gamma \rightarrow \PGL_{d}(\Rb)$ which satisfy the hypothesis of  Theorem~\ref{thm:main_reg_rigid}.

Define a map $\Phi: M \rightarrow \Pb(\wedge^{k+1} \Rb^{d})$ by 
\begin{align*}
\Phi(m) = [v_1 \wedge v_2 \wedge \dots \wedge v_{k+1}]
\end{align*}
where $T_mM = \Pb(\Spanset_{\Rb}\{ v_1, \dots, v_{k+1}\})$. Since $M$ is a $C^2$ submanifold, $\Phi$ is a $C^1$ map. 

\begin{lemma}\label{lem:C1embedding} With the notation above, $\Phi: M \rightarrow \Pb(\wedge^{k+1} \Rb^{d})$ is a $C^1$ immersion. \end{lemma}

\begin{proof} We break the proof into two cases: when $k = 1$ and when $k >1$. \medskip

\noindent \textbf{Case 1:} Assume $k=1$.  We first consider the case when $d(\Phi)_m = 0$ for every $m \in M$. Then there exists a two dimensional subspace $V \subset \Rb^{d}$ such that $T_mM = \Pb(V)$ for all $m$. Then we must have $M \subset \Pb(V)$. Since $\rho$ is irreducible, the elements in $M$ span $\Rb^{d}$ and so $d \leq 2$. Thus we have a contradiction. So $d(\Phi)_m \neq 0$ on an open set in $M$. But since
\begin{align*}
\Phi \circ \rho(\gamma) =( \wedge^{2} \rho(\gamma)) \circ \Phi
\end{align*}
for every $\gamma \in \Gamma$ and $\Gamma$ acts minimally on $M$, we see that $d(\Phi)_m \neq 0$ for every $m$. \medskip

\noindent \textbf{Case 2:} Assume $k > 1$. Then by Theorem~\ref{thm:condC} there exists a properly convex domain $\Omega \subset \Pb(\Rb^{d})$ such that $\rho(\Gamma) \leq \Aut(\Omega)$ is a regular convex cocompact subgroup. Suppose $\Cc \subset \Omega$ is a closed convex subset such that $g\Cc = \Cc$ for all $g \in \rho(\Gamma)$, $\rho(\Gamma) \backslash \Cc$ is compact, and every point in $\partial \Omega \cap \overline{\Cc}$ is a $C^1$ extreme point of $\Omega$. 

We first claim that $\Phi$ is injective. By Lemma~\ref{lem:dynamics} we have 
\begin{align*}
\xi(\partial \Gamma) \subset \partial \Omega \cap \overline{\Cc}
\end{align*}
and 
\begin{align*}
\eta(\partial \Gamma) \subset \partial\Omega^*.
\end{align*}
Then since $\xi(x)$ is a $C^1$ point of $\partial \Omega$ we have
\begin{align*}
T_{\xi(x)} \partial \Omega = \Pb(\ker \eta(x)).
\end{align*}
Further since $M \subset \partial \Omega$  we see that 
 \begin{align*}
T_{\xi(x)} M \subset T_{\xi(x)} \partial \Omega = \Pb(\ker \eta(x))
 \end{align*}
for every $x \in \partial \Gamma$. Now suppose that $T_{\xi(x)}M =T_{\xi(y)}M$ for some $x,y \in \partial \Gamma$. Then 
\begin{align*}
\xi(x) \in T_{\xi(x)}M  = T_{\xi(y)}M  \subset \Pb(\ker \eta(y)).
\end{align*}
So $x=y$ and hence $\Phi$ is injective. 

Since $\Phi$ is injective and $C^1$, $d(\Phi)$ must have full rank at some point. By continuity, $d(\Phi)$ has full rank on an open set. But since
\begin{align*}
\Phi \circ \rho(\gamma) =( \wedge^{k+1} \rho(\gamma)) \circ \Phi
\end{align*}
for every $\gamma \in \Gamma$ and $\Gamma$ acts minimally on $M$, we see that $d(\Phi)$ has full rank everywhere. Hence, since $M$ is compact and $\Phi$ is injective, $\Phi$ is a $C^1$ embedding. 
\end{proof}

Next fix distances $d_1$ on $\Pb(\Rb^{d})$ and $d_2$ on $\Pb(\wedge^{k+1} \Rb^{d})$ which are induced by Riemannian metrics. Since $\Phi$ is a $C^1$ immersion, there exists $C \geq 1$ such that 
\begin{align}
\label{eq:biLip_PHI}
\frac{1}{C} d_1(m_1, m_2) \leq d_2(\Phi(m_1), \Phi(m_2)) \leq C d_1(m_1, m_2)
\end{align}
for all $m_1, m_2 \in M$ sufficiently close. 

Now fix some $\gamma \in \Gamma$ with infinite order and let $g \in \GL_d(\Rb)$ be a lift of $\rho(\gamma)$ with $\det g=\pm 1$. Suppose that 
\begin{align*}
\lambda_1 \geq \lambda_2 \geq \dots \geq \lambda_{d}
\end{align*}
are the absolute values of the eigenvalues of $g$. Then the absolute values of the eigenvalues of $\wedge^{k+1} g$ have the form 
\begin{align*}
\lambda_{i_1}\lambda_{i_2} \cdots\lambda_{i_{k+1}}
\end{align*}
for $1 \leq i_1 < i_2 < \dots < i_{k+1} \leq d$. In particular, 
\begin{align*}
\lambda_1 \lambda_2 \cdots \lambda_{k+1}
\end{align*}
is the absolute value of the largest eigenvalue of $\wedge^{k+1} g$ and 
\begin{align*}
\lambda_1 \lambda_2 \cdots \lambda_k \lambda_{k+2}
\end{align*}
is the absolute value of the second largest eigenvalue of $\wedge^{k+1} g$. 

Next let $x^+, x^-$ be the attracting and repelling fixed points of $\gamma$ in $\partial \Gamma$. 

\begin{lemma} With the notation above,  $\wedge^{k+1} \rho(\gamma)$ is proximal with attracting fixed point $\Phi(\xi(x^+))$.
\end{lemma}

\begin{proof} We first show that $\Phi(\xi(x^+))$ is an eigenline of $\wedge^{k+1} g$ whose eigenvalue has absolute value $\lambda_1 \cdots \lambda_{k+1}$.

Fix a norm on $\End(\wedge^{k+1} \Rb^{d})$. Then we can find a sequence $n_m \rightarrow \infty$ such that $\frac{1}{\norm{(\wedge^{k+1} g)^{n_m}}}(\wedge^{k+1}g)^{n_m}$ converges to some $T \in \End(\wedge^{k+1} \Rb^{d})$. Then
\begin{align}
\label{eq:limit_of_T}
T(v) = \lim_{m \rightarrow \infty} (\wedge^{k+1} \rho(\gamma))^{n_m} v
\end{align}
for every $v \in \Pb(\wedge^k \Rb^d) \setminus \Pb(\ker T)$. 

By Observation~\ref{obs:span_gen_eigenvectors}, every element in the image of $T$ is a sum of generalized complex eigenvectors of $\wedge^{k+1} g$ whose eigenvalue has maximal absolute value (that is, $\lambda_1 \cdots \lambda_{k+1}$). We will show that the image of $T$ is $\Phi(\xi(x^+))$ and hence $\Phi(\xi(x^+))$ is an eigenline of $\wedge^{k+1} g$ whose eigenvalue has absolute value $\lambda_1 \cdots \lambda_{k+1}$.

Now since $\wedge^{k+1} \rho : \Gamma \rightarrow \PGL(\wedge^{k+1} \Rb^d)$ is irreducible, there exists $x_1, \dots, x_N \in \partial \Gamma$ such that 
\begin{align*}
\Phi(\xi(x_1)), \dots, \Phi(\xi(x_N))
\end{align*}
span $\wedge^{k+1} \Rb^d$. By perturbing the $x_i$ (if necessary) we can also assume that
\begin{align*}
x^- \notin \{x_1, \dots, x_N\}.
\end{align*}
Next by relabelling the $x_i$ we can also assume that there exists $1 \leq m \leq N$ such that
\begin{align*}
\Phi(\xi(x_1)) + \dots + \Phi(\xi(x_m)) + \ker T = \wedge^{k+1} \Rb^d
\end{align*}
and 
\begin{align*}
\Big( \Phi(\xi(x_1)) + \dots + \Phi(\xi(x_m)) \Big) \cap \ker T = (0).
\end{align*}
Then by Equation~\eqref{eq:limit_of_T} 
\begin{align*}
T ( \Phi(\xi(x_i)) ) = \lim_{m \rightarrow \infty} (\wedge^{k+1} \rho(\gamma))^{n_m} \Phi(\xi(x_i)) =  \lim_{m \rightarrow \infty} \Phi( \xi( \gamma^{n_m} x)) = \Phi(\xi(x^+))
\end{align*}
for $1 \leq i \leq m$. So the image of $T$ is $\Phi(\xi(x^+))$ and so Observation~\ref{obs:span_gen_eigenvectors} implies that $\Phi(\xi(x^+))$ is an eigenline of $\wedge^{k+1} g$ whose eigenvalue has absolute value $\lambda_1 \cdots \lambda_{k+1}$.

We next argue that $\wedge^{k+1} \rho(\gamma)$ is proximal. Suppose not, then by Observation~\ref{obs:no_proximal_same_speed} there exists a proper subspace $V \subset \Pb(\wedge^{k+1} \Rb^d)$ such that: if $v \in \Pb(\wedge^{k+1} \Rb^d) \setminus V$, then 
\begin{align*}
0  = \lim_{n \rightarrow \infty} \frac{1}{n} \log d_2\Big((\wedge^{k+1}\rho(\gamma))^{n} v, \Phi(\xi(x^+))\Big).
\end{align*}
Since $\wedge^{k+1} \rho : \Gamma \rightarrow \PGL(\wedge^{k+1} \Rb^d)$ is irreducible there exists $x \in \partial \Gamma$ such that $\Phi(\xi(x)) \notin V$. Then by perturbing $x$ (if necessary) we can also assume that $x \neq x^-$. Then 
\begin{align*}
\rho(\gamma)^n \xi(x) = \xi(\gamma^nx) \rightarrow \xi(x^+) \text{ and } (\wedge^{k+1}\rho(\gamma))^{n} \Phi(\xi(x)) =\Phi(\xi(\gamma^nx))\rightarrow \Phi(\xi(x^+)).
\end{align*}
So by Observation~\ref{obs:dynamics} applied to $\rho(\gamma)$
\begin{align*}
0 > \log \frac{\lambda_2}{\lambda_1} & \geq \limsup_{n \rightarrow \infty} \frac{1}{n} \log d_1\Big(\rho(\gamma)^n \xi(x), \xi(x^+)\Big) \\
&= \limsup_{n \rightarrow \infty} \frac{1}{n} \log d_1\Big(\xi(\gamma^nx), \xi(x^+)\Big) \\
&=  \limsup_{n \rightarrow \infty} \frac{1}{n} \log d_2\Big( \Phi(\xi(\gamma^nx)),\Phi( \xi(x^+))\Big) \\
& = \limsup_{n \rightarrow \infty} \frac{1}{n} \log d_2\Big( ( \wedge^{k+1} \rho(\gamma) )^n\Phi(\xi(x)),\Phi( \xi(x^+))\Big) =0.
\end{align*}
Notice that we used Equation~\eqref{eq:biLip_PHI} in the second equality. So we have a contradiction and hence $\wedge^{k+1} \rho(\gamma)$ is proximal.
\end{proof}

By Observation~\ref{obs:dynamics}, there exists a proper subspace $V_1 \subset \Pb(\Rb^{d})$ such that
\begin{align*}
\log \frac{\lambda_2}{\lambda_1} = \lim_{n \rightarrow \infty} \frac{1}{n} \log d_1\Big(\rho(\gamma)^n v, \xi(x^+)\Big)
\end{align*}
for all $v \in \Pb(\Rb^d) \setminus V_1$ with $\rho(\gamma)^n v \rightarrow \xi(x^+)$. By the same observation, there exists a proper subspace $V_2 \subset  \Pb(\wedge^{k+1} \Rb^{d+1})$ such that 
\begin{align*}
\log \frac{\lambda_{k+2}}{\lambda_{k+1}}=\log \frac{\lambda_1 \lambda_2 \cdots \lambda_k \lambda_{k+2}}{\lambda_1 \lambda_2 \cdots \lambda_{k+1}} = \lim_{n \rightarrow \infty} \frac{1}{n} \log d_2\Big((\wedge^{k+1}\rho(\gamma))^{n} w, \Phi(\xi(x^+_\gamma))\Big)
\end{align*}
for all $w \in  \Pb(\wedge^{k+1} \Rb^{d}) \setminus V_2$ with $(\wedge^{k+1} \rho(\gamma))^n w \rightarrow  \Phi(\xi(x^+_\gamma))$. 

Since $\rho$ is irreducible, $\{ \xi(x) : x\in \partial \Gamma\}$ spans $\Rb^d$. So we can pick some $x \in \partial \Gamma$ such that $\xi(x) \notin V_1$. By perturbing $x$ (if necessary) we can also assume that $x \neq x^-$. Then $\gamma^n x \rightarrow x^+$ and so 
\begin{align*}
\log \frac{\lambda_2}{\lambda_1} 
&=  \lim_{n \rightarrow \infty} \frac{1}{n} \log d_1\Big(\rho(\gamma)^n \xi(x), \xi(x^+)\Big)
= \lim_{n \rightarrow \infty} \frac{1}{n} \log d_1\Big( \xi(\gamma^n x), \xi(x^+)\Big)\\
&=  \lim_{n \rightarrow \infty} \frac{1}{n} \log d_2\Big( \Phi(\xi(\gamma^n x)),\Phi( \xi(x^+))\Big)\\
& = \lim_{n \rightarrow \infty} \frac{1}{n} \log d_2\Big( ( \wedge^{k+1} \rho(\gamma) )^n\Phi(\xi(x)),\Phi( \xi(x^+))\Big). 
\end{align*}
Notice that we used Equation~\eqref{eq:biLip_PHI} in the third equality. Then applying Observation~\ref{obs:dynamics} to $\wedge^{k+1} \rho(\gamma)$ we have
\begin{align*}
\log \frac{\lambda_2}{\lambda_1}  \leq \log \frac{\lambda_{k+1}}{\lambda_k}.
\end{align*}

We prove the opposite inequality in exactly the same way. Since $\wedge^{k+1} \rho$ is irreducible, $\{ \Phi(\xi(x)) : x\in \partial \Gamma\}$ spans $\Rb^d$. So we can pick some $x \in \partial \Gamma$ such that $\Phi(\xi(x)) \notin V_2$. By perturbing $x$ (if necessary) we can assume that $x \neq x^-$. Then $\gamma^n x \rightarrow x^+$ and so
\begin{align*}
\log \frac{\lambda_{k+1}}{\lambda_k}  
&= \lim_{n \rightarrow \infty} \frac{1}{n} \log d_2\Big( ( \wedge^{k+1} \rho(\gamma) )^n\Phi(\xi(x)),\Phi( \xi(x^+))\Big) \\
&=  \lim_{n \rightarrow \infty} \frac{1}{n} \log d_2\Big( \Phi(\xi(\gamma^n x)),\Phi( \xi(x^+))\Big)= \lim_{n \rightarrow \infty} \frac{1}{n} \log d_1\Big( \xi(\gamma^n x), \xi(x^+)\Big)\\
& =  \lim_{n \rightarrow \infty} \frac{1}{n} \log d_1\Big(\rho(\gamma)^n \xi(x), \xi(x^+)\Big) \leq \log \frac{\lambda_2}{\lambda_1}.
\end{align*}
Hence 
\begin{align*}
\frac{\lambda_2}{\lambda_1}  = \frac{\lambda_{k+1}}{\lambda_k}
\end{align*}
and since $\gamma \in \Gamma$ was an arbitrary element with infinite order this proves the theorem. 

 \subsection{Proof of Theorem~\ref{thm:main_reg_rigid_hitchin_intro}} We begin by recalling the theorem. 
 
 \begin{theorem}\label{thm:main_reg_rigid_hitchin} Suppose that $\Gamma \leq \PSL_2(\Rb)$ is a torsion-free cocompact lattice  and  $\rho : \Gamma \rightarrow \PSL_d(\Rb)$ is in the Hitchin component. If $\xi: \partial \Gamma \rightarrow \Pb(\Rb^{d})$ is the associated boundary map and $\xi(\partial \Gamma)$ is a $C^2$ submanifold of $\Pb(\Rb^d)$, then 
\begin{align*}
\frac{\lambda_1(\rho(\gamma))}{\lambda_2(\rho(\gamma))} = \frac{\lambda_2(\rho(\gamma))}{\lambda_3(\rho(\gamma))}
\end{align*}
for all $\gamma \in \Gamma$. 
\end{theorem}

For the rest of the section suppose that $\Gamma \leq \PSL_2(\Rb)$ is a torsion-free cocompact lattice  and  $\rho : \Gamma \rightarrow \PSL_d(\Rb)$ is in the Hitchin component.
 
 Let $\Fc(\Rb^d)$ denote the full flag manifold of $\Rb^d$. Then by Theorem 4.1 and Proposition 3.2 in~\cite{L2006} there exists a continuous, $\rho$-equivariant map $F=(\xi^{(1)},\dots, \xi^{(d)}): \partial \Gamma \rightarrow \Fc(\Rb^d)$ such that:
 \begin{enumerate}
 \item\label{item:hitchin1} $\xi=\xi^{(1)}$.
 \item\label{item:hitchin2} If $x,y,z \in \partial \Gamma$ are distinct, $k_1,k_2,k_3 \geq 0$, and $k_1+k_2+k_3 =d$, then 
 \begin{align*}
 \xi^{(k_1)}(x) + \xi^{(k_2)}(y) + \xi^{(k_3)}(z) = \Rb^d
 \end{align*}
 is a direct sum.
  \item\label{item:hitchin3} If $x,y,z \in \partial \Gamma$ are distinct and $0\leq k < d-2$, then
 \begin{align*}
 \xi^{(k+1)}(y) + \xi^{(d-k-2)}(x) + \Big(\xi^{(k+1)}(z) \cap \xi^{(d-k)}(x) \Big)= \Rb^d
 \end{align*}
 is a direct sum.
 \item\label{item:hitchin4} If $\gamma \in \Gamma \setminus \{1\}$, then the absolute values of the eigenvalues of $\rho(\gamma)$ satisfy
 \begin{align*}
 \lambda_1(\rho(\gamma)) > \dots > \lambda_d(\rho(\gamma)).
 \end{align*}
 \item\label{item:hitchin5} If $\gamma \in \Gamma \setminus \{1\}$ and $x^+_\gamma \in \partial \Gamma$ is the attracting fixed point of $\gamma$, then $ \xi^{(k)}(x^+_\gamma)$ is the span of the eigenspaces of $\rho(\gamma)$ corresponding to the eigenvalues 
 \begin{align*}
  \lambda_1(\rho(\gamma)), \dots, \lambda_k(\rho(\gamma)).
  \end{align*} 
\end{enumerate}

Throughout the following argument we will identify a $k$-dimensional subspace $W = \Spanset \{ w_1, \dots, w_k\}$ of $\Rb^d$ with the point $[w_1 \wedge \dots \wedge w_k] \in \Pb(\wedge^k \Rb^d)$. 

Next fix distances $d_1$ on $\Pb(\Rb^{d})$ and $d_2$ on $\Pb(\wedge^{2} \Rb^{d})$ which are induced by Riemannian metrics.

\begin{lemma}\label{lem:hithcin_eigenvalue_est} With the notation above, if $\gamma \in \Gamma \setminus \{1\}$ and $x \in \partial \Gamma \setminus \{x^+_\gamma, x^-_\gamma\}$, then 
\begin{align*}
\log \frac{\lambda_2(\rho(\gamma))}{\lambda_1(\rho(\gamma))} = \lim_{n \rightarrow \infty} \frac{1}{n} \log d_1\Big(\xi(\gamma^nx), \xi(x^+_\gamma)\Big)
\end{align*}
and 
\begin{align*}
\log \frac{\lambda_3(\rho(\gamma))}{\lambda_2(\rho(\gamma))} =\lim_{n \rightarrow \infty} \frac{1}{n} \log d_2\Big(\xi^{(2)}(\gamma^nx), \xi^{(2)}(x^+_\gamma)\Big)
\end{align*}
\end{lemma}

\begin{proof} Fix $\gamma \in \Gamma \setminus \{1\}$ and let $\lambda_i = \lambda_i(\rho(\gamma))$. Then let $v_1, \dots, v_d \in \Rb^d$ be eigenvectors of $\rho(\gamma)$ corresponding to $\lambda_1, \dots, \lambda_d$. Then by Property~\eqref{item:hitchin5}
\begin{align*}
\xi^{(k)}(x_\gamma^+) = \Spanset\{ v_1,\dots, v_k\}
\end{align*}
and
 \begin{align*}
\xi^{(k)}(x_\gamma^-) = \Spanset\{ v_{d-k+1},\dots, v_d\}.
\end{align*}
Further, if $w \notin \Spanset\{ v_1,v_3,\dots, v_d\}$ then 
\begin{align*}
\log \frac{\lambda_2}{\lambda_1} = \lim_{n \rightarrow \infty} \frac{1}{n} \log d_1\Big(\rho(\gamma)^n[w], \xi(x^+_\gamma)\Big).
\end{align*}
Notice, if $x \in \partial \Gamma \setminus \{x^+_\gamma, x^-_\gamma\}$ then Property~\eqref{item:hitchin2} implies that 
\begin{align*}
\xi(x) \notin \Pb(\xi(x^+_\gamma) \oplus \xi^{(d-2)}(x^-_\gamma)) = \Pb(\Spanset\{ v_1,v_3,\dots, v_d\})
\end{align*}
and so 
\begin{align*}
\log \frac{\lambda_2}{\lambda_1} = \lim_{n \rightarrow \infty} \frac{1}{n} \log d_1\Big(\xi(\gamma^nx), \xi(x^+_\gamma)\Big).
\end{align*}

For the second equality, notice that $v_i \wedge v_j$ are  eigenvectors of $\wedge^2 \rho(\gamma)$. So $\lambda_1 \lambda_2$ is the absolute value of the largest eigenvalue of $\wedge^2 \rho(\gamma)$ and  $\lambda_1 \lambda_3$ is the absolute value of the second largest eigenvalue of $\wedge^2 \rho(\gamma)$. So if 
\begin{align*}
w \notin \Spanset\{ v_i \wedge v_j  : \{i,j\} \neq \{ 1,3\} \},
\end{align*}
 then 
\begin{align*}
\log \frac{\lambda_3 }{\lambda_2} =\log \frac{\lambda_1\lambda_3 }{\lambda_1\lambda_2} = \lim_{n \rightarrow \infty} \frac{1}{n} \log d_2\Big( (\wedge^2\rho(\gamma))^n[w], \xi^{(2)}(x^+_\gamma)\Big).
\end{align*}

Now we claim that $\xi^{(2)}(x)  \notin \Pb(\Spanset\{ v_i \wedge v_j  : \{i,j\} \neq \{ 1,3\} \})$ when $x \in \partial \Gamma \setminus \{x^+_\gamma, x^-_\gamma\}$. Suppose that $\xi^{(2)}(x) = [ w_1 \wedge w_2]$
where 
\begin{align*}
w_1 = \sum_{i=1}^d \alpha_i v_i \text{ and } w_2 = \sum_{i=1}^d \beta_i v_i.
\end{align*}
Then 
\begin{align*}
\xi^{(2)}(x) = \left[ \sum_{1 \leq i < j \leq d} (\alpha_i \beta_j - \alpha_j \beta_i) v_i \wedge v_j \right].
\end{align*}
Now Property~\eqref{item:hitchin3} implies that 
\begin{align*}
\xi^{(2)}(x) +\xi^{(d-3)}(x^-_\gamma) + \Big(\xi^{(2)}(x^+_\gamma) \cap \xi^{(d-1)}(x^-_\gamma) \Big) = \Rb^d
\end{align*}
is direct. Since 
\begin{align*}
\xi^{(d-3)}(x^-_\gamma) + \Big(\xi^{(2)}(x^+_\gamma) \cap \xi^{(d-1)}(x^-_\gamma) \Big)= \Spanset\{v_2, v_4,\dots, v_d\}
\end{align*}
we see that $\alpha_1 \beta_3 - \alpha_3 \beta_1 \neq 0$. Thus $\xi^{(2)}(x)  \notin \Pb(\Spanset\{ v_i \wedge v_j  : \{i,j\} \neq \{ 1,3\} \})$. So
\begin{equation*}
\log \frac{\lambda_3 }{\lambda_2} = \lim_{n \rightarrow \infty} \frac{1}{n} \log d_2\Big( \xi^{(2)}(\gamma^n x), \xi^{(2)}(x^+_\gamma)\Big).\qedhere
\end{equation*}
\end{proof}

Now assume that $M = \xi(\partial \Gamma)$ is a $C^2$ submanifold in $\Pb(\Rb^d)$. Then define a map $\Phi: M \rightarrow \Pb(\wedge^{2} \Rb^{d})$ by
\begin{align*}
\Phi(m) = [v_1 \wedge v_2]
\end{align*}
where $T_mM = \Pb(\Spanset_{\Rb}\{ v_1,v_2\})$. Since $M$ is a $C^2$ submanifold, $\Phi$ is a $C^1$ map. 

\begin{lemma} With the notation above, $\Phi(\xi(x)) = \xi^{(2)}(x)$ for all $x \in \partial \Gamma$.
\end{lemma}

\begin{proof} Since $\{ x^+_\gamma : \gamma \in \Gamma \setminus \{1\} \}$ is dense in $\partial \Gamma$, it is enough to show that $\Phi(\xi(x_\gamma^+)) = \xi^{(2)}(x_\gamma^+)$ for $\gamma \in \Gamma \setminus \{1\}$. By property (5) above, $\xi^{(k)}(x_\gamma^+)$ is the span of the eigenspaces of $\rho(\gamma)$ corresponding to the eigenvalues 
 \begin{align*}
  \lambda_1(\rho(\gamma)),\dots,  \lambda_k(\rho(\gamma))
  \end{align*} 
while $\xi^{(k)}(x_\gamma^-)$ is the span of the eigenspaces of $\rho(\gamma)$ corresponding to the eigenvalues 
 \begin{align*}
  \lambda_{d-k+1}(\rho(\gamma)), \dots, \lambda_d(\rho(\gamma)).
  \end{align*} 
  
  Now fix $y \in \partial \Gamma \setminus \{x^+_\gamma, x^-_\gamma\}$. By Properties~\eqref{item:hitchin1} and~\eqref{item:hitchin2}, 
  \begin{align*}
  \xi(y) \notin \Pb(\xi(x_\gamma^+) \oplus \xi^{(d-2)}(x_\gamma^-))
  \end{align*}
  and so $\xi(\gamma^n y) = \rho(\gamma)^n \xi(y)$ approaches $\xi(x^+_\gamma)$ along an orbit tangential to $\xi^{(2)}(x^+_\gamma)$. Which implies that $\Phi(\xi(x^+_\gamma)) = \xi^{(2)}(x^+_\gamma)$.
\end{proof}

\begin{lemma}
With the notation above, $\Phi: M \rightarrow \Pb(\wedge^{2} \Rb^{d})$ is a $C^1$ embedding. 
\end{lemma}

\begin{proof} By the previous lemma and Property~\eqref{item:hitchin2}, $\Phi$ is injective. Since $\Phi$ is also $C^1$, $d(\Phi)_m \neq 0$ for some $m \in M$. So $d(\Phi)_m \neq 0$ on an open set. But since
\begin{align*}
\Phi \circ \rho(\gamma) =( \wedge^{2} \rho(\gamma)) \circ \Phi
\end{align*}
for every $\gamma \in \Gamma$ and $\Gamma$ acts minimally on $M$, we see that $d(\Phi)_m \neq 0$ for all $m \in M$. Hence, since $M$ is compact and $\Phi$ is injective, $\Phi$ is a $C^1$ embedding. 
\end{proof}

Since $\Phi$ is a $C^1$ embedding, there exists $C \geq 1$ such that 
\begin{align*}
\frac{1}{C} d_1(m_1, m_2) \leq d_2(\Phi(m_1), \Phi(m_2)) \leq C d_1(m_1, m_2)
\end{align*}
for all $m_1, m_2 \in M$. Then by Lemma~\ref{lem:hithcin_eigenvalue_est} we have 
\begin{align*}
\frac{\lambda_3(\rho(\gamma))}{\lambda_2(\rho(\gamma))} = \frac{\lambda_2(\rho(\gamma))}{\lambda_1(\rho(\gamma))}
\end{align*}
for all $\gamma \in \Gamma$.

\appendix 

\section{An argument of Liu}\label{sec:Lius_argument}

In this section we explain how an argument of Liu~\cite{L2011} can be adapted to prove the following. 

\begin{proposition} Let $(X,g)$ be a complete simply connected Riemannian $d$-manifold with $\Ric \geq -(d-1)$ and bounded sectional curvature. Suppose $\Gamma \leq \Isom(X,g)$ is a discrete subgroup and there exist $C,r_0>0$ and $x_0 \in X$ such that
\begin{align*}
C e^{(d-1)r} \leq \#\{ \gamma \in \Gamma : d_X(x_0, \gamma x_0) \leq r \}
\end{align*}
 for every $r > r_0$. Then $X$ is isometric to real hyperbolic $d$-space. 
\end{proposition}

Essentially the only change in Liu's argument is replacing the words ``by a standard covering technique'' with the proof of Lemma~\ref{lem:Liu2} below. 

Suppose for the rest of the section that $(X,g)$ is a Riemannian manifold and $\Gamma \leq \Isom(X,g)$ is a discrete subgroup which satisfy the hypothesis of the theorem. Let $d_X:X \rightarrow X \rightarrow \Rb$ be the distance, ${\rm Vol}$ denote the volume form, $\nabla$ denote the gradient, and let $\Delta$ denote the Laplace-Beltrami operator on $(X,g)$. Also, for $x \in X$ and $r > 0$ define
\begin{align*}
B_r(x) = \{ y \in X : d_X(x,y) < r\}.
\end{align*}

We begin by recalling a result of Ledrappier and Wang. 

\begin{lemma}\cite{LW2010}\label{lem:LW} If there exists a $C^\infty$ function $u: X \rightarrow \Rb$ such that $\norm{\nabla u} \equiv 1$ and $\Delta u \equiv d-1$, then $X$ is isometric to real hyperbolic space.
\end{lemma}

\begin{proof} Define $\phi = e^{(d-1)u}$. Then $\phi$ is positive and by the chain rule
\begin{align*}
\Delta(\phi) = e^{(d-1)u}\left( (d-1)^2\norm{\nabla u}^2 - (d-1) \Delta u\right) = 0.
\end{align*}
Further, $\norm{\nabla \log \phi} =(d-1)\norm{\nabla u} \equiv d-1$. So by Theorem 6 in~\cite{LW2010}, $X$ is isometric to real hyperbolic space.
\end{proof}

Next fix  a point $x_0 \in X$ and some very large $R >0$. Let $d_0 : X \rightarrow \Rb$ be the function $d_0(x) = d_X(x,x_0)$. Next let $\Cc_{0} \subset X$ denote the cut locus of $x_0$. Then $d_0$ is smooth on $X \setminus (\Cc_{0}\cup \{x_0\})$ and $\Vol(\Cc_{0}) = 0$. 

\begin{lemma}\label{lem:Liu1} There exists $r_n \rightarrow \infty$ such that: if 
\begin{align*}
A_n = \{ x \in X : r_n-50R \leq d_X(x_0, x) \leq r_n + 50R \},
\end{align*}
then 
\begin{align*}
\lim_{n \rightarrow \infty} \frac{1}{\Vol(A_n)} \int_{A_n \setminus \Cc_{0}} \Delta d_0(x)dV = d-1.
\end{align*}
\end{lemma}

\begin{proof} This is essentially claim 1 and claim 2 from~\cite{L2011}. First, the Laplacian comparison theorem (see Theorem~\cite[Theorem 2.2]{Zhu1997}) immediately implies that 
\begin{align*}
\limsup_{n \rightarrow \infty} \frac{1}{\Vol(A_n)} \int_{A_n \setminus \Cc_{0}} \Delta d_0(x) dV \leq d-1
\end{align*}
and so we just have to prove 
\begin{align*}
\liminf_{n \rightarrow \infty} \frac{1}{\Vol(A_n)} \int_{A_n \setminus \Cc_{0}} \Delta d_0(x) dV \geq d-1.
\end{align*}

Let $S_{x_0} X$ denote the unit tangent sphere at $x_0$. For $v \in S_{x_0} X$ let 
\begin{align*}
 \tau(v) = \min\{ t > 0 : \exp_{x_0}(tv) \in \Cc_0\}.
\end{align*}
Next for $r >0$ define 
\begin{align*}
 C(r) = \{ v \in S_{x_0} X : r < \tau(v) \}.
\end{align*}
Let $J(r,v)$ be the non-negative function defined on $\cup_{r > 0} \{r \} \times C(r)$ such that: if $\varphi \in L^1(X, dV)$, then 
\begin{align*}
 \int_X \varphi(x) dV = \int_0^\infty \int_{C(r)} \varphi(\exp_{x_0}(rv) ) J(r,v) d\mu(v)dr
\end{align*}
where $d\mu$ is the Lebesgue meaure on $S_{x_0}X$. 

For $r > 0$ let 
\begin{align*}
 S_r = \int_{C(r)} J(r,v) d\mu(v).
\end{align*}
Then by Fubini's theorem
\begin{align}\label{eq:Fubini}
 \int_0^R S_r dr= { \rm Vol}(B_g(x_0, R))
\end{align}
for every $R > 0$. We claim that there exists $r_n \rightarrow \infty$ such that 
\begin{align*}
 \liminf_{n \rightarrow \infty} \frac{S_{r_n+50R}}{S_{r_n-50R}} \geq e^{100(d-1)R}.
\end{align*}
Suppose such a sequence does not exist, then there exists $\epsilon > 0$ and $R_0 > 0$ such that 
\begin{align*}
\frac{S_{r+50R}}{S_{r-50R}} \leq e^{100(d-1)R}(1-\epsilon)
\end{align*}
for every $r > R_0$. But then an iteration argument implies that 
\begin{align*}
S_{r} \leq C(1-\epsilon)^{\frac{r}{100R}} e^{(d-1)r}
\end{align*}
for some $C > 0$ which is independent of $r$. But then Equation~\eqref{eq:Fubini} implies that $h_{vol}(X,g) < (d-1)$. So we have a contradiction and hence there exists $r_n \rightarrow \infty$ such that 
\begin{align*}
 \liminf_{n \rightarrow \infty} \frac{S_{r_n+50R}}{S_{r_n-50R}} \geq e^{100(d-1)R}.
\end{align*}

Next for $v \in S_{x_0}X$ and $r \in(0,\tau(v))$, define $H(r,v) = (\Delta d_0)( \exp_p(rv))$. We have the following well known relationship between $J$ and $H$, see for instance~\cite[Equation 1.159]{CLN2006},
\begin{align}\label{eq:H_J_relationship}
 H(r,v) J(r,v) = \frac{\partial}{\partial r} J(r,v). 
\end{align}
Next define
\begin{align*}
a_n(v): = \min\{\tau(v), r_n-50R\} \text{ and } b_n(v) := \min\{ r_n+50R, \tau(v) \}.
\end{align*}
Then by Equation~\eqref{eq:H_J_relationship}
\begin{align*}
  \int_{A_n \setminus \Cc_{0}} & \Delta d_0(x) dV  = \int_{S_{x_0} X} \int_{a_n(v)}^{b_n(v)} H(r,v)J(r,v) dr d\mu(v) \\
  &= \int_{S_{x_0} X} \int_{a_n(v)}^{b_n(v)} \frac{\partial J}{\partial r}(r,v) dr d\mu(v)  = \int_{S_{x_0} X} J(b_n(v), v) - J(a_n(v), v) d \mu(v) \\
  & = S_{r_n+50R}-S_{r_n-50R} + \int_{\{r_n-50R < \tau(v) < r_n+50R\}}   J(b_n(v), v) d \mu(v)\\
  & \geq S_{r_n+50R} - S_{r_n - 50R}. 
\end{align*}
By using the volume comparison theorem for annuli, see~\cite[Theorem 3.1]{Zhu1997}, we have
\begin{align*}
 \limsup_{n \rightarrow \infty} \frac{\Vol(A_n)}{S_{r_n-50R_n}} \leq \frac{1}{d-1}\left( e^{100(d-1)R_n}-1 \right) 
\end{align*}
and so
\begin{align*}
\liminf_{n \rightarrow \infty} & \frac{1}{\Vol(A_n)} \int_{A_n \setminus \Cc_{0}} \Delta d_0(x) dV  \geq \liminf_{n \rightarrow \infty} \frac{S_{r_n+50R} - S_{r_n - 50R}}{\Vol(A_n)} \\
& \geq \left(\liminf_{n \rightarrow \infty} \frac{S_{r_n+50R}}{S_{r_n-50R}}-1 \right)\liminf_{n \rightarrow \infty} \frac{S_{r_n-50R}}{\Vol(A_n)} \\
& \geq d-1.
\end{align*}
\end{proof}

Next let $M_n \subset \Gamma$ be a maximal set such that
\begin{enumerate}
\item if $\gamma \in M_n$, then $\gamma B_R( x_0) \subset A_n$, 
\item if $\gamma_1, \gamma_2 \in M_n$ are distinct, then 
\begin{align*}
\gamma_1 B_R(x_0) \cap \gamma_2B_R(x_0) = \emptyset.
\end{align*}
\end{enumerate}
Then let 
\begin{align*}
E_n = M_n \cdot B_R(x_0) \subset A_n.
\end{align*}

\begin{lemma}\label{lem:Liu2} For $R>0$ sufficiently large, 
\begin{align*}
 \liminf_{n \rightarrow \infty}\frac{\Vol(A_n)}{e^{(d-1)r_n}} >0
\end{align*}
and
\begin{align*}
\liminf_{n \rightarrow \infty}\frac{ \Vol(E_n)}{\Vol(A_n)} >0.
\end{align*}
\end{lemma}

\begin{proof} We prove the second inequality first. Fix some $\delta \in (0,R)$ such that: if $\gamma_1, \gamma_2 \in \Gamma$ and
\begin{align*}
\gamma_1 B_\delta(x_0) \cap \gamma_2 B_{\delta}(x_0) \neq \emptyset,
\end{align*}
then $\gamma_1 x_0= \gamma_2x_0$. Then let $s_0 = \#\{ \gamma \in \Gamma : \gamma x_0 = x_0\}$ and
\begin{align*}
N_n = \{ \gamma \in \Gamma : r_n -49R < d(\gamma x_0,x_0) \leq r_n+49R \}
\end{align*}
Then 
\begin{align*}
\Vol( N_n \cdot B_{\delta}(x_0)) = \frac{\Vol(B_\delta(x_0))}{s_0} \# N_n.
\end{align*}
Moreover, since $M_n$ was chosen maximally, we have
\begin{align*}
N_n \cdot B_{\delta}(x_0) \subset M_n \cdot B_{2R+\delta}(x_0).
 \end{align*}
Then since 
\begin{align*}
\frac{\Vol (M_n \cdot B_{2R+\delta}(x_0)) }{\Vol (M_n \cdot B_{R}(x_0))} \leq \frac{\Vol (B_{2R+\delta}(x_0)) }{\Vol(B_{R}(x_0))}
\end{align*}
we have
\begin{align*}
\liminf_{n \rightarrow \infty}\frac{ \Vol(E_n)}{\# N_n} >0.
\end{align*}
So it is enough to show that 
\begin{align*}
\liminf_{n \rightarrow \infty}\frac{ \# N_n}{\Vol(A_n)} >0.
\end{align*}

Now 
\begin{align*}
\#N_n  
&= \#\{ \gamma \in \Gamma : d_X(x_0, \gamma x_0) \leq r_n+49R \} - \#\{ \gamma \in \Gamma : d_X(x_0, \gamma x_0) \leq r_n-49R \} \\
& \geq C e^{(d-1)(r_n+49R)}-\frac{s_0}{\Vol B_\delta(x_0)} \Vol B_{r_n-49R}(x_0). 
\end{align*}
By the Bishop volume comparison theorem, see~\cite[Corollary 3.3]{Zhu1997},  there exists $V_0 > 0$ so that 
\begin{align*}
\Vol B_{r_n-49R}(x_0) \leq V_0 e^{(d-1)(r_n-49R)}
\end{align*}
for all $n > 0$. So
\begin{align*}
\#N_n \geq e^{(d-1)r_n} \Big( Ce^{49R} - \frac{s_0V_0}{\Vol B_\delta(x_0)}e^{-49R} \Big).
\end{align*}
Now $C, s_0,V_0, \delta$ do not depend on $R>0$ and $R$ is some very large number so we may assume that
\begin{align*}
\Big( Ce^{49R} - \frac{s_0V_0}{\Vol B_\delta(x_0)}e^{-49R} \Big) \geq 1.
\end{align*}
Then
\begin{align}\label{eq:bd_on_NN}
\#N_n \geq e^{(d-1)r_n}.
\end{align}
Finally, by the volume comparison theorem for annuli (see~\cite[Theorem 3.1]{Zhu1997}) we have
\begin{align*}
\liminf_{n \rightarrow \infty}\frac{ e^{(d-1)r_n}}{\Vol(A_n)} >0
\end{align*}
and so
\begin{align*}
\liminf_{n \rightarrow \infty}\frac{ \# N_n}{\Vol(A_n)} >0.
\end{align*}
This proves the second inequality. 

To prove the first inequality, notice that 
\begin{align*}
 \Vol(A_n) \geq \frac{\Vol B_\delta(x_0)}{s_0}\# N_n 
 \end{align*}
 and then use Equation~\ref{eq:bd_on_NN}.

\end{proof}

\begin{lemma}\label{lem:integration_by_parts} There exists a sequence $\epsilon_n > 0$ with $\lim_{n \rightarrow \infty} \epsilon_n = 0$ such that: if $\varphi: A_n \rightarrow [0,1]$ is a $C^\infty$ function compactly supported in $A_n$, then
\begin{align*}
\abs{\int_X d_0 \Delta \varphi dV - \int_{X \setminus \Cc_0} \varphi \Delta d_0 dV } \leq \epsilon_n \Vol(A_n).
\end{align*}
\end{lemma}

\begin{remark} When $d_0$ is smooth on $X \setminus \{x_0\}$ and $\varphi$ is compactly supported in $X \setminus \{x_0\}$, then 
 \begin{align*}
\int_X d_0 \Delta \varphi dV = \int_X \varphi \Delta d_0 dV 
\end{align*}
by integration by parts. So Lemma~\ref{lem:integration_by_parts} says that we can still do integration by parts in the case when $d_0$ is not smooth, but at the cost of some additive error which depends on the support of $\varphi$. 
\end{remark}

\begin{proof} Let $\tau(v)$, $J(r,v)$, $a_n(v)$, and $b_n(v)$ be as in the proof of Lemma~\ref{lem:Liu1}. Let $I =\int_X d_0 \Delta \varphi dV$. Since integration by parts holds for Lipschitz functions,
we have
\begin{align*}
I = - \int_X \nabla d_0 \cdot \nabla \varphi dV = -\int_{S_{x_0} X} \int_{a_n(v)}^{b_n(v)} \frac{\partial \wh{\varphi}}{\partial r}(r,v)J(r,v) dr d\mu(v)
\end{align*}
where $\wh{\varphi}(r,v) = \varphi(\exp_{x_0}(rv))$. Integrating by parts again and using Equation~\eqref{eq:H_J_relationship}
\begin{align*}
I & = \int_{S_{x_0} X} \int_{a_n(v)}^{b_n(v)} \wh{\varphi}(r,v) \frac{\partial J}{\partial r}(r,v) dr d\mu(v) -  \left. \int_{S_{x_0} X}  \wh{\varphi}(r,v) J(r,v)\right|_{a_n(v)}^{b_n(v)}  d\mu(v) \\
&  = \int_{X \setminus C_0} \varphi \Delta d_0 d V -  \left. \int_{S_{x_0} X}  \wh{\varphi}(r,v) J(r,v)\right|_{a_n(v)}^{b_n(v)}  d\mu(v).
\end{align*}

Next we estimate the absolute value of the second term in last equation. If $\tau(v) > r_n+50R$, then
\begin{align*}
\wh{\varphi}(a_n(v), v) = \wh{\varphi}(b_n(v), v)=0
\end{align*}
since $\varphi$ is compactly supported in $A_n$. Further, if $\tau(v) < r_n - 50R$, then $a_n(v) = b_n(v)$. Hence, if
\begin{align*}
\left. \wh{\varphi}(r,v) J(r,v)\right|_{a_n(v)}^{b_n(v)}  \neq 0,
\end{align*}
then we must have $\tau(v) \in [r_n-50R, r_n+50R]$. By the volume comparison theorem, there exists $J_0> 0$ such that 
\begin{align*}
J(r,v) \leq J_0e^{(d-1)r}.
\end{align*}
Then since $\abs{\wh{\varphi}} \leq 1$, we have
\begin{align*}
\abs{  \left. \int_{S_{x_0} X}  \wh{\varphi}(r,v) J(r,v)\right|_{a_n(v)}^{b_n(v)}  d\mu(v) } \leq J_0e^{(d-1)r} \mu\left(\left\{ v : \tau(v) \in  [r_n-50R, r_n+50R] \right\} \right).
\end{align*}
Since $\mu$ is a finite measure, we have
\begin{align*}
\lim_{n \rightarrow \infty}  \mu\left(\left\{ v : \tau(v) \in  [r_n-50R, r_n+50R] \right\} \right) = 0.
\end{align*}
Then by the first part of Lemma~\ref{lem:Liu2},
\begin{align*}
\epsilon_n := J_0 \frac{ e^{(d-1)r} }{\Vol(A_n)}  \mu\left(\left\{ v : \tau(v) \in  [r_n-50R, r_n+50R] \right\} \right)
\end{align*}
satisfies the conclusion of the lemma. 
\end{proof}

Let $\{ \chi_i : i \in \Nb\}$ be a partition of unity for $B_R(x_0)$. Then define $\phi_n = \sum_{i=1}^n \chi_i$. Then each $\phi_n$ is smooth, maps into $[0,1]$, and has compact support in $B_R(x_0)$. Moreover, $\phi_1 \leq \phi_2 \leq \dots$ and if $K \subset B_R(x_0)$ is a compact set, then $K \subset \phi_n^{-1}(1)$ for $n$ sufficiently large. 

Next let $\wt{\phi}_n = \sum_{\gamma \in M_n} \phi_n \circ \gamma^{-1}$. Then $\wt{\phi}_n$ is compactly supported in $E_n$ and
\begin{align}
\label{eq:limits_wt}
\lim_{n \rightarrow \infty} \frac{1}{\Vol(E_n)}\int_X 1_{E_n}-\wt{\phi}_n dV = 0.
\end{align}

\begin{lemma}\label{lem:construct_seqn} There exists a sequence $\gamma_n \in M_n$ such that 
\begin{align*}
\lim_{n \rightarrow \infty} \int_{X}d_0(\gamma_n x) \Delta \phi_n(x) dV = (d-1)\Vol(B_R(x_0))
\end{align*}
\end{lemma}

\begin{proof} Let 
\begin{align*}
c_n = \frac{1}{\Vol(B_R(x_0))}\max_{ \gamma \in M_n} \int_{X}d_0(\gamma x) \Delta \phi_n(x) dV. 
\end{align*}
By the Laplacian comparison theoem (see Theorem~\cite[Theorem 2.2]{Zhu1997}),
\begin{align*}
\limsup_{x \rightarrow \infty} \Delta d_0(x) \leq d-1
\end{align*}
in the sense of distributions, so
\begin{align*}
\limsup_{n \rightarrow \infty} c_n \leq d-1.
\end{align*}
 And we just have to prove that 
 \begin{align*}
 \liminf_{n \rightarrow \infty} c_n \geq d-1.
 \end{align*}
 Using Lemma~\ref{lem:Liu1} and the Laplacian comparison theorem we have
\begin{align*}
d-1 & = \lim_{n \rightarrow \infty} \frac{1}{\Vol(A_n)} \int_{A_n \setminus \Cc_0} \Delta d_0 dV \\
&= \lim_{n \rightarrow \infty} \frac{1}{\Vol(A_n)}\left( \int_{(A_n \setminus E_n) \setminus \Cc_0}  \Delta d_0 dV + \int_{E_n \setminus \Cc_0} \Delta d_0 dV\right) \\
& \leq \liminf_{n \rightarrow \infty} \frac{1}{\Vol(A_n)}\left( (d-1) \Vol(A_n \setminus E_n)+ \int_{E_n \setminus \Cc_0} \Delta d_0(x) dV\right) \\
& =  \liminf_{n \rightarrow \infty} \frac{1}{\Vol(A_n)}\left( (d-1) \Vol(A_n \setminus E_n)+ \int_{X \setminus \Cc_0} \wt{\phi}_n \Delta d_0 dV\right).
\end{align*}
In the last equality above we used Equation~\eqref{eq:limits_wt} and the fact that $\Delta d_0$ is uniformly bounded. 

Then by Lemma~\ref{lem:integration_by_parts} and the definition of $c_n$
\begin{align*}
d-1 & \leq \liminf_{n \rightarrow \infty} \frac{1}{\Vol(A_n)}\left( (d-1) \Vol(A_n \setminus E_n) + \int_{X \setminus \Cc_0}d_0 \Delta \wt{\phi}_n  dV\right) \\
& \leq (d-1)+\liminf_{n \rightarrow \infty} \frac{(c_n-d+1)\Vol(E_n)}{\Vol(A_n)}.
\end{align*}
So by Lemma~\ref{lem:Liu2}, we must have $\liminf_{n \rightarrow \infty} c_n \geq d-1$.
\end{proof}

Next consider the functions $f_n : B_R(x_0) \rightarrow \Rb$ given by 
\begin{align*}
f_n(x) = (d_0 \circ \gamma_n)(x) - (d_0 \circ \gamma_n)(x_0) = d(x_0,\gamma_n x) - d(x_0, \gamma_n x_0)
\end{align*}
Then each $f_n$ is 1-Lipschitz and $f_n(x_0)=0$, so we can pass to a subsequence such that $f_n$ converges locally uniformally to a function $f:B_R(x_0) \rightarrow \Rb$.

\begin{lemma} $f$ is $C^\infty$, $\Delta f \equiv d-1$, and $\norm{\nabla f}\equiv 1$. \end{lemma}

\begin{proof} Using elliptic regularity, to show the first two assertions it is enough to verify that $\Delta f \equiv d-1$ in the sense of distributions on $B_R(x_0)$. Let $\varphi$ be a positive $C^\infty$ function compactly supported in $B_R(x_0)$. We can assume that $\varphi \leq 1$. Then
\begin{align*}
 \int_X f(x) \Delta \varphi(x) dV= \lim_{n \rightarrow \infty} \int_{B_R(x_0)} d_0(\gamma_n x)\Delta \varphi(x) dV.
\end{align*}
So by the Laplacian comparison theorem (see Theorem~\cite[Theorem 2.2]{Zhu1997})
\begin{align*}
 \int_X f(x) \Delta \varphi(x) dV \leq (d-1) \int_X \varphi(x) dV. 
\end{align*}

By Lemma~\ref{lem:construct_seqn}
\begin{align*}
\lim_{n \rightarrow \infty} & \int_X d_0(\gamma_n x) \Delta (\phi_n-\varphi)(x) dV = (d-1)\Vol(B_R(x_0)) - \int_X f(x) \Delta \varphi(x) dV.
\end{align*}
Since $\varphi \leq 1$ and is compactly supported in $B_R(x_0)$, the function $\phi_n - \varphi$ is non-negative for large $n$ and so by the Laplacian comparison theorem 
\begin{align*}
\lim_{n \rightarrow \infty} & \int_X d_0(\gamma_n x) \Delta (\phi_n-\varphi)(x) dV \leq (d-1)\lim_{n \rightarrow \infty} \int_X (\phi_n-\varphi) dV \\
& = (d-1)\Vol(B_R(x_0)) - (d-1)\int_X \varphi(x) dV
\end{align*}
Thus
\begin{align*}
\int_X f(x) \Delta \varphi(x) dV \geq (d-1)\int_X \varphi(x) dV.
\end{align*}
Hence $\Delta f \equiv d-1$ on $B_R(x_0)$.

Finally, by construction $f$ is the restriction of some Busemann function to $B_R(x_0)$ and so $\norm{\nabla f}\equiv 1$ on $B_R(x_0)$ by Lemma 1 part (1) in~\cite{LW2010}.
\end{proof}

Now we fix a sequence $R_n \rightarrow \infty$ and repeat the above argument to obtain functions $h_n : B_{R_n}(x_0) \rightarrow \Rb$ which satisfy $\norm{\nabla h_n} \equiv 1$ and $\Delta h_n \equiv d-1$ on $B_{R_n}(x_0)$. Since each $h_n$ is 1-Lipschitz and $h_n(x_0)=0$, we can pass to a subsequence so that $h_n \rightarrow h$ where $h : X \rightarrow \Rb$ satisfies $\norm{\nabla h} \equiv 1$ and $\Delta h \equiv d-1$. Then $X$  is isometric to real hyperbolic space by Lemma~\ref{lem:LW}.

\section{Eigenvalues of certain subgroups}\label{sec:eigenvalues}

\begin{proposition}
Suppose $d \geq 3$, $\Lambda \leq \PSL_{d}(\Rb)$ is a discrete subgroup, and $G \leq \PSL_d(\Rb)$ is the Zariski closure of $\Lambda$. If 
\begin{enumerate}
\item $G = \PSL_d(\Rb)$, 
\item $d = 2n>2$ and $G$ is conjugate to $\PSp(2n,\Rb)$, 
\item $d = 2n+1 > 3$ and $G$ is conjugate to $\PSO(n,n+1)$, or
\item $d=7$ and $G$ is conjugate to the standard realization of $G_2$ in $\PSL_{7}(\Rb)$,
\end{enumerate}
then there exists some $\gamma \in \Lambda$ such that 
\begin{align*}
\frac{\lambda_1(\gamma)}{\lambda_2(\gamma)} \neq \frac{\lambda_2(\gamma)}{\lambda_3(\gamma)}.
\end{align*}
\end{proposition}

\begin{proof}
By conjugating, we can assume that either $G=\PSL_d(\Rb)$, $d = 2n>2$ and $G=\PSp(2n,\Rb)$, $d = 2n+1 > 3$ and $G=\PSO(n,n+1)$, or $d=7$ and $G$ coincides with the standard realization of $G_2$ in $\PSL_{7}(\Rb)$.

By the main theorem in~\cite{B1997} it is enough to find some element $g \in G$ such that 
\begin{align*}
\frac{\lambda_1(g)}{\lambda_2(g)} \neq \frac{\lambda_2(g)}{\lambda_3(g)}.
\end{align*}
This is clearly possible when $G = \PSL_d(\Rb)$ and $d \geq 3$.

Consider the case when $d = 2n>2$ and $G = \PSp(2n,\Rb)$. Then for any $\sigma_1, \dots, \sigma_n \in \Rb$, $G$ contains the matrix 
\begin{align*}
{\small \begin{bmatrix} 
e^{\sigma_1} & & & & & \\
& \ddots & & & & \\
& & e^{\sigma_n} & & & \\
& & & e^{-\sigma_1} & & \\
& & & & \ddots & \\
& & & & & e^{-\sigma_n}
\end{bmatrix}.}
\end{align*}
So picking $\sigma_1 > \sigma_2  > \dots  > \sigma_n > 0$ with $\sigma_1-\sigma_2 \neq \sigma_2-\sigma_3$ does the job. 

Consider the case when $d=2n+1 > 3$ and $G = \PSO(n,n+1)$.  Then for any $\sigma_1, \dots, \sigma_n \in \Rb$, $G$ contains a matrix $g$ which is conjugate to the block diagonal matrix 
\begin{align*}
 {\small \begin{bmatrix} 
\cosh(\sigma_1) & \sinh(\sigma_1) & & & & \\
\sinh(\sigma_1) &  \cosh(\sigma_1) & & & &  \\
& & \ddots & & & \\
& & & \cosh(\sigma_n) & \sinh(\sigma_n) &  \\
& & & \sinh(\sigma_n) &  \cosh(\sigma_n) &  \\
& & & & & 1
\end{bmatrix}.}
\end{align*}
Notice that this matrix has eigenvalues $e^{\sigma_1}, e^{-\sigma_1}, \dots, e^{\sigma_n}, e^{-\sigma_n}, 1$. So picking $\sigma_1 > \sigma_2  > \dots  > \sigma_n > 0$ with $\sigma_1-\sigma_2 \neq \sigma_2-\sigma_3$ does the job when $n \geq 3$ and picking $\sigma_1 > \sigma_2 > 0$ with $\sigma_1 - \sigma_2 \neq \sigma_2$ does the job when $n=2$. 

Finally consider the case when $d=7$ and $G$ coincides with the standard realization of $G_2$ in $\PSL_{7}(\Rb)$. The standard realization of $G_2$ in $\PSL_7(\Rb)$ can be described as follows. First let 
\begin{align*}
\Hb = \{ a_1 + a_2 i + a_3 j + a_4 k : a_1, \dots, a_4 \in \Rb\}
\end{align*}
be the quaternions. Then define the split Cayley algebra $\mathfrak{C}^{\prime} = \Hb \oplus \Hb e$ with multiplication 
\begin{align*}
(a+be)(c+de) = (ac+\overline{d}b)+(b\overline{c}+da)e.
\end{align*}
This is an 8-dimensional algebra over $\Rb$ with conjugation 
\begin{align*}
\overline{(a+be)} = \overline{a} - be.
\end{align*}
Next let $\Gb_2$ be the $\Rb$-linear transformations of $\mathfrak{C}^{\prime}$ which satisfy
\begin{align*}
\alpha(xy) = \alpha(x)\alpha(y).
\end{align*}
Then for $\alpha \in \Gb_2$ and $x \in \mathfrak{C}^{\prime}$ it is straightforward to verify that $\alpha(\overline{x}) = \overline{\alpha(x)}$ (see for instance~\cite[Proposition 2]{Y1977}). So $\Gb_2$ preserves the subspace 
\begin{align*}
\Spanset_{\Rb} \{ i, j, k, e, ie, je, ke\} 
\end{align*}
of purely imaginary elements. Since $\alpha(1) = 1$ for every $\alpha \in  \Gb_2$,  if we identify $ i, j, k, e, ie, je, ke$ with  $e_1, \dots, e_7$ the standard basis of $\Rb^7$ we obtain an embedding $\Gb_2 \hookrightarrow \PSL_7(\Rb)$. 

Now if $t,s \in \Rb$ a tedious calculation shows that 
\begin{align*}
{\small \begin{bmatrix} 
\cosh(t) & 0 & 0 & 0 & \sinh(t) & 0 & 0 \\
0 & \cosh(s) & 0 & 0 & 0 & \sinh(s) & 0 \\
0 & 0 & \cosh(s+t) & 0 & 0 & 0 & \sinh(s+t) \\
0 & 0 & 0 & 1 & 0 & 0 & 0 \\
\sinh(t) & 0 & 0 & 0 & \cosh(t) & 0 & 0 \\
0 & \sinh(s) & 0 & 0 & 0 & \cosh(s) & 0 \\
0 & 0 & \sinh(s+t) & 0 & 0 & 0 & \cosh(s+t) 
\end{bmatrix}}
\end{align*}
is contained in the image of this embedding. This matrix has eigenvalues 
\begin{align*}
e^t, e^{-t}, e^s, e^{-s}, e^{s+t}, e^{-(s+t)}, 1.
\end{align*}
 So picking $t > s > 0$ with $s \neq t -s$ does the job. 
\end{proof}

\section{Facts about linear transformations}\label{app:linear_maps}

In this section we describe some basic properties of the action of $\PGL_d(\Rb)$ on $\Pb(\Rb^d)$. These facts are used in Section~\ref{sec:reg_rigid} and are all simple consequences of Gelfand's formula. In this section we let $\norm{v}$ denote the Euclidean norm of a vector $v \in \Rb^d$. 

For a non-zero $d$-by-$d$ real matrix $A$ let 
\begin{align*}
\lambda_d(A) \leq \dots \leq \lambda_1(A)
\end{align*}
to be the absolute values of the eigenvalues of $A$ (counting multiplicity) and let 
\begin{align*}
\sigma_d(A) \leq \dots \leq \sigma_1(A)
\end{align*}
denote the singular values of $A$. 

\begin{theorem}[Gelfand's Formula]\label{thm:Gelfand} Suppose that $A$ is a non-zero $d$-by-$d$ real matrix. Then 
\begin{align*}
\lambda_1(A) = \lim_{n \rightarrow \infty} \sigma_1(A^n)^{1/n}.
\end{align*}
Moreover, there exists a proper subspace $V \subset \Rb^d$ such that 
\begin{align*}
\log \lambda_1(A) = \lim_{n \rightarrow \infty} \frac{1}{n} \log \norm{A^nv }
\end{align*}
for all $v \in \Rb^d \setminus V$. 
\end{theorem}

Since the ``moreover'' part is usually not included in statements of Gelfand's formula we sketch the proof. 

\begin{proof}[Proof of the ``Moreover'' part] Notice that the first part of Gelfand's formula implies that
\begin{align*}
\limsup_{n \rightarrow \infty} \frac{1}{n} \log \norm{A^nv } \leq \limsup_{n \rightarrow \infty} \frac{1}{n} \log \left(\sigma_1(A^n)\norm{v }\right) = \log \lambda_1(A)
\end{align*}
for nonzero $v \in \Rb^d$. So we just have to show that there exists a proper subspace $V \subset \Rb^d$ such that 
\begin{align*}
\liminf_{n \rightarrow \infty} \frac{1}{n} \log \norm{A^nv }\geq \log \lambda_1(A).
\end{align*}
for all $v \in \Rb^d \setminus V$. 

Using the Jordan decomposition we can write $A$ as a product of three commuting matrices $A = E S U$ where $E$ is elliptic, $S$ is real diagonalizable, and $U$ is unipotent. Let $\chi_1, \dots, \chi_k$ be the eigenvalues of $S$ (not counting multiplicity) and let $\Rb^d = \oplus_{i=1}^k V_i$ denote the corresponding eigenspace decomposition. Then let
\begin{align*}
V = \oplus \{ V_i : \abs{\chi_i} \neq \lambda_1(S)\}.
\end{align*}
Also, define a new norm $\norm{\cdot}_*$ on $\Rb^d$ by 
\begin{align*}
\norm{w}_* = \sqrt{\sum_{i=1}^k \norm{v_i}^2}
\end{align*}
where $w = \sum_{i=1}^k w_i$ and $w_i \in V_i$. 

Since $E$ is elliptic, there exists $C > 1$ such that:
\begin{align*}
\frac{1}{C} \norm{w} \leq \norm{E^nw} \leq C \norm{w}
\end{align*}
for all $n \in \Zb$ and $w \in \Rb^d$. Further, since $U^{-1}$ is unipotent, Gelfand's formula implies that
\begin{align*}
\lim_{n \rightarrow \infty} \frac{1}{n} \log \sigma_1(U^{-n}) =0
\end{align*}
Then if $v \in \Rb^d \setminus V$ we have 
\begin{align*}
\liminf_{n \rightarrow \infty} & \frac{1}{n} \log \norm{A^nv } = \liminf_{n \rightarrow \infty} \frac{1}{n} \log \norm{E^nS^nU^nv } \\
& = \liminf_{n \rightarrow \infty} \frac{1}{n} \log \norm{U^nS^nv } \geq \liminf_{n \rightarrow \infty} \frac{1}{n} \log \left( \frac{1}{\sigma_1(U^{-n}) }\norm{S^nv } \right)\\
&= \liminf_{n \rightarrow \infty} \frac{1}{n} \log \norm{S^n v}.
\end{align*}
Then, by the equivalence of finite dimensional norms, 
\begin{align*}
\liminf_{n \rightarrow \infty} \frac{1}{n} \log \norm{S^n v} = \liminf_{n \rightarrow \infty} \frac{1}{n} \log \norm{S^n v}_* = \liminf_{n \rightarrow \infty} \frac{1}{n} \log \left( \lambda_1(A)^n \norm{v} \right)= \log \lambda_1(A).
\end{align*}

\end{proof}

For the rest of the section, let $d_{\Pb}$ be a distance on $\Pb(\Rb^{d})$ induced by a Riemannian metric. We will use the following estimate.

\begin{observation}\label{obs:biLip_affine_chart} Suppose $\mathbb{A} \subset \Pb(\Rb^d)$ is an affine chart and $\iota : \Rb^{d-1} \rightarrow \mathbb{A}$ is an affine automorphism. Then for any compact set $K \subset \Rb^{d-1}$ there exists $C>1$ such that 
\begin{align*}
\frac{1}{C} \norm{v-w} \leq d_{\Pb}(\iota(v), \iota(w)) \leq C \norm{v-w}
\end{align*}
for all $v,w \in K$. 
\end{observation}

\begin{proof} This follows from a compactness argument. \end{proof}

\begin{observation}\label{obs:dynamics_app} Suppose $g \in \PGL_{d}(\Rb)$ is proximal and $\ell_g^+ \in \Pb(\Rb^d)$ is the eigenline of $g$ corresponding to the eigenvalue of largest absolute value.  If $v \neq \ell^+_g$ and $g^n v \rightarrow \ell^+_g$, then
\begin{align*}
\log\frac{\lambda_2(g)}{\lambda_1(g)} \geq \limsup_{n \rightarrow \infty} \frac{1}{n} \log d_{\Pb}\Big(g^n v, \ell^+_g \Big).
\end{align*}
Moreover, there exists a proper subspace $V \subset \Pb(\Rb^d)$ such that: if $v \in \Pb(\Rb^{d}) \setminus V$ and
$g^n v \rightarrow \ell^+_g$, then
\begin{align*}
\log\frac{\lambda_2(g)}{\lambda_1(g)} = \lim_{n \rightarrow \infty} \frac{1}{n} \log d_{\Pb}\Big(g^n v, \ell^+_g \Big).
\end{align*}
\end{observation}

\begin{proof} By changing coordinates we can assume that 
\begin{align*}
g = \begin{bmatrix}
\lambda & 0 \\ 
0 & A 
\end{bmatrix},
\end{align*} 
$\ell_g^+ = [1:0:\dots:0]$, $\abs{\lambda} = \lambda_1(g)$, and $\lambda_1(A) = \lambda_2(g)$. 

Through out the proof we will use the notation $[v_1 : v_2] \in \Pb(\Rb^d)$ where $v_1 \in \Rb$ and $v_2\in \Rb^{d-1}$. With this notation 
\begin{align}
\label{eq:apply_g}
g^n \cdot [v_1 : v_2] = [ \lambda^n v_1 : A^nv_2 ] = \left[ v_1 : \frac{A^n}{\lambda^n} v_2 \right].
\end{align}
By Gelfand's formula $\frac{A^n}{\lambda^n} \rightarrow 0$ and so $g^n \cdot v \rightarrow \ell_g^+$ if and only if $v_1 \neq 0$. 

Next we fix a small neighborhood $U$ of $\ell_g^+$ such that 
\begin{align*}
\overline{U} \subset \{ [v_1 : v_2] : v_1 \neq 0\}.
\end{align*}
By Observation~\ref{obs:biLip_affine_chart} there exists $C > 1$ such that if $v=[v_1 : v_2]$ and $w = [w_1 : w_2] $ are in $U$, then 
\begin{align}
\label{eq:bi_lip}
\frac{1}{C} \norm{ v_2/v_1 - w_2/w_1} \leq d_{\Pb}(v,w) \leq C \norm{ v_2/v_1 - w_2/w_1}.
\end{align}
So if $v=[v_1 : v_2] \in \Pb(\Rb^d)$ and $g^n v \rightarrow \ell_g^+$, then by Equations~\eqref{eq:apply_g} and~\eqref{eq:bi_lip} we have
\begin{align*}
\limsup_{n \rightarrow \infty} & \frac{1}{n} \log d_{\Pb}\Big(g^n v, \ell^+_g \Big) =  \limsup_{n \rightarrow \infty} \frac{1}{n} \log\left( \frac{1}{\abs{\lambda}^n} \norm{A^n v_2} \right)\\
& \leq \limsup_{n \rightarrow \infty} \frac{1}{n} \log\left( \frac{1}{\abs{\lambda}^n} \sigma_1(A^n) \right)= \log\frac{\lambda_2(g)}{\lambda_1(g)}.
\end{align*}

Using the ``moreover'' part of Gelfand's formula, there exists a proper subspace $V_0 \subset \Rb^{d-1}$ such that 
\begin{align*}
\log \lambda_1(A) = \lim_{n \rightarrow \infty} \frac{1}{n} \log \norm{A^nv }
\end{align*}
for all $v \in \Rb^{d-1} \setminus V_0$. Then let 
\begin{align*}
V = \{ [v_1 : v_2 ] \in \Pb(\Rb^d) : v_2 \in V_0\}.
\end{align*}
Then if $v=[v_1:v_2] \in  \Pb(\Rb^d)  \setminus V$ and  $g^n v \rightarrow \ell_g^+$, Equations~\eqref{eq:apply_g} and~\eqref{eq:bi_lip} imply that
 \begin{align*}
\lim_{n \rightarrow \infty} & \frac{1}{n} \log d_{\Pb}\Big(g^n v, \ell^+_g \Big) =  \lim_{n \rightarrow \infty} \frac{1}{n} \log \left(\frac{1}{\abs{\lambda}^n} \norm{A^n v_2} \right)= \log\frac{\lambda_2(g)}{\lambda_1(g)}. \qedhere
\end{align*}
\end{proof}

\begin{observation}\label{obs:span_gen_eigenvectors} Suppose that $A \in \GL_d(\Rb)$ and there exists $n_k \rightarrow \infty$ such that 
\begin{align*}
T = \lim_{k \rightarrow \infty} \frac{1}{\norm{A^{n_k}}} A^{n_k}
\end{align*}
in $\End(\Rb^d)$. If $v \in { \rm Im}(T)$, then there exists generalized eigenvectors $v_1, \dots, v_m \in \Cb^d$ of $A$ such that 
\begin{align*}
v= v_1 + \dots + v_m
\end{align*}
and the eigenvalues corresponding to $v_1, \dots, v_m$ all have absolute value $\lambda_1(A)$. 
\end{observation}

\begin{proof} By changing coordinates we can assume that 
\begin{align*}
A = \begin{pmatrix} A_1 & 0 \\ 0 & A_2 \end{pmatrix}
\end{align*}
where $A_1 \in \GL_{k}(\Rb)$, $A_2 \in \GL_{d-k}(\Rb)$, every eigenvalue of $A_1$ has absolute value $\lambda_1(A)$, and every eigenvalue of $A_2$ has absolute value strictly less than $\lambda_1(A)$. Then every $v \in \Spanset\{ e_1,\dots, e_k\}$ can be written as a linear combination of generalized eigenvectors in $\Cb^d$ whose corresponding eigenvalues have absolute value $\lambda_1(A)$. Further by Gelfand's formula
\begin{align*}
0=\lim_{k \rightarrow \infty} \frac{1}{\norm{A^{n_k}}} A_2^{n_k}.
\end{align*}
and so
\begin{align*}
T = \begin{pmatrix} T_1 & 0 \\ 0 & 0 \end{pmatrix}
\end{align*}
for some $k$-by-$k$ matrix $T_1$. 
\end{proof}

\begin{observation}\label{obs:no_proximal_same_speed} Suppose that $g \in \GL_d(\Rb)$, $\lambda_1(g) = \lambda_2(g)$, and $v_0 \in \Rb^d$ is an eigenvector of $g$ whose eigenvalue has absolute value $\lambda_1(g)$. Then there exists a proper subspace $V \subset \Pb(\Rb^d)$ such that: 
 \begin{align*}
0=\lim_{n \rightarrow \infty} & \frac{1}{n} \log d_{\Pb}\Big(g^n v, [v_0] \Big) 
\end{align*}
for every $v \in \Pb(\Rb^d) \setminus V$. 
\end{observation}

\begin{proof}Suppose that $gv_0 = \lambda v_0$. Let $e_1,\dots, e_m$ be the standard basis of $\Rb^d$. By making a change of coordinates we can assume that $v_0 = e_1$ and
\begin{align*}
 g = \begin{pmatrix} J & 0 \\ 0 & A \end{pmatrix}
\end{align*}
where $J$ is a $m$-by-$m$ upper triangular matrix with $\lambda, \dots, \lambda$ down the diagonal. 

By Observation~\ref{obs:biLip_affine_chart}, we can fix a small neighborhood $U$ of $[e_1]$ and $C > 1$ such that: if $w=[w_1:\dots:w_d] \in U$, then
\begin{align}
\label{eq:dist_est_C5}
\frac{1}{C} \norm{ (w_2/w_1,\dots, w_d/w_1) } \leq d_{\Pb}([e_1],w) \leq C \norm{ (w_2/w_1,\dots, w_d/w_1) }.
\end{align}
Then fix $\delta > 0$ such that: if $w \notin U$, then $d_{\Pb}(w,[e_1]) \geq \delta$. 

We consider two cases: \\

\noindent \textbf{Case 1:} $m> 1$. Since $J$ is upper triangular with $\lambda, \dots, \lambda$ on the diagonal, 
\begin{align*}
ge_i \in \lambda v_i + \Spanset\{ e_1,\dots, e_{i-1}\} \text{ for } i = 1,\dots, m.
\end{align*}
Let
\begin{align*}
V = [\Spanset\{e_1,\dots, e_{m-1}, e_{m+1}, \dots, e_d \}].
\end{align*}

Suppose that $v=\prescript{t}{}{(v_1,\dots,v_d)} \in \Rb^d$ and $[v] \notin V$. Then $v_m \neq 0$. Let 
\begin{align*}
\prescript{t}{}{\left(v_1^{(n)}, \dots, v_d^{(n)}\right)} := g^n v.
\end{align*}
Then 
\begin{align*}
 \abs{v_1^{(n)}} \leq \norm{g^n v} \leq \sigma_1(g^n)\norm{v}
\end{align*}
and $v_m^{(n)} = \lambda^n v_m$. Since $d_{\Pb}$ has finite diameter we see that 
 \begin{align*}
0 \geq \limsup_{n \rightarrow \infty} & \frac{1}{n} \log d_{\Pb}\Big(g^n [v], [e_1] \Big).
\end{align*}
If $g^n [v] \notin U$, then 
 \begin{align*}
 \frac{1}{n} & \log d_{\Pb}\Big(g^n [v], [e_1] \Big) \geq \frac{1}{n} \log \delta.
 \end{align*}
And if $g^n [v] \in U$, then by Equation~\eqref{eq:dist_est_C5}
 \begin{align*}
 \frac{1}{n} & \log d_{\Pb}\Big(g^n v, [e_1] \Big) \geq \frac{-1}{n}\log(C) +  \frac{1}{n} \log \abs{ \frac{ v_m^{(n)}}{v_1^{(n)}}} \\
 & \geq \frac{-1}{n}\log(C) +  \frac{1}{n} \log\abs{ \lambda} -\frac{1}{n} \log\sigma_1(g^n)-\frac{1}{n} \log \norm{v}.
\end{align*}
Hence Gelfand's formula implies that 
  \begin{align*}
0 \leq \liminf_{n \rightarrow \infty} & \frac{1}{n} \log d_{\Pb}\Big(g^n [v], [e_1] \Big).
\end{align*}
So 
  \begin{align*}
0= \lim_{n \rightarrow \infty} & \frac{1}{n} \log d_{\Pb}\Big(g^n [v], [e_1] \Big).
\end{align*}

\noindent \textbf{Case 2:} $m = 1$. Then
\begin{align*}
g = \begin{pmatrix} \lambda & 0 \\ 0 & A \end{pmatrix}
\end{align*}
where $A \in \GL_{d-1}(\Rb)$. Since $\lambda_1(g) = \lambda_2(g)$, we see that $\lambda_1(A) = \lambda_1(g)$. By the ``moreover'' part of Gelfand's formula there exists some proper subspace $V_0 \subset \Rb^{d-1}$ such that
\begin{align}
\label{eq:moreover_Gelfand_pr_lemC5}
\log \lambda_1(A) = \lim_{n \rightarrow \infty} \frac{1}{n} \log \norm{A^nv }
\end{align}
for all $v \in \Rb^{d-1} \setminus V_0$. 

We will use the notation $[v_1 : v_2] \in \Pb(\Rb^d)$ where $v_1 \in \Rb$ and $v_2\in \Rb^{d-1}$. With this notation 
\begin{align*}
g^n \cdot [v_1 : v_2] = [ \lambda^n v_1 : A^nv_2 ] = \left[ v_1 : \frac{A^n}{\lambda^n} v_2 \right].
\end{align*}
Then define
\begin{align*}
V = \{ [v_1 : v_2 ] \in \Pb(\Rb^d) : v_2 \in V_0\}.
\end{align*}

Fix some $v \in \Pb(\Rb^d) \setminus V$. Since $d_{\Pb}$ has finite diameter we see that 
 \begin{align*}
0 \geq \limsup_{n \rightarrow \infty} & \frac{1}{n} \log d_{\Pb}\Big(g^n [v], [e_1] \Big).
\end{align*}
If $g^n [v] \notin U$, then 
 \begin{align*}
 \frac{1}{n} & \log d_{\Pb}\Big(g^n [v], [e_1] \Big) \geq \frac{1}{n} \log \delta.
 \end{align*}
And if $g^n [v] \in U$, then by Equation~\eqref{eq:dist_est_C5}
\begin{align*}
\frac{1}{n} \log d_{\Pb}\Big(g^n [v], [e_1] \Big) & \geq - \frac{1}{n} \log C +  \frac{1}{n} \log \norm{ \frac{1}{\lambda^n} A^nv_2}\\
&  = - \frac{1}{n} \log C +- \log \abs{\lambda} + \frac{1}{n} \log \norm{ A^nv_2}.
\end{align*}
So by Equation~\eqref{eq:moreover_Gelfand_pr_lemC5}
\begin{align*}
\liminf_{n \rightarrow \infty} & \frac{1}{n} \log d_{\Pb}\Big(g^n [v], [e_1] \Big) \geq 0.
\end{align*} 
So 
  \begin{align*}
0= \lim_{n \rightarrow \infty} & \frac{1}{n} \log d_{\Pb}\Big(g^n [v], [e_1] \Big).
\end{align*}

\end{proof}

\bibliographystyle{plain}
\bibliography{geom}

\begin{thebibliography}{10}

\bibitem{APS2004}
Mats Andersson, Mikael Passare, and Ragnar Sigurdsson.
\newblock {\em Complex convexity and analytic functionals}, volume 225 of {\em
  Progress in Mathematics}.
\newblock Birkh\"{a}user Verlag, Basel, 2004.

\bibitem{BDL2018}
Samuel Ballas, Jeffrey Danciger, and Gye-Seon Lee.
\newblock Convex projective structures on nonhyperbolic three-manifolds.
\newblock {\em Geom. Topol.}, 22(3):1593--1646, 2018.

\bibitem{BM2012}
Thierry Barbot and Quentin M\'erigot.
\newblock Anosov {A}d{S} representations are quasi-{F}uchsian.
\newblock {\em Groups Geom. Dyn.}, 6(3):441--483, 2012.

\bibitem{BMZ2015}
Thomas Barthelm\'e, Ludovic Marquis, and Andrew Zimmer.
\newblock Entropy rigidity of {H}ilbert and {R}iemannian metrics.
\newblock {\em Int. Math. Res. Not. IMRN}, (22):6841--6866, 2017.

\bibitem{B1997}
Y.~Benoist.
\newblock Propri\'et\'es asymptotiques des groupes lin\'eaires.
\newblock {\em Geom. Funct. Anal.}, 7(1):1--47, 1997.

\bibitem{B2000}
Yves Benoist.
\newblock Automorphismes des c\^ones convexes.
\newblock {\em Invent. Math.}, 141(1):149--193, 2000.

\bibitem{B2003a}
Yves Benoist.
\newblock Convexes hyperboliques et fonctions quasisym\'etriques.
\newblock {\em Publ. Math. Inst. Hautes \'Etudes Sci.}, (97):181--237, 2003.

\bibitem{B2004}
Yves Benoist.
\newblock Convexes divisibles. {I}.
\newblock In {\em Algebraic groups and arithmetic}, pages 339--374. Tata Inst.
  Fund. Res., Mumbai, 2004.

\bibitem{B2006}
Yves Benoist.
\newblock Convexes divisibles. {IV}. {S}tructure du bord en dimension 3.
\newblock {\em Invent. Math.}, 164(2):249--278, 2006.

\bibitem{B2008}
Yves Benoist.
\newblock A survey on divisible convex sets.
\newblock In {\em Geometry, analysis and topology of discrete groups}, volume~6
  of {\em Adv. Lect. Math. (ALM)}, pages 1--18. Int. Press, Somerville, MA,
  2008.

\bibitem{BH2013}
Yves Benoist and Dominique Hulin.
\newblock Cubic differentials and finite volume convex projective surfaces.
\newblock {\em Geom. Topol.}, 17(1):595--620, 2013.

\bibitem{BPS2016}
Jairo Bochi, Rafael Potrie, and Andr\'{e}s Sambarino.
\newblock Anosov representations and dominated splittings.
\newblock {\em J. Eur. Math. Soc. (JEMS)}, 21(11):3343--3414, 2019.

\bibitem{B1995}
Marc Bourdon.
\newblock Structure conforme au bord et flot g\'eod\'esique d'un {${\rm
  CAT}(-1)$}-espace.
\newblock {\em Enseign. Math. (2)}, 41(1-2):63--102, 1995.

\bibitem{BCLS2015}
Martin Bridgeman, Richard Canary, Fran{\c{c}}ois Labourie, and Andres
  Sambarino.
\newblock The pressure metric for {A}nosov representations.
\newblock {\em Geom. Funct. Anal.}, 25(4):1089--1179, 2015.

\bibitem{BH1999}
Martin~R. Bridson and Andr{\'e} Haefliger.
\newblock {\em Metric spaces of non-positive curvature}, volume 319 of {\em
  Grundlehren der Mathematischen Wissenschaften [Fundamental Principles of
  Mathematical Sciences]}.
\newblock Springer-Verlag, Berlin, 1999.

\bibitem{HB1955}
Herbert Busemann.
\newblock {\em The geometry of geodesics}.
\newblock Academic Press Inc., New York, N. Y., 1955.

\bibitem{BK1953}
Herbert Busemann and Paul~J. Kelly.
\newblock {\em Projective geometry and projective metrics}.
\newblock Academic Press Inc., New York, N. Y., 1953.

\bibitem{C1972}
Eugenio Calabi.
\newblock Complete affine hyperspheres. {I}.
\newblock In {\em Symposia {M}athematica, {V}ol. {X} ({C}onvegno di {G}eometria
  {D}ifferenziale, {INDAM}, {R}ome, 1971)}, pages 19--38. Academic Press,
  London, 1972.

\bibitem{CLN2006}
Bennett Chow, Peng Lu, and Lei Ni.
\newblock {\em Hamilton's {R}icci flow}, volume~77 of {\em Graduate Studies in
  Mathematics}.
\newblock American Mathematical Society, Providence, RI; Science Press Beijing,
  New York, 2006.

\bibitem{CLT2015}
D.~Cooper, D.~D. Long, and S.~Tillmann.
\newblock On convex projective manifolds and cusps.
\newblock {\em Adv. Math.}, 277:181--251, 2015.

\bibitem{CK2002}
M.~Coornaert and G.~Knieper.
\newblock Growth of conjugacy classes in {G}romov hyperbolic groups.
\newblock {\em Geom. Funct. Anal.}, 12(3):464--478, 2002.

\bibitem{C1993}
Michel Coornaert.
\newblock Mesures de {P}atterson-{S}ullivan sur le bord d'un espace
  hyperbolique au sens de {G}romov.
\newblock {\em Pacific J. Math.}, 159(2):241--270, 1993.

\bibitem{Cra2009}
Micka{\"e}l Crampon.
\newblock Entropies of strictly convex projective manifolds.
\newblock {\em J. Mod. Dyn.}, 3(4):511--547, 2009.

\bibitem{CM2014}
Micka\"el Crampon and Ludovic Marquis.
\newblock Finitude g\'eom\'etrique en g\'eom\'etrie de {H}ilbert.
\newblock {\em Ann. Inst. Fourier (Grenoble)}, 64(6):2299--2377, 2014.

\bibitem{DGK2017b}
J.~{Danciger}, F.~{Gu{\'e}ritaud}, and F.~{Kassel}.
\newblock {Convex cocompact actions in real projective geometry}.
\newblock {\em ArXiv e-prints}, April 2017.

\bibitem{DGK2017a}
Jeffrey Danciger, Fran\c{c}ois Gu\'eritaud, and Fanny Kassel.
\newblock Convex cocompactness in pseudo-{R}iemannian hyperbolic spaces.
\newblock {\em Geom. Dedicata}, 192:87--126, 2018.

\bibitem{dlH2000}
Pierre de~la Harpe.
\newblock {\em Topics in geometric group theory}.
\newblock Chicago Lectures in Mathematics. University of Chicago Press,
  Chicago, IL, 2000.

\bibitem{G1992}
David Gabai.
\newblock Convergence groups are {F}uchsian groups.
\newblock {\em Ann. of Math. (2)}, 136(3):447--510, 1992.

\bibitem{GGKW2015}
Fran\c{c}ois Gu\'eritaud, Olivier Guichard, Fanny Kassel, and Anna Wienhard.
\newblock Anosov representations and proper actions.
\newblock {\em Geom. Topol.}, 21(1):485--584, 2017.

\bibitem{GW2012}
Olivier Guichard and Anna Wienhard.
\newblock Anosov representations: domains of discontinuity and applications.
\newblock {\em Invent. Math.}, 190(2):357--438, 2012.

\bibitem{KB2002}
Ilya Kapovich and Nadia Benakli.
\newblock Boundaries of hyperbolic groups.
\newblock In {\em Combinatorial and geometric group theory ({N}ew {Y}ork,
  2000/{H}oboken, {NJ}, 2001)}, volume 296 of {\em Contemp. Math.}, pages
  39--93. Amer. Math. Soc., Providence, RI, 2002.

\bibitem{KLP2014}
M.~{Kapovich}, B.~{Leeb}, and J.~{Porti}.
\newblock {Morse actions of discrete groups on symmetric space}.
\newblock {\em ArXiv e-prints}, March 2014.

\bibitem{KL2018}
Michael Kapovich and Bernhard Leeb.
\newblock Finsler bordifications of symmetric and certain locally symmetric
  spaces.
\newblock {\em Geom. Topol.}, 22(5):2533--2646, 2018.

\bibitem{KLP_TG_2016}
Michael Kapovich, Bernhard Leeb, and Joan Porti.
\newblock Some recent results on {A}nosov representations.
\newblock {\em Transform. Groups}, 21(4):1105--1121, 2016.

\bibitem{KLP_EJM_2017}
Michael Kapovich, Bernhard Leeb, and Joan Porti.
\newblock Anosov subgroups: dynamical and geometric characterizations.
\newblock {\em Eur. J. Math.}, 3(4):808--898, 2017.

\bibitem{KLP2013}
Michael Kapovich, Bernhard Leeb, and Joan Porti.
\newblock Dynamics on flag manifolds: domains of proper discontinuity and
  cocompactness.
\newblock {\em Geom. Topol.}, 22(1):157--234, 2018.

\bibitem{KLP2014b}
Michael Kapovich, Bernhard Leeb, and Joan Porti.
\newblock A {M}orse lemma for quasigeodesics in symmetric spaces and
  {E}uclidean buildings.
\newblock {\em Geom. Topol.}, 22(7):3827--3923, 2018.

\bibitem{KN2002}
Anders Karlsson and Guennadi~A. Noskov.
\newblock The {H}ilbert metric and {G}romov hyperbolicity.
\newblock {\em Enseign. Math. (2)}, 48(1-2):73--89, 2002.

\bibitem{KL2006}
Bruce Kleiner and Bernhard Leeb.
\newblock Rigidity of invariant convex sets in symmetric spaces.
\newblock {\em Invent. Math.}, 163(3):657--676, 2006.

\bibitem{L2006}
Fran{\c{c}}ois Labourie.
\newblock Anosov flows, surface groups and curves in projective space.
\newblock {\em Invent. Math.}, 165(1):51--114, 2006.

\bibitem{LW2010}
Fran{\c{c}}ois Ledrappier and Xiaodong Wang.
\newblock An integral formula for the volume entropy with applications to
  rigidity.
\newblock {\em J. Differential Geom.}, 85(3):461--477, 2010.

\bibitem{L2011}
Gang Liu.
\newblock A short proof to the rigidity of volume entropy.
\newblock {\em Math. Res. Lett.}, 18(1):151--153, 2011.

\bibitem{Lof2001}
John~C. Loftin.
\newblock Affine spheres and convex {$\mathbb{RP}^n$}-manifolds.
\newblock {\em Amer. J. Math.}, 123(2):255--274, 2001.

\bibitem{M2014}
Ludovic Marquis.
\newblock Around groups in {H}ilbert geometry.
\newblock In {\em Handbook of {H}ilbert geometry}, volume~22 of {\em IRMA Lect.
  Math. Theor. Phys.}, pages 207--261. Eur. Math. Soc., Z\"urich, 2014.

\bibitem{M2007}
Geoffrey Mess.
\newblock Lorentz spacetimes of constant curvature.
\newblock {\em Geom. Dedicata}, 126:3--45, 2007.

\bibitem{N1992}
M.~H.~A. Newman.
\newblock {\em Elements of the topology of plane sets of points}.
\newblock Dover Publications, Inc., New York, second edition, 1992.

\bibitem{OV1990}
A.~L. Onishchik and \`E.~B. Vinberg.
\newblock {\em Lie groups and algebraic groups}.
\newblock Springer Series in Soviet Mathematics. Springer-Verlag, Berlin, 1990.
\newblock Translated from the Russian and with a preface by D. A. Leites.

\bibitem{P2005}
Panos Papasoglu.
\newblock Quasi-isometry invariance of group splittings.
\newblock {\em Ann. of Math. (2)}, 161(2):759--830, 2005.

\bibitem{PS2014}
Rafael Potrie and Andr\'es Sambarino.
\newblock Eigenvalues and entropy of a {H}itchin representation.
\newblock {\em Invent. Math.}, 209(3):885--925, 2017.

\bibitem{P1994}
Gopal Prasad.
\newblock {${\bf R}$}-regular elements in {Z}ariski-dense subgroups.
\newblock {\em Quart. J. Math. Oxford Ser. (2)}, 45(180):541--545, 1994.

\bibitem{Q2005}
J.-F. Quint.
\newblock Groupes convexes cocompacts en rang sup\'erieur.
\newblock {\em Geom. Dedicata}, 113:1--19, 2005.

\bibitem{Q2010}
Jean-Fran{\c{c}}ois Quint.
\newblock Convexes divisibles (d'apr\`es {Y}ves {B}enoist).
\newblock {\em Ast\'erisque}, (332):Exp. No. 999, vii, 45--73, 2010.
\newblock S{\'e}minaire Bourbaki. Volume 2008/2009. Expos{\'e}s 997--1011.

\bibitem{S2016}
A.~Sambarino.
\newblock On entropy, regularity and rigidity for convex representations of
  hyperbolic manifolds.
\newblock {\em Math. Ann.}, 364(1-2):453--483, 2016.

\bibitem{S1971}
John Stallings.
\newblock {\em Group theory and three-dimensional manifolds}.
\newblock Yale University Press, New Haven, Conn.-London, 1971.
\newblock A James K. Whittemore Lecture in Mathematics given at Yale
  University, 1969, Yale Mathematical Monographs, 4.

\bibitem{S1968}
John~R. Stallings.
\newblock On torsion-free groups with infinitely many ends.
\newblock {\em Ann. of Math. (2)}, 88:312--334, 1968.

\bibitem{S1996}
G.~A. Swarup.
\newblock On the cut point conjecture.
\newblock {\em Electron. Res. Announc. Amer. Math. Soc.}, 2(2):98--100
  (electronic), 1996.

\bibitem{Tho2015}
Nicolas Tholozan.
\newblock Volume entropy of {H}ilbert metrics and length spectrum of {H}itchin
  representations into {${\rm PSL}(3,\Bbb{R})$}.
\newblock {\em Duke Math. J.}, 166(7):1377--1403, 2017.

\bibitem{T1984}
Pekka Tukia.
\newblock The {H}ausdorff dimension of the limit set of a geometrically finite
  {K}leinian group.
\newblock {\em Acta Math.}, 152(1-2):127--140, 1984.

\bibitem{T1988}
Pekka Tukia.
\newblock Homeomorphic conjugates of {F}uchsian groups.
\newblock {\em J. Reine Angew. Math.}, 391:1--54, 1988.

\bibitem{V1970}
Jacques Vey.
\newblock Sur les automorphismes affines des ouverts convexes saillants.
\newblock {\em Ann. Scuola Norm. Sup. Pisa (3)}, 24:641--665, 1970.

\bibitem{VLS1991}
Luc Vrancken, An~Min Li, and Udo Simon.
\newblock Affine spheres with constant affine sectional curvature.
\newblock {\em Math. Z.}, 206(4):651--658, 1991.

\bibitem{Y1977}
Ichiro Yokota.
\newblock Non-compact simple {L}ie group {$G_{2}'$} of type {$G_{2}$}.
\newblock {\em J. Fac. Sci. Shinshu Univ.}, 12(1):45--52, 1977.

\bibitem{Zhu1997}
Shunhui Zhu.
\newblock The comparison geometry of {R}icci curvature.
\newblock In {\em Comparison geometry ({B}erkeley, {CA}, 1993--94)}, volume~30
  of {\em Math. Sci. Res. Inst. Publ.}, pages 221--262. Cambridge Univ. Press,
  Cambridge, 1997.

\end{thebibliography}

\end{document}